\documentclass[12pt,english]{amsart}
\usepackage{ae,aecompl}
\usepackage[T1]{fontenc}
\usepackage{geometry}
\usepackage{babel}
\usepackage{array}
\usepackage{multirow}
\usepackage{amstext}
\usepackage{amsthm}
\usepackage{graphicx}
\usepackage[unicode=true,
 bookmarks=true,bookmarksnumbered=false,bookmarksopen=false,
 breaklinks=false,pdfborder={0 0 1},backref=false,colorlinks=false]
 {hyperref}

\makeatletter

\providecommand{\tabularnewline}{\\}

\numberwithin{equation}{section}
\numberwithin{figure}{section}
\theoremstyle{remark}
\newtheorem*{rem*}{\protect\remarkname}
\theoremstyle{plain}
\newtheorem{thm}{\protect\theoremname}
\theoremstyle{remark}
\newtheorem{rem}[thm]{\protect\remarkname}

\@ifundefined{showcaptionsetup}{}{%
 \PassOptionsToPackage{caption=false}{subfig}}
\usepackage{subfig}
\makeatother

\providecommand{\remarkname}{Remark}
\providecommand{\theoremname}{Theorem}

\begin{document}
\title[Approximation of a Gompertz tumor growth model]{An explicit Fourier-Klibanov method for an age-dependent tumor growth
model of Gompertz type}
\author{Nguyen Thi Yen Ngoc}
\address{Department of Mechanics, Faculty of Mathematics and Computer Science,
University of Science, Vietnam National University - Ho Chi Minh City,
Vietnam}
\email{ntyngoc@hcmus.edu.vn, yenngoc0202@gmail.com}
\author{Vo Anh Khoa}
\address{Department of Mathematics, Florida A\&M University, Tallahassee, FL
32307, USA}
\email{anhkhoa.vo@famu.edu, vakhoa.hcmus@gmail.com}
\keywords{Tumor growth, Gompertz law, age-structured models, Fourier-Klibanov
series, finite difference method, stability estimate.}
\begin{abstract}
This paper proposes an explicit Fourier-Klibanov method as a new approximation
technique for an age-dependent population PDE of Gompertz type in
modeling the evolution of tumor density in a brain tissue. Through
suitable nonlinear and linear transformations, the Gompertz model
of interest is transformed into an auxiliary third-order nonlinear
PDE. Then, a coupled transport-like PDE system is obtained via an
application of the Fourier-Klibanov method, and, thereby, is approximated
by the explicit finite difference operators of characteristics. The
stability of the resulting difference scheme is analyzed under the
standard 2-norm topology. Finally, we present some computational results
to demonstrate the effectiveness of the proposed method.
\end{abstract}

\maketitle

\section{Introduction and problem statement}

Mathematical modeling is a widely used approach to gain insight into
the growth and invasion of cancer cell populations. By leveraging
mathematical and computational methods in scientific oncology, researchers
can explain various cancer concepts and develop more effective treatment
strategies. To investigate the evolutionary dynamics of cancer, several
deterministic mathematical models have thus been developed for specific
cancer types and stages; see e.g. the monograph \cite{Wodarz2014}.
In principle, models of population dynamics are often formulated using
a continuum approach based on partial differential equations (PDEs).
However, finding analytical solutions to these PDEs is challenging
due to the nonlinear structures that account for the complex mechanisms
of cancer. Consequently, researchers resort to various approximation
methods to obtain numerical solutions to these equations.

In this work, we propose an explicit Fourier-Klibanov method for obtaining
numerical solutions to a brain tumor growth model. Tumor growth in
population dynamics can be described by various laws (cf. e.g. \cite{Murray1993}),
including the Von Bertalanffy law that involves
\begin{align*}
f\left(u\right) & =\rho u\quad\text{(exponential law)},\\
f\left(u\right) & =\rho u\left(1-\frac{u}{\mathfrak{C}_{\max}}\right)\quad\text{(Pearl-Verhulst logistic law)},
\end{align*}
and the Gompertz law,
\[
f\left(u\right)=-\rho u\ln\left(\frac{u}{e^{\mathfrak{K}/\text{d}}}\right).
\]

We have introduced above various parameters. As to the Von Bertalanffy
law, $\rho>0$ denotes the net proliferation rate ($\text{month}^{-1}$),
while $\mathfrak{C}_{\max}>0$ represents the maximum number of tumor
cells that can occupy a cubic millimeter of brain tissue. For the
Gompertz law, the parameter $\mathfrak{K}>0$ describes an exponential
increase when $u$ is small, while the damping constant $\text{d}>0$
is to constrain the growth rate when $u$ is large.

Although several studies have focused on approximating tumor growth
models using the Von Bertalanffy law (cf. e.g. \cite{Oezugurlu2015,Jaroudi2019}),
there has been limited research specifically dedicated to exploring
the Gompertz dynamic. The Gompertzian model is based on the notion
that as the tumor size increases, the tumor microenvironment becomes
more hostile, and the availability of nutrients and oxygen decreases.
Consequently, the tumor's growth rate decreases, leading to a deceleration
in tumor growth. Despite the Gompertzian model's potential for improving
our understanding of tumor dynamics, there is still much to be explored
in this area. Further research is needed to fully appreciate the implications
of the Gompertzian model in tumor growth dynamics and its potential
applications in cancer treatment.

Tumor growth models usually account for variations of tumor in space
and time, but their evolution can be further characterized using other
parametric variables; see e.g. \cite{Sinko1967} and references cited
therein for the derivation of population models with distinctive parametric
arguments. In this study, we focus on the age-dependent process of
cell division, where a dividing mother cell gives rise to two daughter
cells. By incorporating an aging variable into the above-mentioned
Gompertz model, our computational approach is introduced to approximate
the following PDE:
\begin{equation}
\partial_{t}u+\partial_{a}u-D\left(t,a\right)\partial_{xx}u+\mu\left(a\right)u=-\rho u\ln\left(\frac{u}{e^{\mathfrak{K}/\text{d}}}\right)\quad\text{for }t,a>0,x\in\left(-\ell,\ell\right),\label{eq:PDE}
\end{equation}
where $u=u\left(t,a,x\right)\ge0$ ($\text{cells (in thousands)}/\text{cm}$)
denotes the distribution of the number of tumor cells at time $t$,
age $a$ and spatial location $x$. The diffusion of the tumor throughout
the brain over time and age is described by the function $D=D\left(t,a\right)$
($\text{cm}^{2}/\text{month}$) in equation (\ref{eq:PDE}). Meanwhile,
the mortality rate of the population is presented by the age-dependent
function $\mu\left(a\right)>0$ ($\text{month}^{-1}$). As we develop
our proposed method, we simplify the spatial complexity by considering
a one-dimensional model with a length scale of $\ell>0$ ($\text{cm}$).
This simplification allows us to focus on the essential features of
the tumor growth dynamics and develop a more tractable model for numerical
simulations.

To complete the age-dependent Gompertz model, we equip (\ref{eq:PDE})
with the no-flux boundary condition:
\begin{equation}
\partial_{x}u\left(t,a,-\ell\right)=\partial_{x}u\left(t,a,\ell\right)=0\quad\text{for }t,a>0,\label{bound1}
\end{equation}
and the initial data
\begin{align}
u\left(0,a,x\right) & =u_{0}\left(a,x\right)\quad\text{for }a>0,x\in\left(-\ell,\ell\right),\label{bound2}\\
u\left(t,0,x\right) & =\overline{u}_{0}\left(t,x\right)\quad\text{for }t>0,x\in\left(-\ell,\ell\right).\label{bound3}
\end{align}

Here, we assume that $u_{0}$ and $\overline{u}_{0}$ satisfy the
standard compatibility condition $u_{0}\left(0,\cdot\right)=\overline{u}_{0}\left(0,\cdot\right)$.
The zero Neumann boundary condition is commonly imposed to model situations
where the tumor is confined to a specific region, such as a tumor
in a specific tissue or organ. Furthermore, it is assumed that the
functions involved in our PDE system possess sufficient regularity
to enable numerical simulation needed for our present purposes. Investigation
of the regularity of these functions will be undertaken in forthcoming
research.

In many cases, the initial condition for a newborn, denoted by $\overline{u}_{0}$,
is accompanied by a nonlocal operator that takes into account the
reproductive process, which is weighted by the bounded intrinsic maternity;
cf. e.g. \cite{Kim1995,Ayati2006}. However, this is not the primary
focus of our work. Instead, we only examine the regular condition
in our Gompertz model. It is worth noting that the nonlocal operator
has been successfully linearized through numerical methods in \cite{Iannelli2012}.
For every recursive step, the linearization process seeks numerical
solutions with the regular newborn boundary condition.

Also, we would like to stress that previous publications have referred
to (\ref{eq:PDE}) as a type of ultra-parabolic equations, with a
range of applications beyond the oncological context of this work.
For instance, ultra-parabolic PDEs have been found to be crucial in
describing heat transfer through a continuous medium in which the
presence of the parametric variable $a$ is due to the propagating
direction of a shock wave; cf. e.g. \cite{Lorenzi1998,Kuznetsov2018}.
Furthermore, these PDEs have been employed in mathematical finance,
specifically for the computation of call option prices. The derivation
of the ultra-parabolic PDEs has been detailed in \cite{Marcozzi2009},
utilizing the ultradiffusion process, wherein the parametric parameter
$a$ is determined by the asset price's path history. It is then worth
mentioning that in terms of the ultra-parabolic PDEs, several numerical
schemes have been developed for their approximate solution. To name
a few, some attempts have been made to design numerical methods for
linear equations \cite{Akrivis1994,Ashyralyev2012}, as well as for
nonlinear equations with a globally Lipschitzian source term \cite{Khoa2015}.

Our paper is three-fold. In section \ref{sec:2}, we focus on developing
an explicit Fourier-Klibanov method for approximating the Gompertz
model of interest. The method relies on the derivation of a new coupled
nonlinear transport-like PDE system. This can be done by an application
of some nonlinear and linear transformations and the special truncated
Fourier-Klibanov series. Subsequently, the proposed method is established
by applying the explicit finite difference method along with the characteristics
of time and age directions. Section \ref{sec:3} is devoted to the
$2$-norm stability analysis of the numerical scheme that we have
introduced in section \ref{sec:2}. Finally, to demonstrate the effectiveness
of the proposed method, numerical examples are presented in section
\ref{sec:4}.

\section{Explicit Fourier-Klibanov method\label{sec:2}}

The Fourier-Klibanov method is a technique that utilizes a Fourier
series driven by a special orthonormal basis of $L^{2}$. This basis
was first constructed in \cite{Klibanov2017}, and the Fourier-Klibanov
method has since been applied to various physical models of inverse
problems. Examples of these models include imaging of land mines,
crosswell imaging, and electrical impedance tomography, as demonstrated
in recent studies such as \cite{Klibanov2018,Khoa2020,Klibanov2020,Le2022,Klibanov2023}
and other works cited therein. Thus, our paper is the first to attempt
the application of this special basis to approximate a specific class
of nonlinear age-dependent population models.

Prior to defining the special basis, several transformations to the
PDE (\ref{eq:PDE}) are employed in the following subsection.

\subsection{Derivation of an auxiliary third-order PDE}

Let $a_{\dagger}\in\left(0,\infty\right)$ be the maximum age of the
cell population in the model. We define the survival probability,
\[
\Pi\left(a\right)=e^{-\int_{0}^{a}\mu\left(\sigma\right)d\sigma}.
\]
Now, take into account the nonlinear transformation $v=\frac{\ln\left(u/e^{\mathfrak{K}/\text{d}}\right)}{\Pi\left(a\right)}$
or $\frac{u}{e^{\mathfrak{K}/\text{d}}}=e^{\Pi\left(a\right)v}$ for
$a\in\left(0,a_{\dagger}\right)$. We compute that
\begin{equation}
\partial_{t}u=e^{\mathfrak{K}/\text{d}}\Pi\left(a\right)e^{\Pi\left(a\right)v}\partial_{t}v,\quad\partial_{a}u=e^{\mathfrak{K}/\text{d}}\Pi\left(a\right)e^{\Pi\left(a\right)v}\left(\partial_{a}v-\mu\left(a\right)v\right),\quad\partial_{x}u=e^{\mathfrak{K}/\text{d}}\Pi\left(a\right)e^{\Pi\left(a\right)v}\partial_{x}v,\label{eq:2.2}
\end{equation}
\begin{equation}
\partial_{xx}u=e^{\mathfrak{K}/\text{d}}\Pi\left(a\right)\partial_{x}\left(e^{\Pi\left(a\right)v}\partial_{x}v\right)=e^{\mathfrak{K}/\text{d}}\Pi\left(a\right)e^{\Pi\left(a\right)v}\left(\partial_{xx}v+\Pi\left(a\right)\left(\partial_{x}v\right)^{2}\right),\label{eq:2.3}
\end{equation}
\begin{equation}
\rho u\ln\left(\frac{u}{e^{\mathfrak{K}/\text{d}}}\right)=\rho e^{\Pi\left(a\right)v}\ln\left(e^{\Pi\left(a\right)v}\right)=\rho e^{\Pi\left(a\right)v}\Pi\left(a\right)v.\label{eq:2.4}
\end{equation}
Combining (\ref{eq:2.2})\textendash (\ref{eq:2.4}), we arrive at
the following PDE for $v$:
\begin{align*}
0 & =\partial_{t}u+\partial_{a}u-D\left(t,a\right)\partial_{xx}u+\mu\left(a\right)u+\rho u\ln\left(\frac{u}{e^{\mathfrak{K}/\text{d}}}\right)\\
 & =e^{\mathfrak{K}/\text{d}}\Pi\left(a\right)e^{\Pi\left(a\right)v}\left[\partial_{t}v+\partial_{a}v-\mu\left(a\right)v-D\left(t,a\right)\partial_{xx}v\right.\\
 & \left.\hspace{8em}\hspace{4em}\quad-D\left(t,a\right)\Pi\left(a\right)\left(\partial_{x}v\right)^{2}+\mu\left(a\right)\Pi^{-1}\left(a\right)+\rho e^{-\mathfrak{K}/\text{d}}v\right]\\
 & =e^{\mathfrak{K}/\text{d}}\Pi\left(a\right)e^{\Pi\left(a\right)v}\left[\partial_{t}v+\partial_{a}v-D\left(t,a\right)\left(\partial_{xx}v+\Pi\left(a\right)\left(\partial_{x}v\right)^{2}\right)\right.\\
 & \left.\hspace{8em}\hspace{4em}\qquad-\left(\mu\left(a\right)-\rho e^{-\mathfrak{K}/\text{d}}\right)v+\mu\left(a\right)\Pi^{-1}\left(a\right)\right].
\end{align*}
Vanishing $e^{\mathfrak{K}/\text{d}}\Pi\left(s\right)e^{\Pi\left(s\right)v}$
on the right-hand side of the above equation and then, applying $\partial_{x}$
to the resulting equation, we obtain the following auxiliary PDE:
\begin{equation}
\partial_{tx}v+\partial_{ax}v-D\left(t,a\right)\partial_{xxx}v-2D\left(t,a\right)\Pi\left(a\right)\partial_{x}v\partial_{xx}v-\left(\mu\left(a\right)-\rho e^{-\mathfrak{K}/\text{d}}\right)\partial_{x}v=0.\label{auxi}
\end{equation}

\subsection{A coupled transport-like system via the Fourier-Klibanov basis}

Equation (\ref{auxi}) is a non-trivial third-order PDE, and we thus
propose to apply the Fourier-Klibanov basis $\left\{ \Psi_{n}\left(x\right)\right\} _{n=1}^{\infty}$
in $L^{2}\left(-\ell,\ell\right)$ to solve it.

To construct this basis, we start by considering $\varphi_{n}\left(x\right)=x^{n-1}e^{x}$
for $x\in\left[-\ell,\ell\right]$ and $n\in\mathbb{N}$. The set
$\left\{ \varphi_{n}\left(x\right)\right\} _{n\in\mathbb{N}^{*}}$
is linearly independent and complete in $L^{2}\left(-\ell,\ell\right)$.
We then apply the standard Gram-Schmidt orthonormalization procedure
to obtain the basis $\left\{ \Psi_{n}\left(x\right)\right\} _{n=1}^{\infty}$,
which takes the form $P_{n}\left(x\right)e^{x}$, where $P_{n}\left(x\right)$
is the polynomial of the degree $n$.

The Fourier-Klibanov basis possesses the following properties:
\begin{itemize}
\item $\Psi_{n}\in C^{\infty}\left[-\ell,\ell\right]$ and $\Psi_{n}'\left(x\right)=\Psi_{n}\left(x\right)+P_{n}'\left(x\right)e^{x}$
is not identically zero for any $n\in\mathbb{N}^{*}$;
\item Let $s_{mn}=\left\langle \Psi_{n}',\Psi_{m}\right\rangle $ where
$\left\langle \cdot,\cdot\right\rangle $ denotes the scalar product
in $L^{2}\left(-\ell,\ell\right)$. Then the square matrix $S_{N}=\left(s_{mn}\right)_{m,n=1}^{N}\in\mathbb{R}^{N\times N}$
is invertible for any $N$ since
\end{itemize}
\[
s_{mn}=\begin{cases}
1 & \text{if }n=m,\\
0 & \text{if }n<m.
\end{cases}
\]

Essentially, $S_{N}$ is an upper triangular matrix with $\det\left(S_{N}\right)=1$.
\begin{rem*}
The Fourier-Klibanov basis is similar to the orthogonal polynomials
formed by the so-called Laguerre functions. However, the Laguerre
polynomials are used for the $L^{2}$ basis in the semi-infinite interval
with a decaying weighted inner product. We also notice that the second
property of the Fourier-Klibanov basis does not hold for either classical
orthogonal polynomials or the classical basis of trigonometric functions.
The first column of $S_{N}$ obtained from either of the two conventional
bases would be zero.
\end{rem*}
Let $N\ge1$ now be the cut-off constant. We will discuss how to choose
$N$ later in the numerical section. Consider the truncated Fourier
series for $v$ in the following sense:
\begin{equation}
v\left(t,a,x\right)=\sum_{n=1}^{N}\left\langle v\left(t,a,\cdot\right),\Psi_{n}\left(\cdot\right)\right\rangle \Psi_{n}\left(x\right)=\sum_{n=1}^{N}v_{n}\left(t,a\right)\Psi_{n}\left(x\right).\label{eq:truncated}
\end{equation}
Plugging this truncated series into the auxiliary PDE (\ref{auxi}),
we have
\begin{align}
 & \sum_{n=1}^{N}\left[\left(\partial_{t}+\partial_{a}\right)-\left(\mu\left(a\right)-\rho e^{-\mathfrak{K}/\text{d}}\right)\right]v_{n}\left(t,a\right)\Psi_{n}'\left(x\right)\label{eq:2.5}\\
 & =D\left(t,a\right)\sum_{n=1}^{N}v_{n}\left(t,a\right)\Psi_{n}'''\left(x\right)+2D\left(t,a\right)\Pi\left(a\right)\sum_{n=1}^{N}\sum_{k=1}^{N}v_{n}\left(t,a\right)v_{k}\left(t,a\right)\Psi_{n}'\left(x\right)\Psi_{k}''\left(x\right).\nonumber 
\end{align}
For each $1\le m\le N$, we multiply both sides of (\ref{eq:2.5})
by $\Psi_{m}$ and then integrate both sides of the resulting equation
from $x=-\ell$ to $x=\ell$. Let $V=\left(v_{1},v_{1},...,v_{N}\right)^{\text{T}}\in\mathbb{R}^{N}$
be the $N$-dimensional vector-valued function that contains all of
the Fourier coefficients $v_{n}$. We obtain the following nonlinear
PDE system:
\begin{align}
\sum_{n=1}^{N}s_{mn} & \left[\left(\partial_{t}+\partial_{a}\right)-\left(\mu\left(a\right)-\rho e^{-\mathfrak{K}/\text{d}}\right)\right]v_{n}\left(t,a\right)\label{eq:2.6}\\
 & =D\left(t,a\right)\sum_{n=1}^{N}\kappa_{mn}v_{n}\left(t,a\right)+2D\left(t,a\right)\Pi\left(a\right)\sum_{n=1}^{N}\sum_{k=1}^{N}\varsigma_{mnk}v_{n}\left(t,a\right)v_{k}\left(t,a\right),\nonumber 
\end{align}
where $\kappa_{mn}=\left\langle \Psi_{n}''',\Psi_{m}\right\rangle $
and $\varsigma_{mnk}=\left\langle \Psi_{n}'\Psi_{k}'',\Psi_{m}\right\rangle $.
It is straightforward to see that system (\ref{eq:2.6}) is of the
transport-like form. By the boundary data (\ref{bound2}) and (\ref{bound3}),
we associate (\ref{eq:2.6}) with the following boundary conditions:
\begin{align}
v_{n}\left(0,a\right) & =\left\langle v\left(0,a,\cdot\right),\Psi_{n}\left(\cdot\right)\right\rangle =\Pi^{-1}\left(a\right)\left\langle \ln\left(u_{0}\left(a,\cdot\right)/e^{\mathfrak{K}/\text{d}}\right),\Psi_{n}\left(\cdot\right)\right\rangle ,\label{ini1}\\
v_{n}\left(t,0\right) & =\left\langle v\left(t,0,\cdot\right),\Psi_{n}\left(\cdot\right)\right\rangle =\left\langle \ln\left(\overline{u}_{0}\left(t,\cdot\right)/e^{\mathfrak{K}/\text{d}}\right),\Psi_{n}\left(\cdot\right)\right\rangle .\label{ini2}
\end{align}

\subsection{Finite difference operators}

To approximate the transport-like system (\ref{eq:2.6})\textendash (\ref{ini2}),
we propose to apply the so-called finite difference method of characteristics
in the time-age direction. To do so, we take into account the time
increment $\Delta t=T/M$ for $M\ge2$ being a fixed integer. Then,
we set the mesh-point in time by $t_{i}=i\Delta t$ for $0\le i\le M$.
Using the number $M$, we define $K=\left[a_{\dagger}/\Delta t+1\right]$
and set the mesh-point in age by 
\[
a_{j}=j\Delta t\quad\text{for }0\le j<K.
\]

Thus, we discretize the differential operator $\partial_{t}+\partial_{a}$
along the characteristic $t=a$, as follows:
\[
\left(\partial_{t}+\partial_{a}\right)v_{n}\left(t_{i},a_{j}\right)=\frac{v_{n,j}^{i}-v_{n,j}^{i-1}}{\Delta t}+\frac{v_{n,j}^{i-1}-v_{n,j-1}^{i-1}}{\Delta t}=\frac{v_{n,j}^{i}-v_{n,j-1}^{i-1}}{\Delta t}.
\]
Here and to this end, the subscript $j$ indicates the age level $a_{j}$
and the superscript $i$ implies the time level $t_{i}$. By the forward
Euler procedure, we then seek the discrete solution $v_{n,j}^{i}=v_{n}\left(t_{i},a_{j}\right)$
of the following systematic scheme:
\begin{align*}
\sum_{n=0}^{N-1}s_{mn}v_{n,j}^{i} & =\sum_{n=0}^{N-1}s_{mn}v_{n,j-1}^{i-1}+\Delta t\sum_{n=0}^{N-1}\left[s_{mn}\left(\mu_{j-1}-\rho e^{-\mathfrak{K}/\text{d}}\right)+D_{j-1}^{i-1}\kappa_{mn}\right]v_{n,j-1}^{i-1}\\
 & +2\Delta tD_{j-1}^{i-1}\Pi_{j-1}\sum_{n=0}^{N-1}\sum_{k=0}^{N-1}\varsigma_{mnk}v_{n,j-1}^{i-1}v_{k,j-1}^{i-1}
\end{align*}
for $i,j\ge1$. Now, for every step $\left(i,j\right)$, let $\mathbb{K}_{j}^{i}=\left(\mathbb{K}_{mn,j}^{i}\right)_{m,n=1}^{N}$
with $\mathbb{K}_{mn,j}^{i}=s_{mn}\left(\mu_{j}-\rho e^{-\mathfrak{K}/\text{d}}\right)+D_{j}^{i}\kappa_{mn}$
and let $\mathbb{G}_{m,j}^{i}=\left(\mathbb{G}_{mnk,j}^{i}\right)_{n,k=1}^{N}$
with $\mathbb{G}_{mnk,j}^{i}=2D_{j}^{i}\Pi_{j}\varsigma_{mnk}$. Henceforth,
our discretized coupled system has the following form:
\begin{equation}
S_{N}V_{j}^{i}=S_{N}V_{j-1}^{i-1}+\Delta t\mathbb{K}_{j-1}^{i-1}V_{j-1}^{i-1}+\Delta t\begin{bmatrix}\left(V_{j-1}^{i-1}\right)^{\text{T}}\mathbb{G}_{1,j-1}^{i-1}V_{j-1}^{i-1}\\
\left(V_{j-1}^{i-1}\right)^{\text{T}}\mathbb{G}_{2,j-1}^{i-1}V_{j-1}^{i-1}\\
\left(V_{j-1}^{i-1}\right)^{\text{T}}\mathbb{G}_{3,j-1}^{i-1}V_{j-1}^{i-1}\\
\vdots\\
\left(V_{j-1}^{i-1}\right)^{\text{T}}\mathbb{G}_{N,j-1}^{i-1}V_{j-1}^{i-1}
\end{bmatrix}.\label{discrete}
\end{equation}

By (\ref{ini1}) and (\ref{ini2}), this system is associated with
the starting points $V_{j}^{0}=\left(v_{n}\left(0,a_{j}\right)\right)_{n=1}^{N}$
and $V_{0}^{i}=\left(v_{n}\left(t_{i},0\right)\right)_{n=1}^{N}$.

\section{Stability analysis\label{sec:3}}

In this section, we want to analyze the stability of the proposed
explicit approach. Before doing so, we want to ensure that the discrete
solution $V_{j}^{i}$ is uniformly bounded for any bounded data $V_{j}^{0}$
and $V_{0}^{i}$. To this end, we make use of the following $2$-norm
of any vector $X=\left(X_{n}\right)_{n=1}^{N}\in\mathbb{R}^{N}$.
\begin{equation}
\left|X\right|_{2}=\left(\sum_{1\le n\le N}\left|X_{n}\right|^{2}\right)^{1/2}.\label{eq:3.1}
\end{equation}
Also, for any square matrix $X=\left(X_{mn}\right)_{m,n=1}^{N}\in\mathbb{R}^{N\times N}$,
we use the Frobenius norm,
\begin{equation}
\left|X\right|_{F}=\left(\sum_{m=1}^{N}\sum_{n=1}^{N}\left|X_{mn}\right|^{2}\right)^{1/2}.\label{eq:3.2}
\end{equation}

It follows from (\ref{discrete}) that for $i,j\ge1$,

\begin{equation}
S_{N}V_{j}^{i}=\begin{cases}
S_{N}V_{0}^{i-j}+\Delta t\sum_{\mathfrak{i}=0}^{j-1}\mathbb{F}\left(V_{\mathfrak{i}}^{\mathfrak{i}+i-j}\right), & \text{if }i\ge j,\\
S_{N}V_{j-i}^{0}+\Delta t\sum_{\mathfrak{j}=0}^{i-1}\mathbb{F}\left(V_{\mathfrak{j}+j-i}^{\mathfrak{j}}\right), & \text{if }i<j,
\end{cases}\label{compact}
\end{equation}
where the nonlinear term $\mathbb{F}:\mathbb{R}^{N}\to\mathbb{R}^{N}$
is defined as
\[
\mathbb{F}\left(V_{\mathfrak{j}}^{\mathfrak{i}}\right)=\mathbb{K}_{\mathfrak{j}}^{\mathfrak{i}}V_{\mathfrak{j}}^{\mathfrak{i}}+\begin{bmatrix}\left(V_{\mathfrak{j}}^{\mathfrak{i}}\right)^{\text{T}}\mathbb{G}_{1,\mathfrak{j}}^{\mathfrak{i}}V_{\mathfrak{j}}^{\mathfrak{i}}\\
\left(V_{\mathfrak{j}}^{\mathfrak{i}}\right)^{\text{T}}\mathbb{G}_{2,\mathfrak{j}}^{\mathfrak{i}}V_{\mathfrak{j}}^{\mathfrak{i}}\\
\left(V_{\mathfrak{j}}^{\mathfrak{i}}\right)^{\text{T}}\mathbb{G}_{3,\mathfrak{j}}^{\mathfrak{i}}V_{\mathfrak{j}}^{\mathfrak{i}}\\
\vdots\\
\left(V_{\mathfrak{j}}^{\mathfrak{i}}\right)^{\text{T}}\mathbb{G}_{N,\mathfrak{j}}^{\mathfrak{i}}V_{\mathfrak{j}}^{\mathfrak{i}}
\end{bmatrix}.
\]

Let $S_{N}^{-1}=\left(\tilde{s}_{mn}\right)_{m,n=1}^{N}\in\mathbb{R}^{N\times N}$
be the inverse of $S_{N}$. We can rewrite (\ref{compact}) as
\begin{equation}
V_{j}^{i}=\begin{cases}
V_{0}^{i-j}+\Delta tS_{N}^{-1}\sum_{\mathfrak{i}=0}^{j-1}\mathbb{F}\left(V_{\mathfrak{i}}^{\mathfrak{i}+i-j}\right), & \text{if }i\ge j,\\
V_{j-i}^{0}+\Delta tS_{N}^{-1}\sum_{\mathfrak{j}=0}^{i-1}\mathbb{F}\left(V_{\mathfrak{j}+j-i}^{\mathfrak{j}}\right), & \text{if }i<j.
\end{cases}\label{compact2}
\end{equation}

\begin{thm}
\label{thm:2}Assume that there exists a constant $C>0$ independent
of $i$ and $j$ such that
\begin{equation}
\max_{j\ge0}\left|V_{j}^{0}\right|_{2}\le C,\quad\text{and }\max_{i\ge0}\left|V_{i}^{0}\right|_{2}\le C.
\end{equation}
Moreover, suppose that for each $i,j\ge0$, we can find a constant
$P_{j}^{i}\ge0$ such that
\begin{equation}
\left|\mathbb{K}_{j}^{i}\right|_{F}+\sum_{m=1}^{N}\left|\mathbb{G}_{m,\mathfrak{j}}^{\mathfrak{i}}\right|_{F}\le P_{j}^{i}.\label{Pij}
\end{equation}
Then for any time step $\Delta t$ sufficiently small with
\begin{equation}
\Delta t\left|S_{N}^{-1}\right|_{F}\sum_{i,j\ge0}P_{j}^{i}\le\ln\left(\frac{C+1}{C+\frac{1}{2}}\right),\label{Deltat}
\end{equation}
the discrete solution $V_{j}^{i}$ in (\ref{compact}) is bounded
by
\begin{equation}
\max_{i,j\ge0}\left|V_{j}^{i}\right|_{2}\le2C.\label{target}
\end{equation}
\end{thm}

\begin{proof}
It is straightforward to see that
\begin{align*}
\left|\mathbb{F}\left(V_{\mathfrak{j}}^{\mathfrak{i}}\right)\right|_{2} & \le\left|\mathbb{K}_{\mathfrak{j}}^{\mathfrak{i}}V_{\mathfrak{j}}^{\mathfrak{i}}\right|_{2}+\left|\left(V_{\mathfrak{j}}^{\mathfrak{i}}\right)^{\text{T}}\mathbb{G}_{\mathfrak{j}}^{\mathfrak{i}}V_{\mathfrak{j}}^{\mathfrak{i}}\right|_{2}\\
 & \le\left|\mathbb{K}_{\mathfrak{j}}^{\mathfrak{i}}\right|_{F}\left|V_{\mathfrak{j}}^{\mathfrak{i}}\right|_{2}+\sum_{m=1}^{N}\left|\mathbb{G}_{m,\mathfrak{j}}^{\mathfrak{i}}\right|_{F}\left|V_{\mathfrak{j}}^{\mathfrak{i}}\right|_{2}^{2}.
\end{align*}
Therefore, by (\ref{Pij}) and (\ref{compact2}), we estimate that
for $i\ge j$,
\begin{align}
\left|V_{j}^{i}\right|_{2} & \le\left|V_{0}^{i-j}\right|_{2}+\Delta t\left|S_{N}^{-1}\right|_{F}\sum_{\mathfrak{i}=0}^{j-1}\left|\mathbb{F}\left(V_{\mathfrak{i}}^{\mathfrak{i}+i-j}\right)\right|_{2}\nonumber \\
 & \le C+\Delta t\left|S_{N}^{-1}\right|_{F}\sum_{\mathfrak{i}=0}^{j-1}P_{\mathfrak{i}}^{\mathfrak{i}+i-j}\left(\left|V_{\mathfrak{i}}^{\mathfrak{i}+i-j}\right|_{2}+\left|V_{\mathfrak{i}}^{\mathfrak{i}+i-j}\right|_{2}^{2}\right).\label{3.7}
\end{align}
Now, with the aid of the discrete Gronwall-Bellman-Ou-Iang inequality
(cf. \cite[Theorem 2.1]{Cheung2006}) applied to (\ref{3.7}), we
deduce that
\begin{equation}
\left|V_{j}^{i}\right|_{2}\le\Phi^{-1}\left[\Phi\left(C\right)+\Delta t\left|S_{N}^{-1}\right|_{F}\sum_{\mathfrak{i}=0}^{i-1}\sum_{\mathfrak{j}=0}^{j-1}P_{\mathfrak{j}}^{\mathfrak{i}}\right],\label{3.8}
\end{equation}
where $\Phi^{-1}$ is the inverse of $\Phi$, and $\Phi$ is given
by
\[
\Phi\left(f\right)=\int_{1}^{f}\frac{dr}{r+r^{2}}=\ln\left(\frac{2f}{f+1}\right)\quad\text{for }f>0.
\]
Herewith, the denominator is obtained from the structure of the second
component on the right-hand side of the bound (\ref{3.7}). With this
structure in mind, we can find $\Phi^{-1}$ as follows:
\[
\Phi^{-1}\left(f\right)=\frac{e^{f}}{2-e^{f}}=\frac{1}{2e^{-f}-1}\quad\text{for }f\in\left(0,\ln2\right).
\]
Next, by our choice in (\ref{Deltat}), we mean that
\[
\left(C+1\right)\exp\left(-\Delta t\left|S_{N}^{-1}\right|_{F}\sum_{\mathfrak{i}=0}^{i-1}\sum_{\mathfrak{j}=0}^{j-1}P_{\mathfrak{j}}^{\mathfrak{i}}\right)\ge C+\frac{1}{2}>C,
\]
Henceforth, for $i\ge j$, it follows from (\ref{3.8}) that
\begin{align*}
\left|V_{j}^{i}\right|_{2} & \le\left[2\exp\left(-\Phi\left(C\right)-\Delta t\left|S_{N}^{-1}\right|_{F}\sum_{\mathfrak{i}=0}^{i-1}\sum_{\mathfrak{j}=0}^{j-1}P_{\mathfrak{j}}^{\mathfrak{i}}\right)-1\right]^{-1}\\
 & =\left[2\exp\left(-\ln\left(\frac{2C}{C+1}\right)-\Delta t\left|S_{N}^{-1}\right|_{F}\sum_{\mathfrak{i}=0}^{i-1}\sum_{\mathfrak{j}=0}^{j-1}P_{\mathfrak{j}}^{\mathfrak{i}}\right)-1\right]^{-1}\\
 & =\left[\frac{C+1}{C}\exp\left(-\Delta t\left|S_{N}^{-1}\right|_{F}\sum_{\mathfrak{i}=0}^{i-1}\sum_{\mathfrak{j}=0}^{j-1}P_{\mathfrak{j}}^{\mathfrak{i}}\right)-1\right]^{-1}.
\end{align*}
which allows us to obtain the target estimate (\ref{target}) for
$i\ge j$. Using the same argument, we also have (\ref{target}) for
$i<j$ by means of the following estimate:
\[
\left|V_{j}^{i}\right|_{2}\le C+\Delta t\left|S_{N}^{-1}\right|_{F}\sum_{\mathfrak{j}=0}^{i-1}P_{\mathfrak{j}+j-i}^{\mathfrak{j}}\left(\left|V_{\mathfrak{j}+j-i}^{\mathfrak{j}}\right|_{2}+\left|V_{\mathfrak{j}+j-i}^{\mathfrak{j}}\right|_{2}^{2}\right)\quad\text{for }i<j.
\]

Hence, we complete the proof of the theorem.
\end{proof}
Theorem \ref{thm:2} implies that for any $i,j\ge1$, the discrete
solution $V_{j}^{i}$ of the numerical scheme and its initial data
stay in the same ball $\mathcal{B}\left(0,2C\right)$ under the topology
of the $2$-norm. This essence allows us to analyze the stability
of the proposed scheme (\ref{compact}). Let $W_{j}^{i}=V_{j}^{i}-\tilde{V}_{j}^{i}$,
where $V_{j}^{i}$ and $\tilde{V}_{j}^{i}$ satisfy \ref{compact2}
corresponding to the initial data $\left(V_{j}^{0},V_{0}^{i}\right)$
and $\left(\tilde{V}_{j}^{0},\tilde{V}_{0}^{i}\right)$, respectively.
Our stability analysis is formulated in the following theorem.
\begin{thm}
Under the assumptions of Theorem \ref{thm:2}, we can show that for
all $i,j\ge0$, the following estimate holds true:
\[
\left|W_{j}^{i}\right|_{2}\le\left(\left|W_{0}^{i-j}\right|_{2}+\left|W_{j-i}^{0}\right|_{2}\right)\left[1+\left(1+4C\right)\ln\left(\frac{C+1}{C+\frac{1}{2}}\right)\right].
\]
\end{thm}

\begin{proof}
From \ref{compact2},we can compute that for $i\ge j$,
\begin{equation}
W_{j}^{i}=W_{0}^{i-j}+\Delta tS_{N}^{-1}\sum_{\mathfrak{i}=0}^{j-1}\left[\mathbb{F}\left(V_{\mathfrak{i}}^{\mathfrak{i}+i-j}\right)-\mathbb{F}\left(\tilde{V}_{\mathfrak{i}}^{\mathfrak{i}+i-j}\right)\right].\label{eq:WW}
\end{equation}
In view of the fact that
\begin{align*}
V^{T}G_{m}V-\tilde{V}^{T}G_{m}\tilde{V} & =V^{T}\left(G_{m}V-G_{m}\tilde{V}\right)+\left(V^{T}-\tilde{V}^{T}\right)G_{m}\tilde{V}\\
 & =V^{T}G_{m}\left(V-\tilde{V}\right)+\left(V^{T}-\tilde{V}^{T}\right)G_{m}\tilde{V},
\end{align*}
we, in conjunction with (\ref{Pij}) and (\ref{target}), have
\begin{align*}
\left|\mathbb{F}\left(V_{\mathfrak{j}}^{\mathfrak{i}}\right)-\mathbb{F}\left(\tilde{V}_{\mathfrak{j}}^{\mathfrak{i}}\right)\right|_{2} & \le\left|\mathbb{K}_{\mathfrak{j}}^{\mathfrak{i}}\right|_{F}\left|W_{\mathfrak{j}}^{\mathfrak{i}}\right|_{2}+\sum_{m=1}^{N}\left|\mathbb{G}_{m,\mathfrak{j}}^{\mathfrak{i}}\right|_{F}\left(\left|V_{\mathfrak{j}}^{\mathfrak{i}}\right|_{2}+\left|\tilde{V}_{\mathfrak{j}}^{\mathfrak{i}}\right|_{2}\right)\left|W_{\mathfrak{j}}^{\mathfrak{i}}\right|_{2}\\
 & \le P_{j}^{i}\left(1+4C\right)\left|W_{\mathfrak{j}}^{\mathfrak{i}}\right|_{2}.
\end{align*}
Therefore, applying to the $2$-norm to (\ref{eq:WW}), we estimate
that
\begin{align*}
\left|W_{j}^{i}\right|_{2} & \le\left|W_{0}^{i-j}\right|_{2}+\Delta t\left|S_{N}^{-1}\right|_{F}\sum_{\mathfrak{i}=0}^{j-1}\left|\mathbb{F}\left(V_{\mathfrak{i}}^{\mathfrak{i}+i-j}\right)-\mathbb{F}\left(\tilde{V}_{\mathfrak{i}}^{\mathfrak{i}+i-j}\right)\right|_{2}\\
 & \le\left|W_{0}^{i-j}\right|_{2}+\Delta t\left|S_{N}^{-1}\right|_{F}\sum_{\mathfrak{i}=0}^{j-1}P_{j}^{i}\left(1+4C\right)\left|W_{\mathfrak{j}}^{\mathfrak{i}}\right|_{2}.
\end{align*}
Using the Salem-Raslan inequality (cf. \cite[Theorem 6.1.3]{Qin2016}),
we thus obtain that for $i\ge j$,
\[
\left|W_{j}^{i}\right|_{2}\le\left|W_{0}^{i-j}\right|_{2}\left[1+\Delta t\left|S_{N}^{-1}\right|_{F}\left(1+4C\right)\sum_{\mathfrak{i}=0}^{j-1}P_{j}^{i}\right]\le\left|W_{0}^{i-j}\right|_{2}\left[1+\left(1+4C\right)\ln\left(\frac{C+1}{C+\frac{1}{2}}\right)\right].
\]
For $i<j$, we can follow the same procedure to get
\[
\left|W_{j}^{i}\right|_{2}\le\left|W_{j-i}^{0}\right|_{2}\left[1+\left(1+4C\right)\ln\left(\frac{C+1}{C+\frac{1}{2}}\right)\right].
\]
Hence, we complete the proof of the theorem.
\end{proof}

\section{Numerical results\label{sec:4}}

In this section, we want to verify the numerical performance of the
proposed explicit Fourier-Klibanov method. By the inception of the
method, comparing it with the other numerical methods is not within
the scope of this paper. In fact, from our best knowledge, there appears
to be no numerical schemes studied for the age-dependent population
diffusion model of Gompertz type.

In our population model of interest, the mortality function is often
chosen as $\mu\left(a\right)=\left(a_{\dagger}-a\right)^{-1}$, which
is unbounded at $a=a_{\dagger}$. We will use this function in all
numerical experiments. Furthermore, for this particular choice, the
survival probability can be computed explicitly,
\[
\Pi\left(a\right)=e^{-\int_{0}^{a}\mu\left(\sigma\right)d\sigma}=\frac{a_{\dagger}-a}{a_{\dagger}},\quad\text{and }\Pi\left(a_{\dagger}\right)=0.
\]
Besides, we fix $\ell=1$ and $\Delta x=0.05$ , indicating that we
are looking at the dynamics of tumor cells within a length scale of
2 (cm) and a step-size of $0.5$ (mm). We also fix $T=10$ (months)
and examine the model with the maximum age of $a_{\dagger}=12$ (months).
Besides, $\mathfrak{K}=\text{d}=1$ are dimensionless and fixed.
\begin{rem}
\label{rem:Total}The population dynamics of cancer typically involve
determining the total population of tumor cells within specific age
and spatial ranges. This information is crucial in understanding the
local/global burden of cancer in the body, including the size and
extent of the tumor. In our numerical experiments, we consider
\[
p\left(t\right)=\int_{-\ell}^{\ell}\int_{0}^{a_{\dagger}}u\left(t,a,x\right)dadx
\]
as the (global) total population in time. In addition to ensuring
numerical stability of the density $u\left(t,a,x\right)$, we also
place importance on the stability of the total population, as it allows
us to obtain a macroscopic perspective of the entire simulation. To
approximate the time-dependent two-dimensional function $p\left(t\right)$
based on discrete approximations of $u\left(t,a,x\right)$ (via the
proposed explicit Fourier-Klibanov scheme), we use the standard trapezoidal
rule. It is worth noting that, as mentioned above, we have fixed the
following parameters: $\ell=1,a_{\dagger}=12$ and $\Delta x=0.05$,
while $\Delta a$ varies depending on the chosen value of $M$.
\end{rem}

\subsection{Choice of the cut-off constant}

We discuss how we determine the cut-off constant $N$ for the truncated
Fourier-Klibanov series. For each example, we explicitly select the
initial data $u_{0}\left(a,x\right)$, which we use to calculate $v_{0}\left(a,x\right)$
through the nonlinear transformation $v=\frac{\ln\left(u/e^{\mathfrak{K}/\text{d}}\right)}{\Pi\left(a\right)}$.
Denote the resulting transformed initial data as $v_{0}^{\text{true}}\left(a,x\right)$
and its corresponding approximation as $v_{0,N}^{\text{true}}\left(a,x\right)$.
With the explicit form of $v_{0}^{\text{true}}\left(a,x\right)$ and
the basis $\left\{ \Psi_{n}\left(x\right)\right\} _{n=1}^{\infty}$,
we can plug them into the truncated series (\ref{eq:truncated}) to
compute $v_{0,N}^{\text{true}}\left(a,x\right)$. This allows us to
compute the following relative max error:

\[
E_{\max}\left(v_{0}^{\text{true}},v_{0,N}^{\text{true}}\right)=\frac{\max_{j,l\ge0}\left|v_{0}^{\text{true}}\left(a_{j},x_{l}\right)-v_{0,N}^{\text{true}}\left(a_{j},x_{l}\right)\right|}{\max_{j,l\ge0}\left|v_{0}^{\text{true}}\left(a_{j},x_{l}\right)\right|}\times100\%.
\]

For convenience, we set the age step to 40, ensuring consistency with
the number of nodes in the space variable. It should be noted that
$\Delta x=0.05$, $\ell=1$ and $a_{\dagger}=12$ remain fixed in
seeking $N$. 

Cf. Table \ref{tab:N}, we can see that the relative max error, $E_{\max}\left(v_{0}^{\text{true}},v_{0,N}^{\text{true}}\right)$,
decreases rapidly as $N$ increases. In Example 1, increasing $N$
from 2 to 4 results in a reduction of $E_{\max}$ by a factor of approximately
30, while it is 10 in Examples 2 and 3. To better address how well
the truncated Fourier-Klibanov works, we present in Figure \ref{fig:1}
graphical representations of $v_{0}^{\text{true}}$ and $v_{0,N}^{\text{true}}$
for $N=2$ and $N=6$ for all examples.

Observe the first row of Figure \ref{fig:1}, which represents the
approximation by the truncated series for Example 1. When $N=2$,
the approximation is not good in terms of two criteria: the maximum
absolute value and the shape of the graph. However, when $N=6$, the
approximation fulfills both criteria very well.

The second row of Figure \ref{fig:1} shows the same performance for
Example 2. Since the Fourier-Klibanov series $v_{0,N}^{\text{true}}$
approximates $v_{0}^{\text{true}}$ so well in this example, we deliberately
present the log scale for the value of such $v_{0}^{\text{true}}$
and $v_{0,N}^{\text{true}}$ to show the convergence. Without the
log scale, it is difficult to see the significant improvement of the
series in terms of the two criteria when increasing $N$ from 2 to
6.

The last row of Figure \ref{fig:1} also demonstrates the same performance.
When $N=2$, $v_{0}^{\text{true}}$ and $v_{0,N}^{\text{true}}$ have
the same shape, yielding a relatively good accuracy of the series.
When $N$ is increased to 6, the value of $v_{0,N}^{\text{true}}$
is very close to $v_{0}^{\text{true}}$.

It is noteworthy that the maximum absolute value of $v_{0}^{\text{true}}$
in all examples is very large. However, despite this, the truncated
Fourier-Klibanov series demonstrates high accuracy in the relative
maximum error. Overall, the graphical representation shows that the
approximation for $N=6$ is in complete agreement with the true value,
confirming the accuracy we have analyzed. Based on these numerical
observations, we choose $N=6$ for all subsequent experiments.
\begin{center}
\begin{table}
\begin{centering}
\begin{tabular}{|c|c|c|c|c|c|c|c|c|}
\hline 
Example & $N$ & $E_{\max}$ & Example & $N$ & $E_{\max}$ & Example & $N$ & $E_{\max}$\tabularnewline
\hline 
\multirow{3}{*}{1} & 2 & 29.35\% & \multirow{3}{*}{2} & 2 & 61.68\% & \multirow{3}{*}{3} & 2 & 61.33\%\tabularnewline
\cline{2-3} \cline{3-3} \cline{5-6} \cline{6-6} \cline{8-9} \cline{9-9} 
 & 4 & 0.972\% &  & 4 & 6.746\% &  & 4 & 5.951\%\tabularnewline
\cline{2-3} \cline{3-3} \cline{5-6} \cline{6-6} \cline{8-9} \cline{9-9} 
 & 6 & 0.012\% &  & 6 & 0.160\% &  & 6 & 0.134\%\tabularnewline
\hline 
\end{tabular}
\par\end{centering}
\caption{The relative max error $E_{\max}\left(v_{0}^{\text{true}},v_{0,N}^{\text{true}}\right)$
decreases when $N$ becomes large. This is observed while keeping
the mesh grid of space and age fixed with $\Delta x=\Delta t=0.05$.
By these numerical observations, taking $N=6$ is acceptable for all
numerical experiments in the sense that the relative max error is
sufficiently small.\label{tab:N}}

\end{table}
\par\end{center}

\begin{center}
\begin{figure}
\begin{centering}
\subfloat[$v_{0}^{\text{true}}$]{\begin{centering}
\includegraphics[scale=0.115]{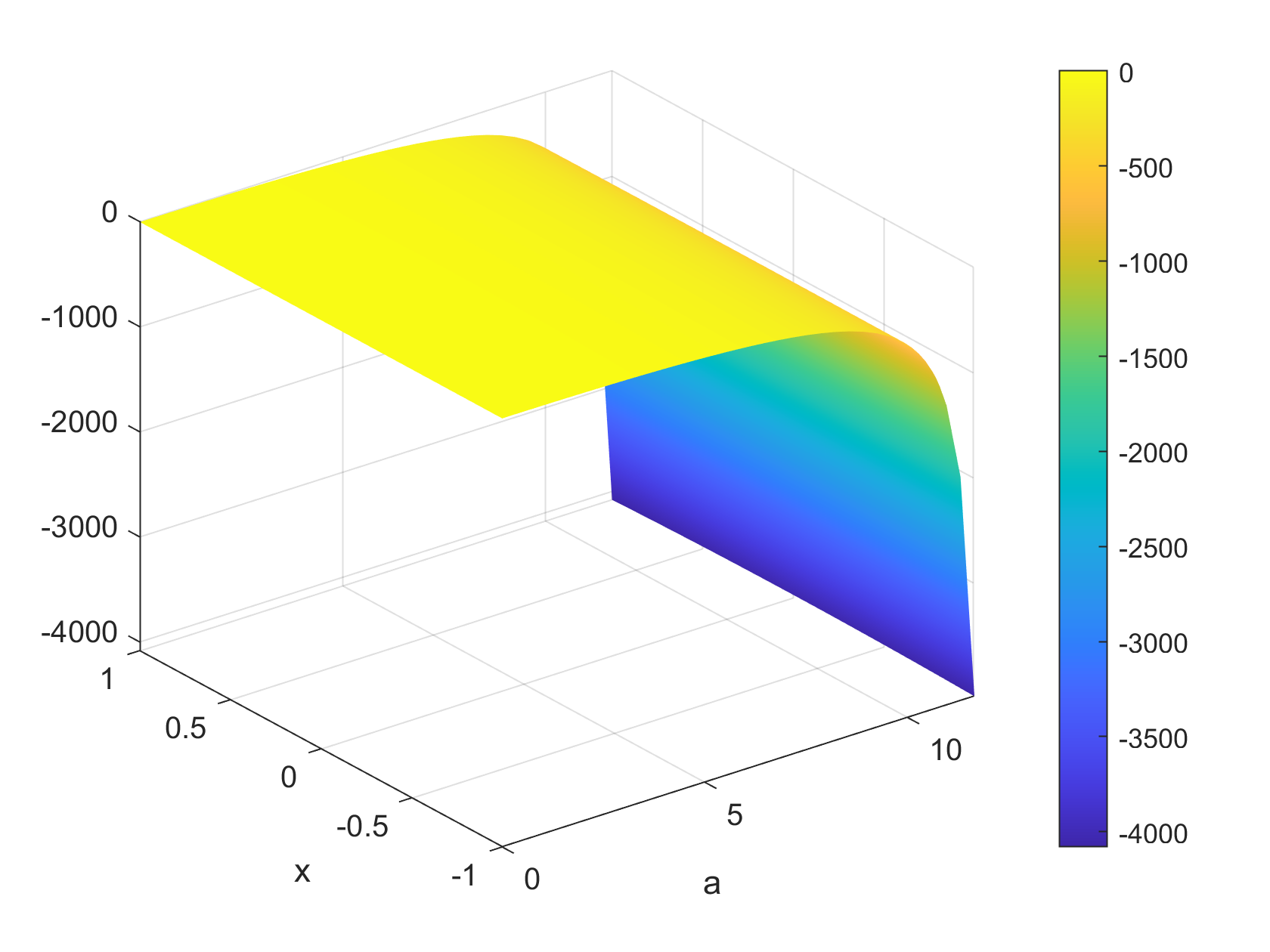}
\par\end{centering}
}\subfloat[$v_{0,N}^{\text{true}}$ ($N=2$)]{\begin{centering}
\includegraphics[scale=0.115]{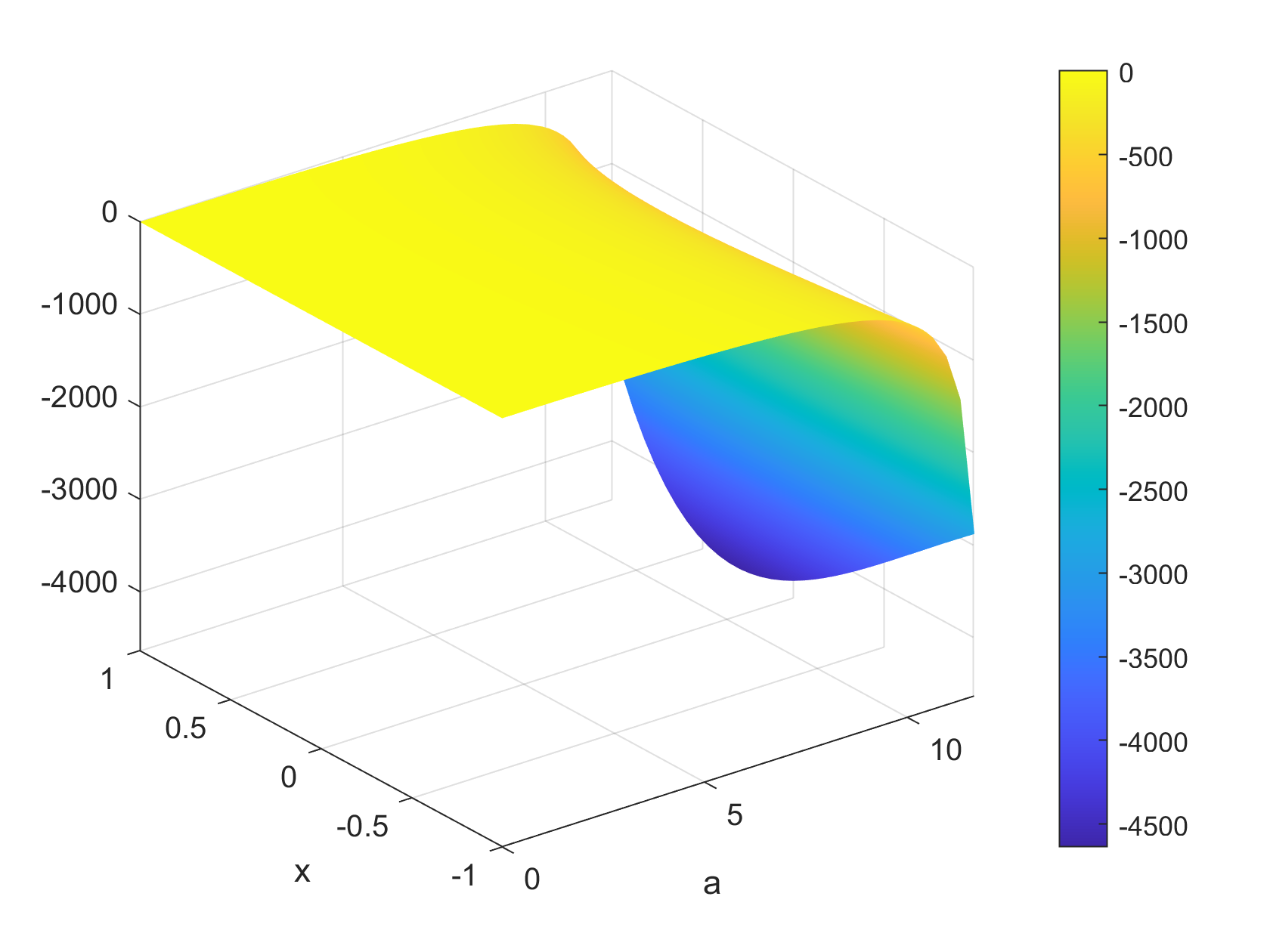}
\par\end{centering}
}\subfloat[$v_{0,N}^{\text{true}}$ ($N=6$)]{\begin{centering}
\includegraphics[scale=0.115]{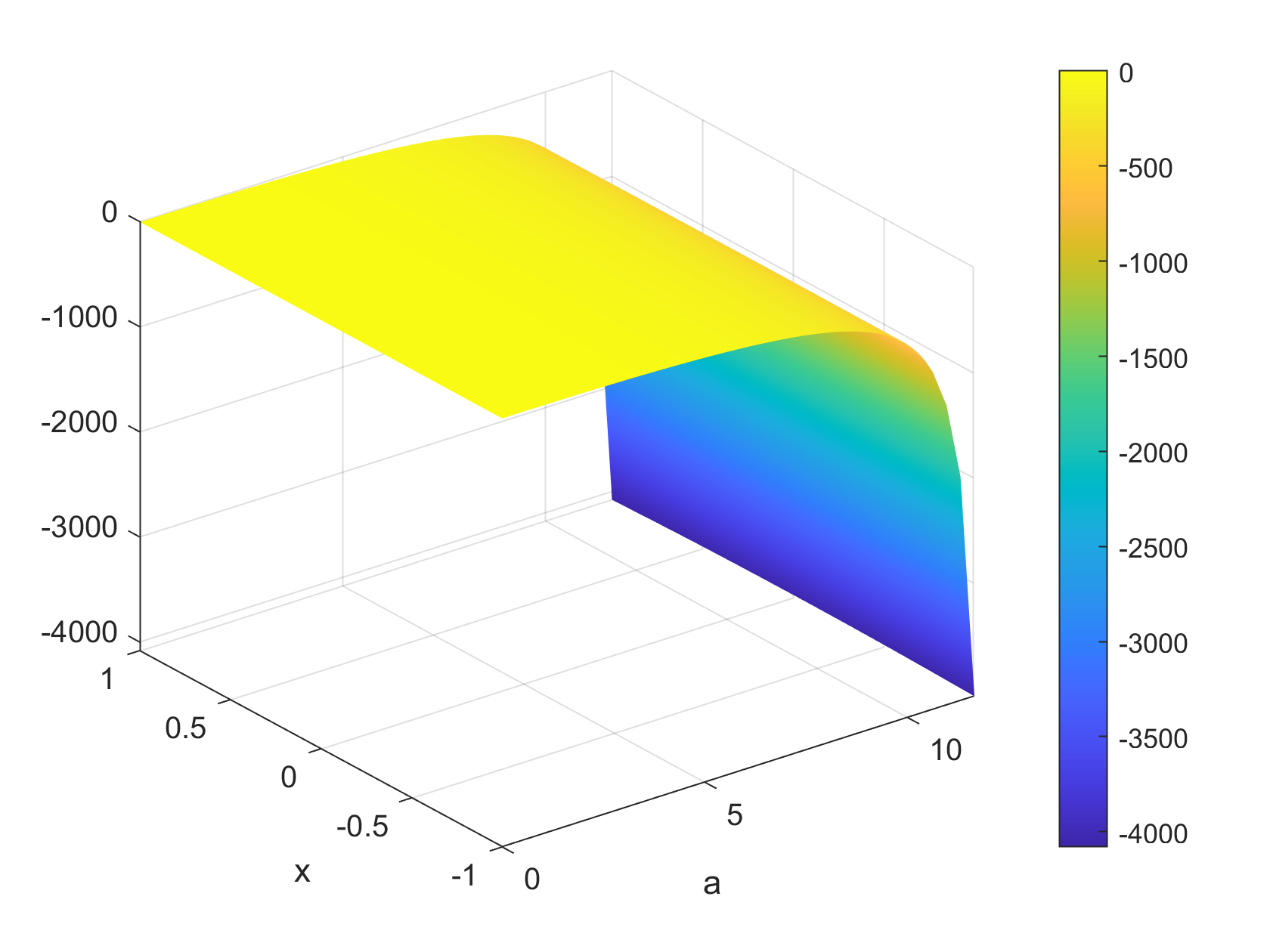}
\par\end{centering}
}
\par\end{centering}
\begin{centering}
\subfloat[$v_{0}^{\text{true}}$]{\begin{centering}
\includegraphics[scale=0.2]{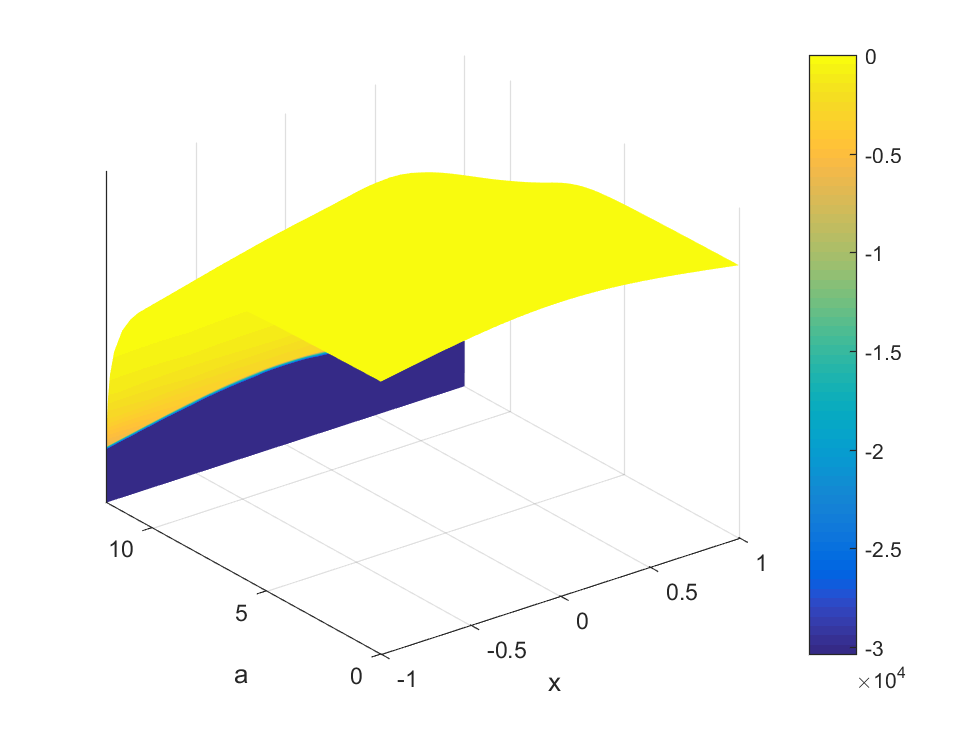}
\par\end{centering}
}\subfloat[$v_{0,N}^{\text{true}}$ ($N=2$)]{\begin{centering}
\includegraphics[scale=0.2]{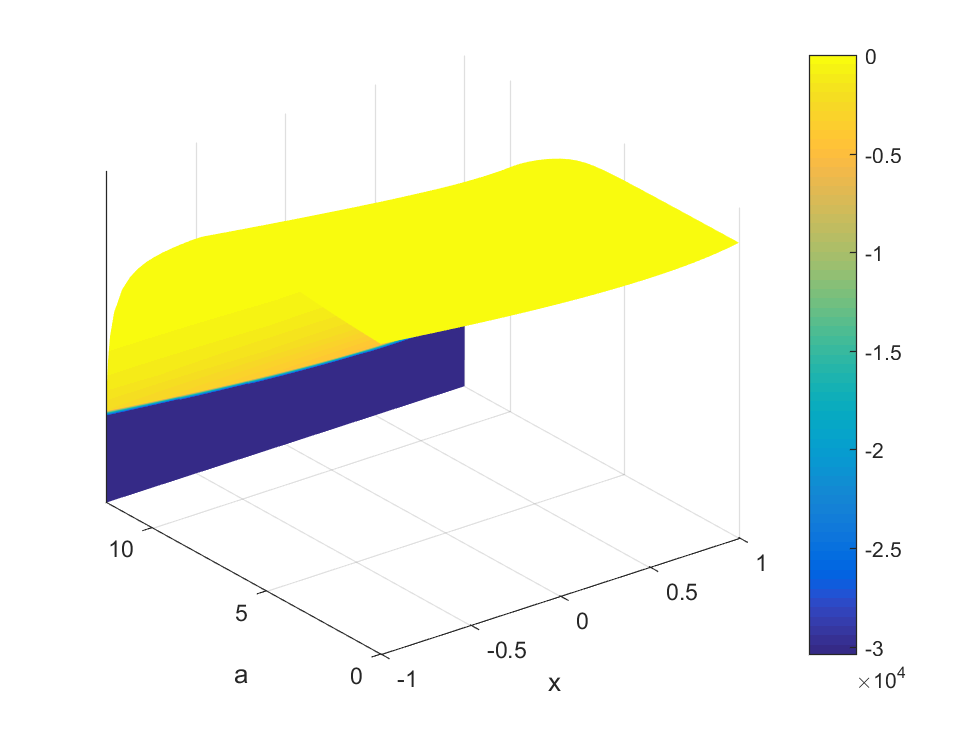}
\par\end{centering}
}\subfloat[$v_{0,N}^{\text{true}}$ ($N=6$)]{\begin{centering}
\includegraphics[scale=0.2]{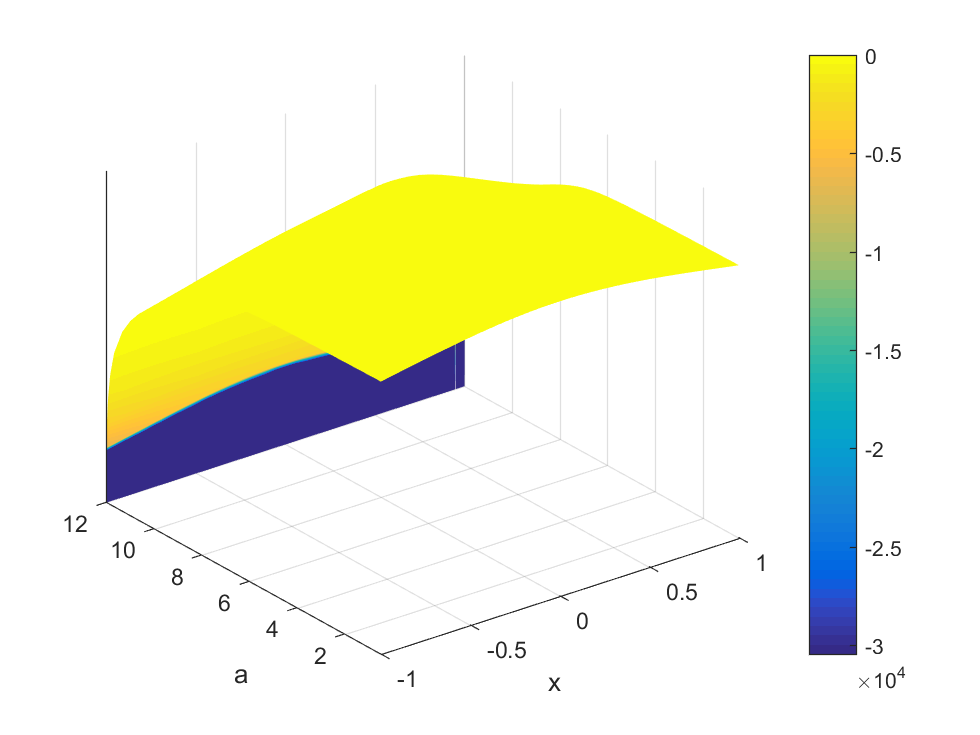}
\par\end{centering}
}
\par\end{centering}
\begin{centering}
\subfloat[$v_{0}^{\text{true}}$]{\begin{centering}
\includegraphics[scale=0.115]{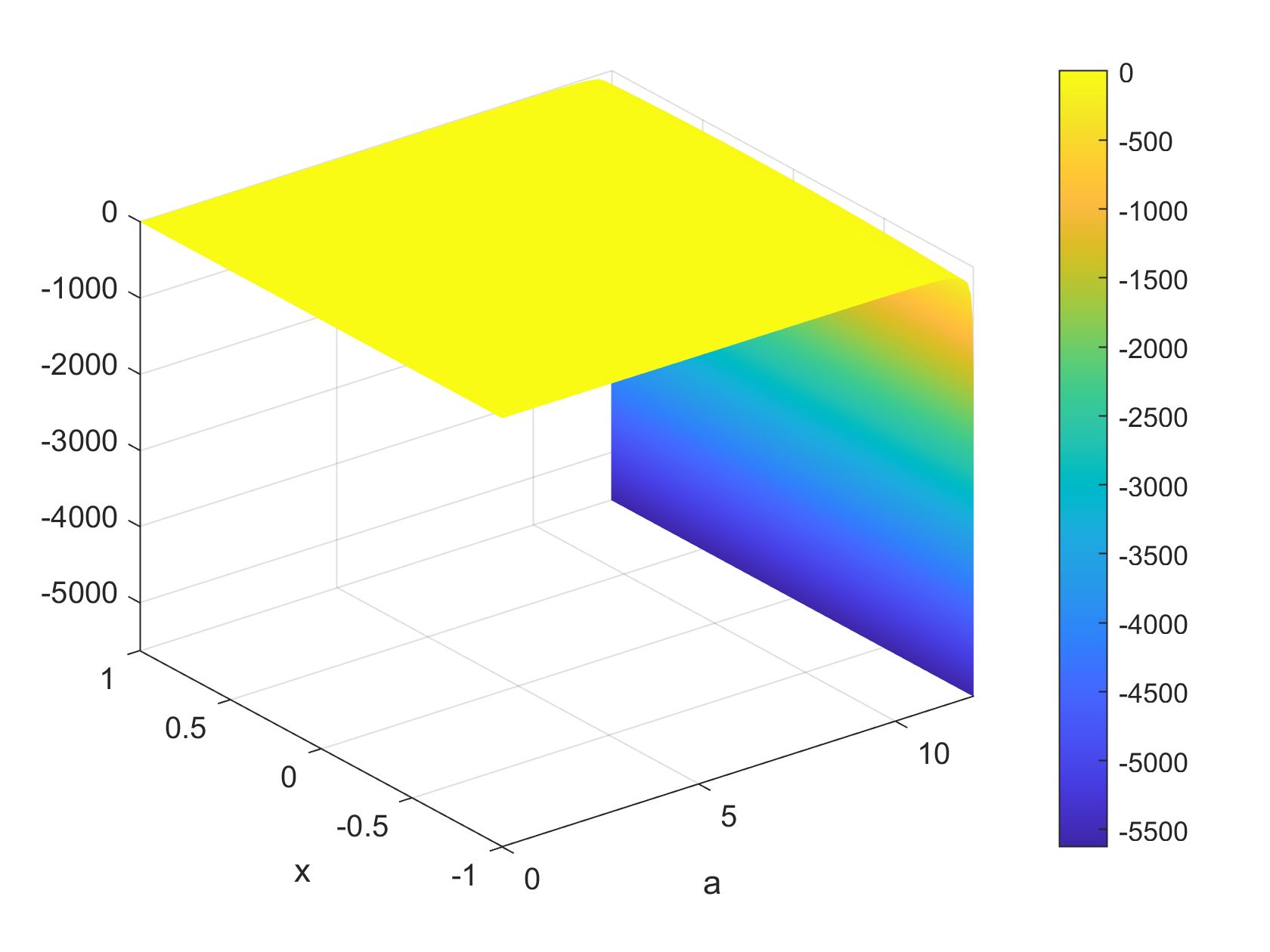}
\par\end{centering}
}\subfloat[$v_{0,N}^{\text{true}}$ ($N=2$)]{\begin{centering}
\includegraphics[scale=0.115]{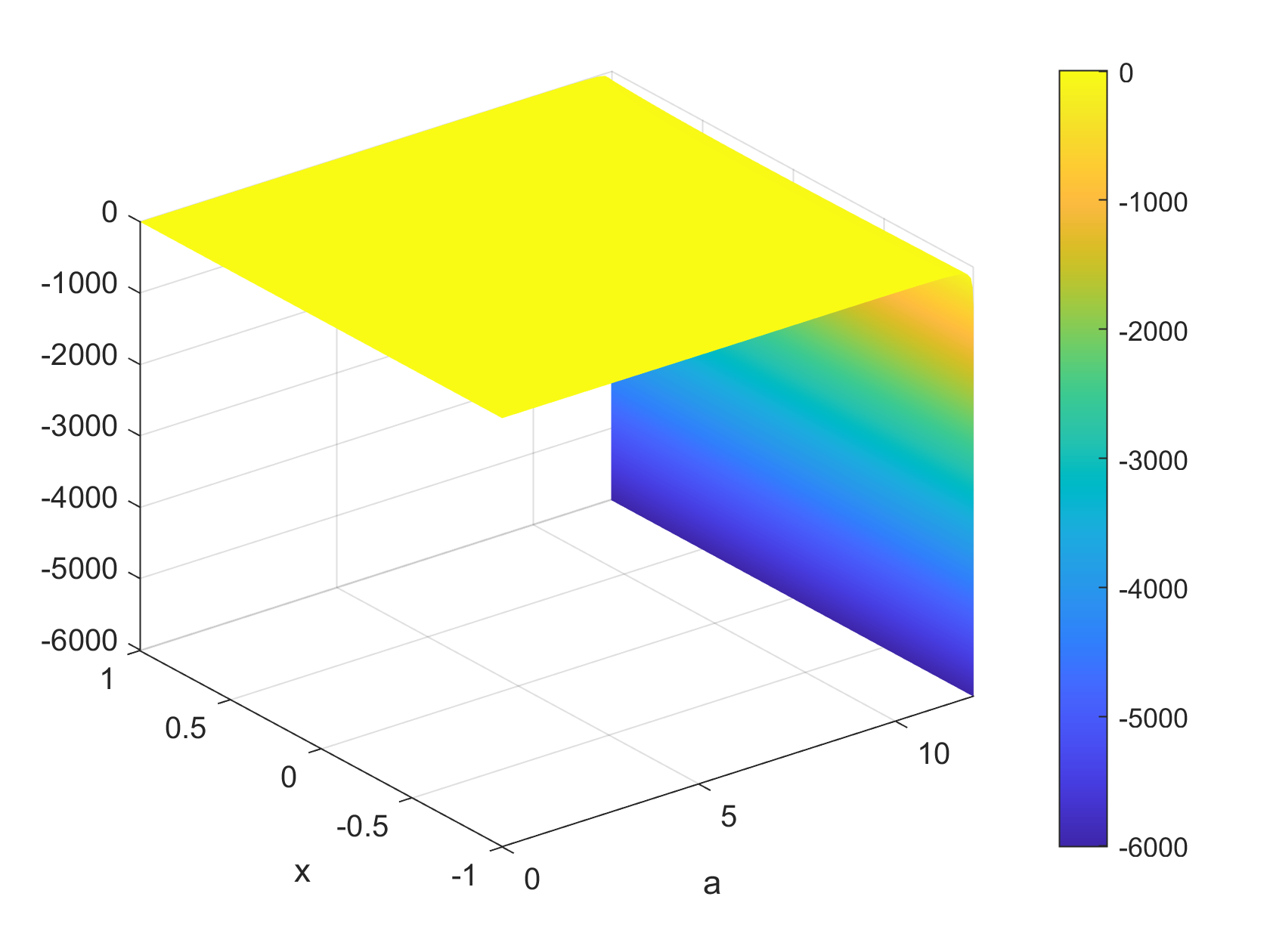}
\par\end{centering}
}\subfloat[$v_{0,N}^{\text{true}}$ ($N=6$)]{\begin{centering}
\includegraphics[scale=0.115]{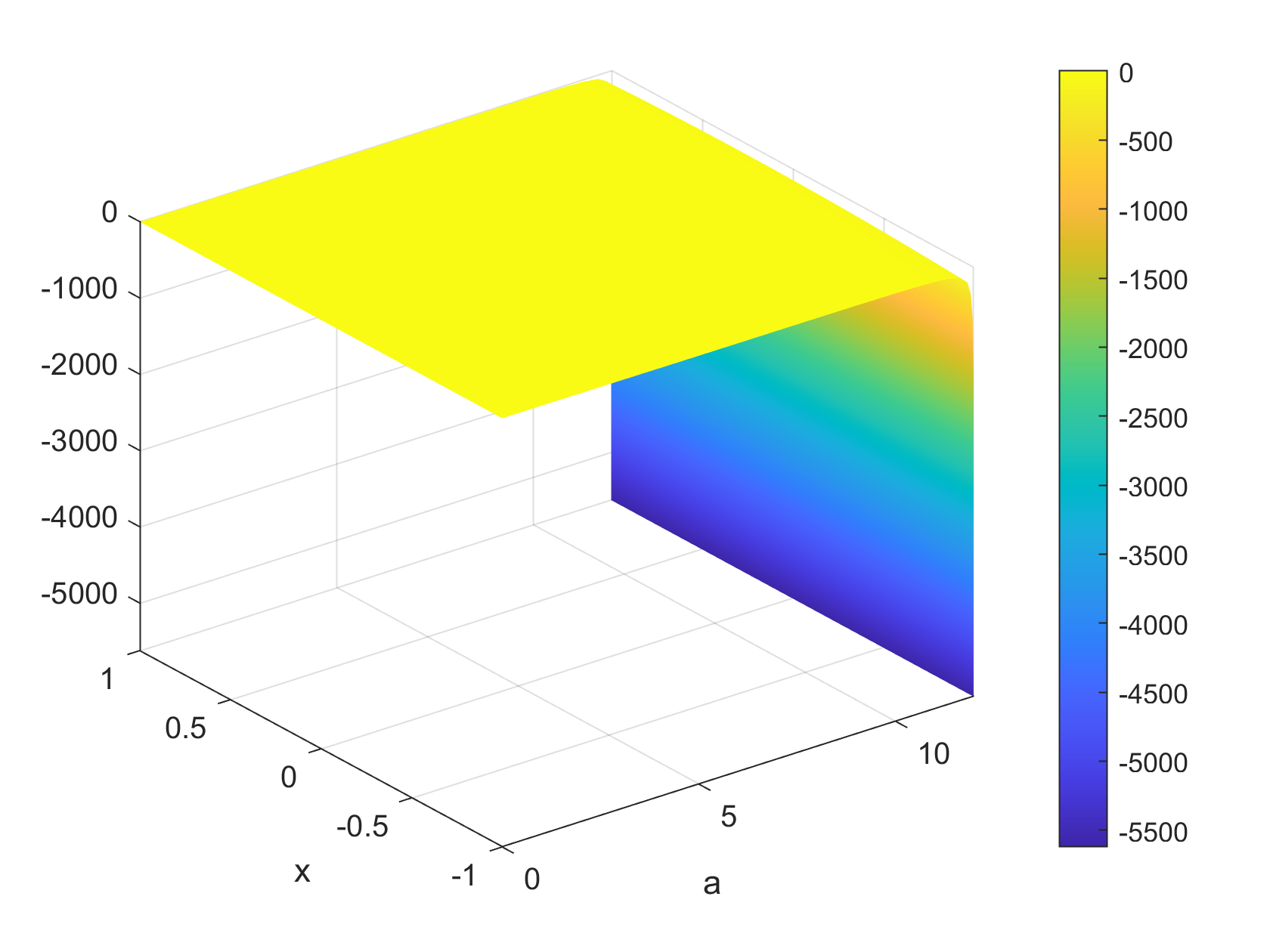}
\par\end{centering}
}
\par\end{centering}
\caption{Graphical illustrations of $v_{0}^{\text{true}}$ and $v_{0,N}^{\text{true}}$
for $N=2,6$ for Example 1 (row 1), Example 2 (row 2), and Example
3 (row 3). The mesh grid of space and age are fixed with $\Delta x=\Delta t=0.05$.
By these graphical observations, $N=6$ would be the best choice for
all numerical experiments because of its very high accuracy and fast-computing.\label{fig:1}}

\end{figure}
\par\end{center}

\subsection{Example 1: Semiellipse initial profiles with youngster's immobility}

We begin by a numerical example with the semiellipse-shaped initial
profiles near the boundary $a=0$. In particular, for $\varepsilon=0.75$,
we choose
\begin{align}
u_{0}\left(a,x\right) & =\frac{1}{\sqrt{2\pi}\varepsilon}\exp\left(-\frac{1}{2\varepsilon^{2}}\left(x^{2}+\left(a-0.15\right)^{2}\right)\right),\label{eq:4.1}\\
\overline{u}_{0}\left(t,x\right) & =\frac{1}{\sqrt{2\pi}\varepsilon}\exp\left(-\frac{1}{2\varepsilon^{2}}\left(x^{2}+\left(t-0.15\right)^{2}\right)\right),\label{eq:4.2}
\end{align}
which satisfy the compatibility condition $u_{0}\left(0,\cdot\right)=\overline{u}_{0}\left(0,\cdot\right)$;
see Figure \ref{fig:Ex1-1} for the graphical illustrations of the
initial profile $u_{0}$. Besides, we choose $\rho=0.5$ for the net
proliferation rate, and the diffusion term is chosen as 

\[
D\left(t,a\right)=0.03-0.03\exp\left(-\frac{\left(\frac{a_{\dagger}}{8}-a\right)^{2}}{a}\right),
\]
indicating that ``youngster'' individuals are less mobile than newborn
and old.

In our numerical experiments, we do not know the exact solutions,
but we can generate data to evaluate the accuracy of the approximation.
In this sense, numerical convergence can be assessed by examining
numerical stability and thus, by the behavior of computed solutions
at different time points and for various values of $M$. Figure \ref{fig:Ex1-2}
shows the computed solutions for $M=200,400,800$ and at different
time points $t=2.5,5.0,7.5,10.0$. We can see that the values of the
approximate solutions become more stable as $M$ increases. Moreover,
the approximate solutions behave similarly at every time observation,
as shown in the first row of \ref{fig:Ex1-2}.

To assess better the numerical stability, we plot the total population
as a function of time for different values of $M=100,200,400,800$
in Figure \ref{fig:total1-1}. The total population is formulated
in Remark \ref{rem:Total}, as noted above. Figure \ref{fig:total1-1}
reveals the entire time evolution of tumor cells, with the total density
significantly increasing within the first quarter of the time frame.
Starting from a small amount at $t=0$ (see values of $u_{0}$ in
Figure \ref{fig:u03d}), the total density peaks at about $t=2.1$,
reaching 18 million cells/cm. Afterward, the total density gradually
tends towards extinction within the next six months.

In addition, the simulation in Figure \ref{fig:Ex3-2} shows that
tumor cells are slightly localized along the direction of $a$ between
the overpopulation and extinction phases.
\begin{rem*}
It is worth noting that since the final time observation $T$ and
maximal age $a_{\dagger}$ are large, we may need huge values of $M$
to get better approximation. In fact, by the explicit numerical scheme
being used, the local truncation error is of the order one. 
\end{rem*}
\begin{center}
\begin{figure}
\begin{centering}
\subfloat[$u_{0}$ (3D)\label{fig:u03d}]{\begin{centering}
\includegraphics[scale=0.2]{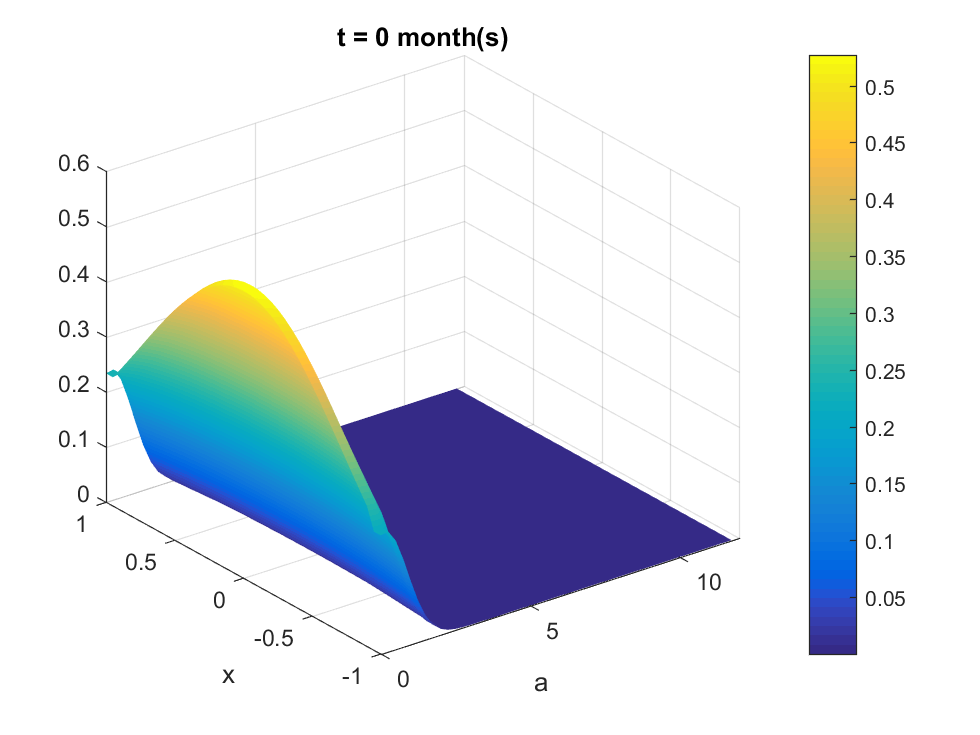}
\par\end{centering}
}\subfloat[$u_{0}$ (2D)]{\begin{centering}
\includegraphics[scale=0.2]{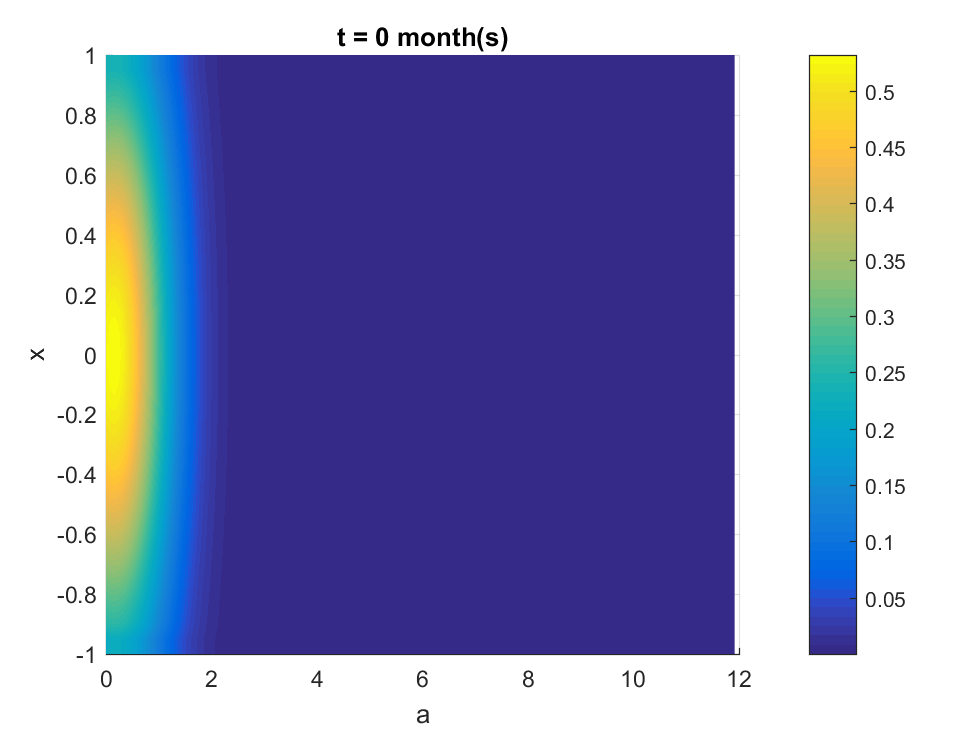}
\par\end{centering}
}\subfloat[$p\left(t\right)$\label{fig:total1-1}]{\begin{centering}
\includegraphics[scale=0.2]{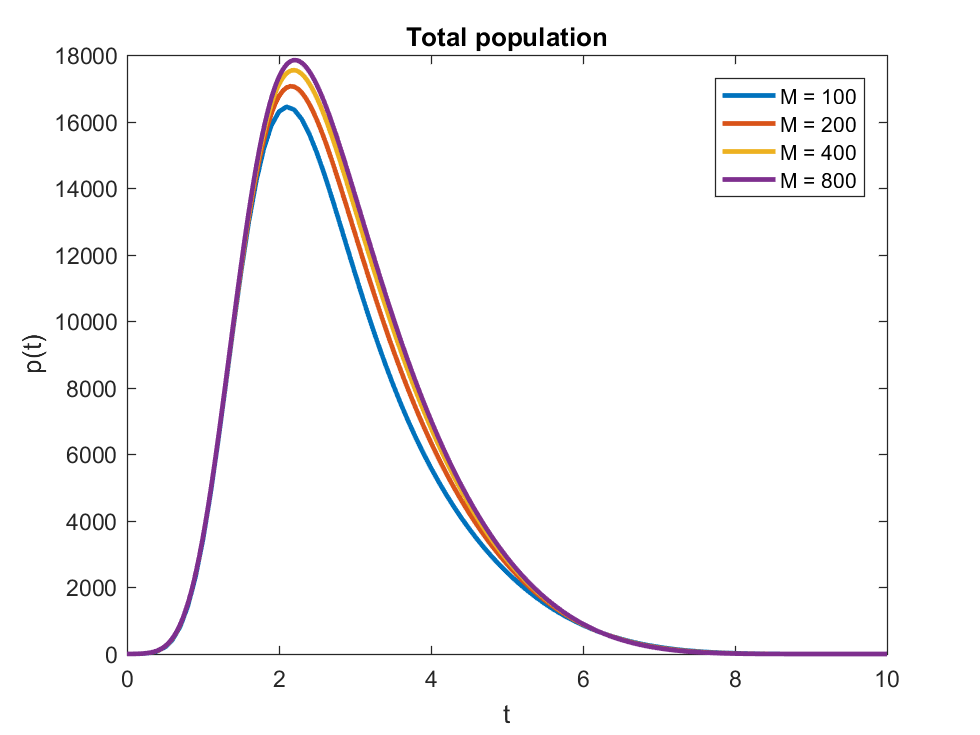}
\par\end{centering}
}
\par\end{centering}
\caption{Left: 3D representation of the initial data $u_{0}$ in Example 1.
Middle: 2D representation of the initial data $u_{0}$ in Example
1. Right: Total population of tumor cells $p\left(t\right)$ for varying
values of $M$.\label{fig:Ex1-1}}

\end{figure}
\par\end{center}

\begin{center}
\begin{figure}
\begin{centering}
\subfloat[$M=200$]{\begin{centering}
\includegraphics[scale=0.2]{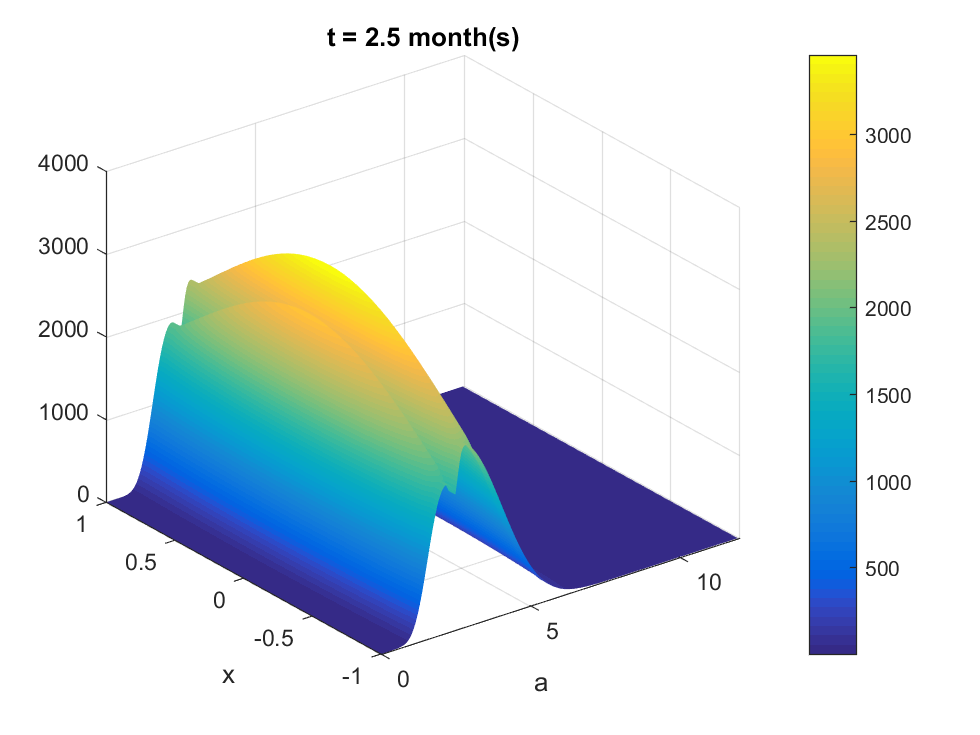}
\par\end{centering}
}\subfloat[$M=400$]{\begin{centering}
\includegraphics[scale=0.2]{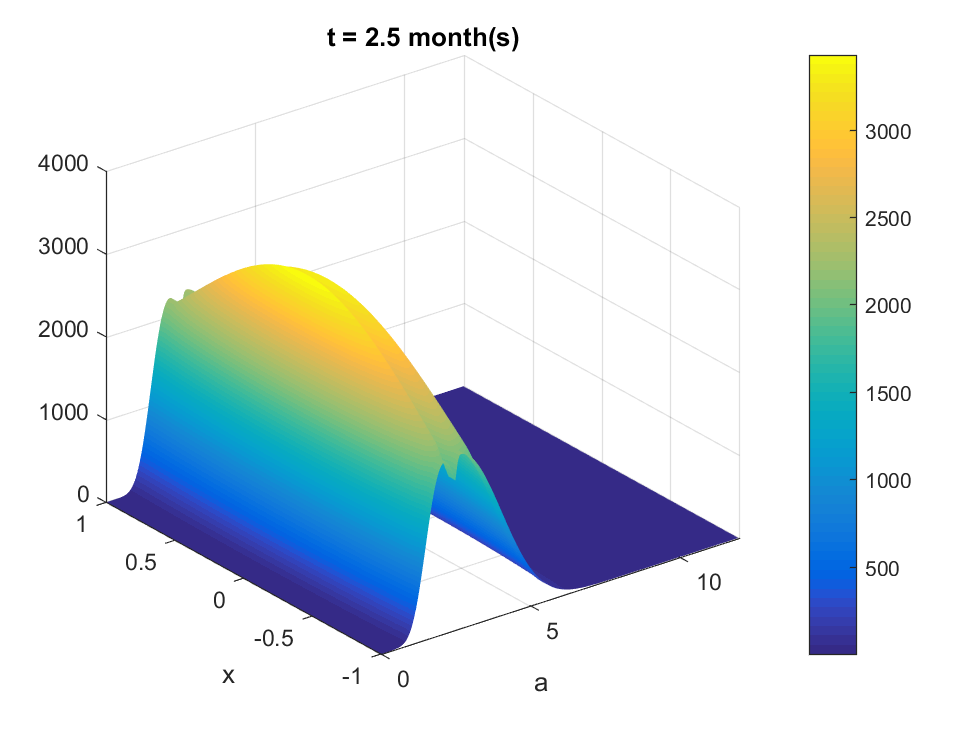}
\par\end{centering}
}\subfloat[$M=800$]{\begin{centering}
\includegraphics[scale=0.2]{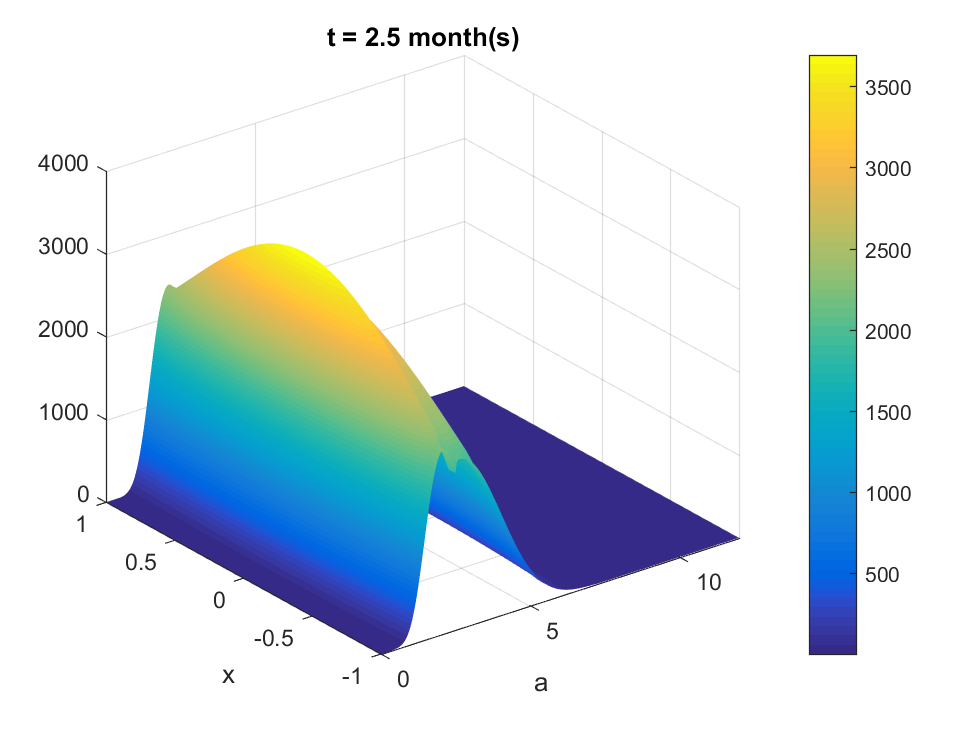}
\par\end{centering}
}
\par\end{centering}
\begin{centering}
\subfloat[$M=200$]{\begin{centering}
\includegraphics[scale=0.2]{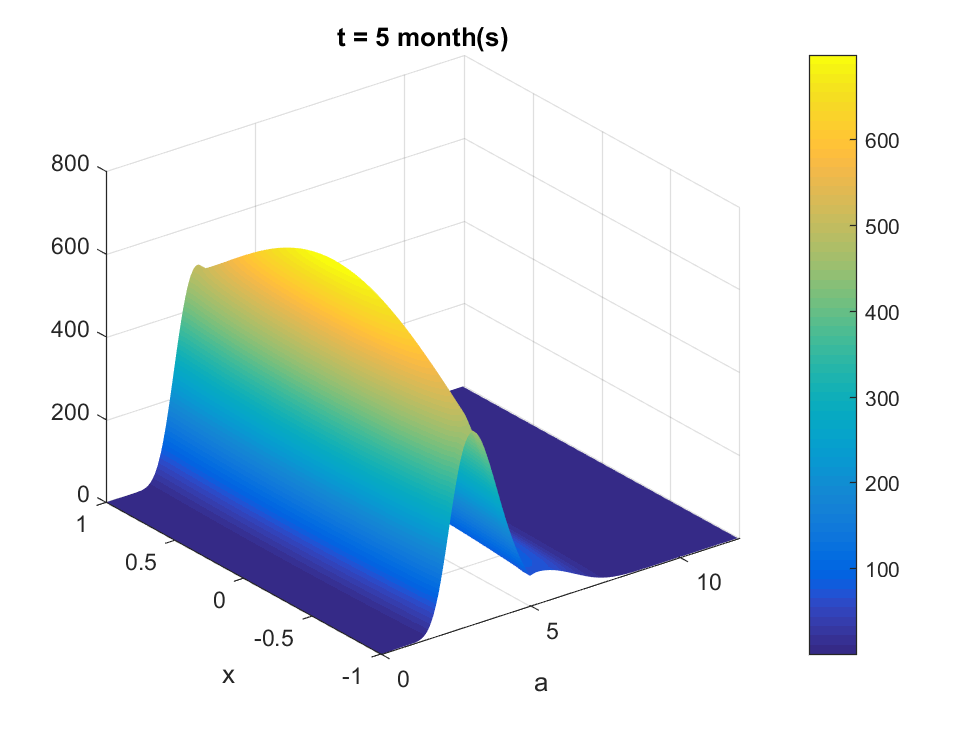}
\par\end{centering}
}\subfloat[$M=400$]{\begin{centering}
\includegraphics[scale=0.2]{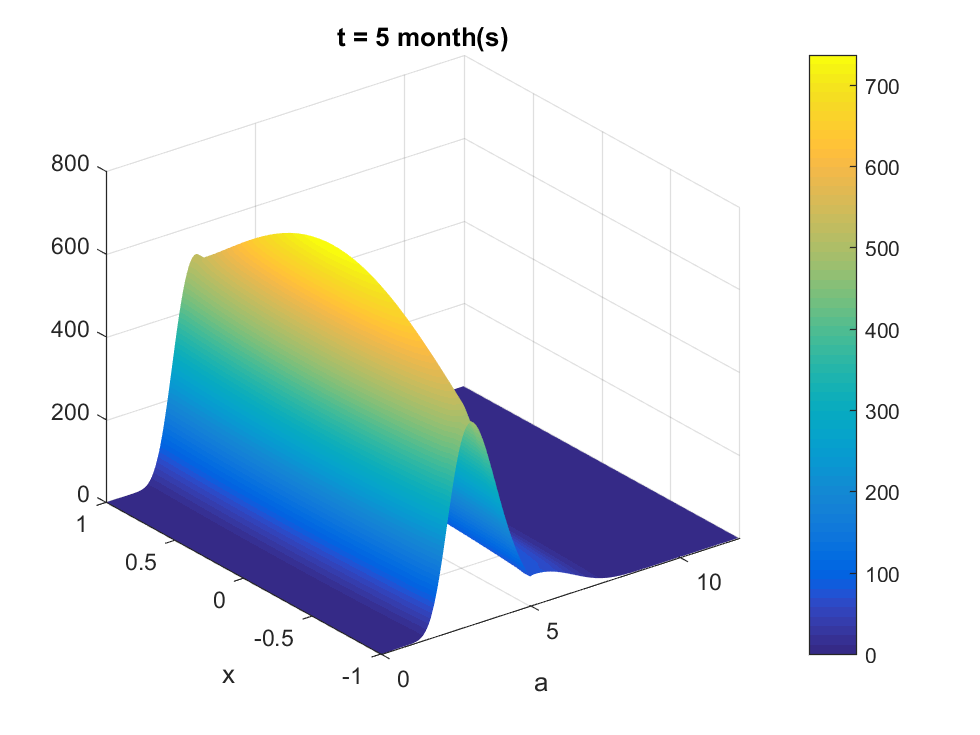}
\par\end{centering}
}\subfloat[$M=800$]{\begin{centering}
\includegraphics[scale=0.2]{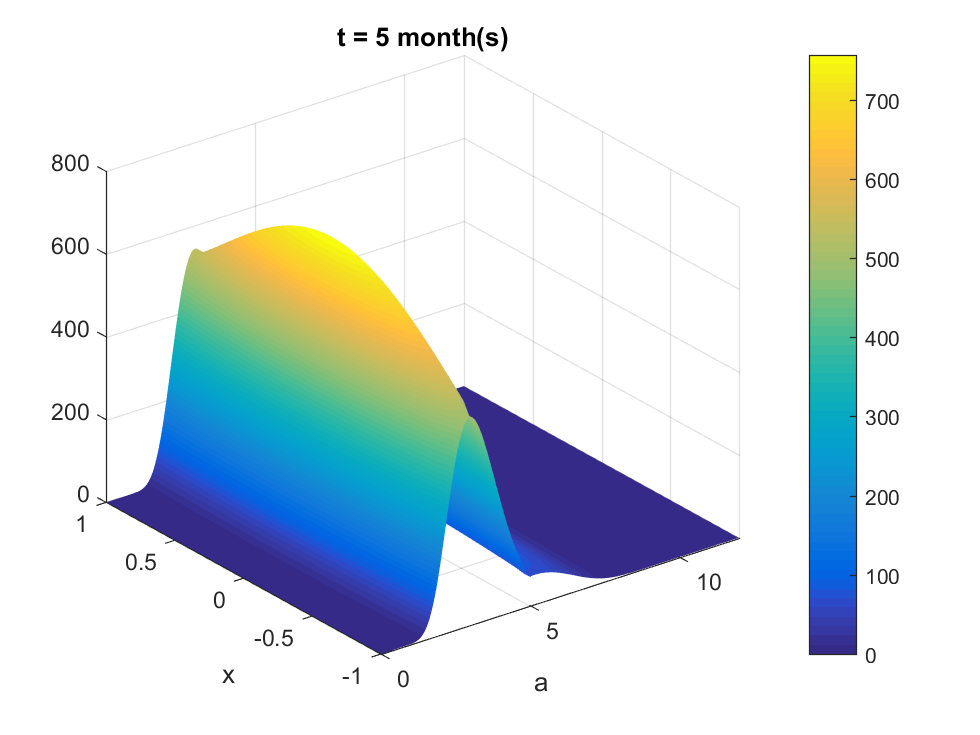}
\par\end{centering}
}
\par\end{centering}
\begin{centering}
\subfloat[$M=200$]{\begin{centering}
\includegraphics[scale=0.2]{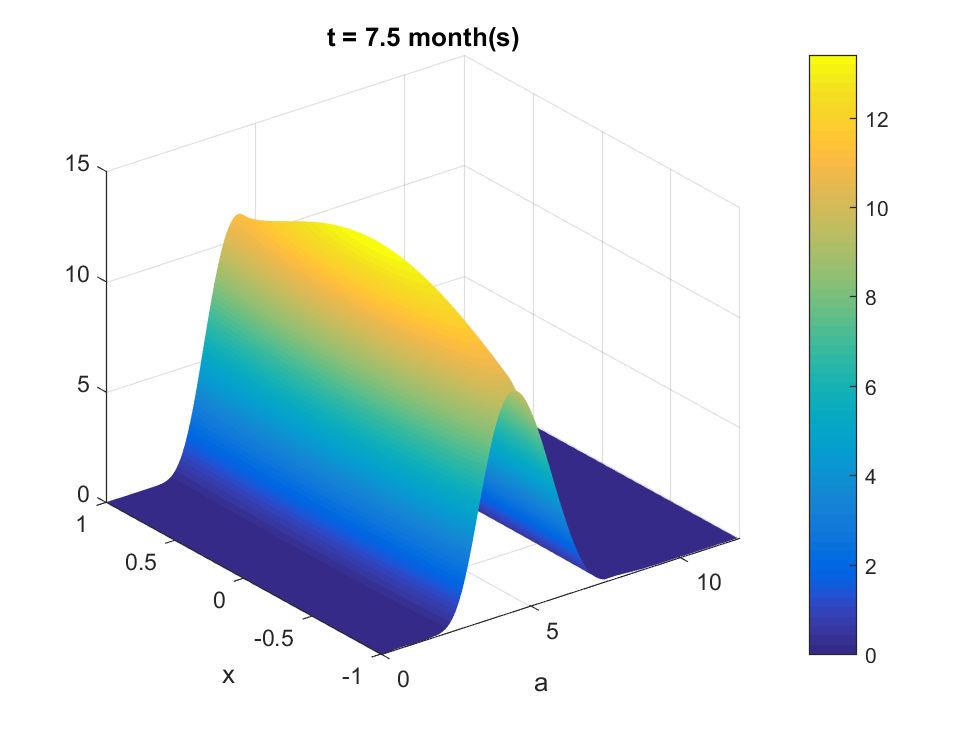}
\par\end{centering}
}\subfloat[$M=400$]{\begin{centering}
\includegraphics[scale=0.2]{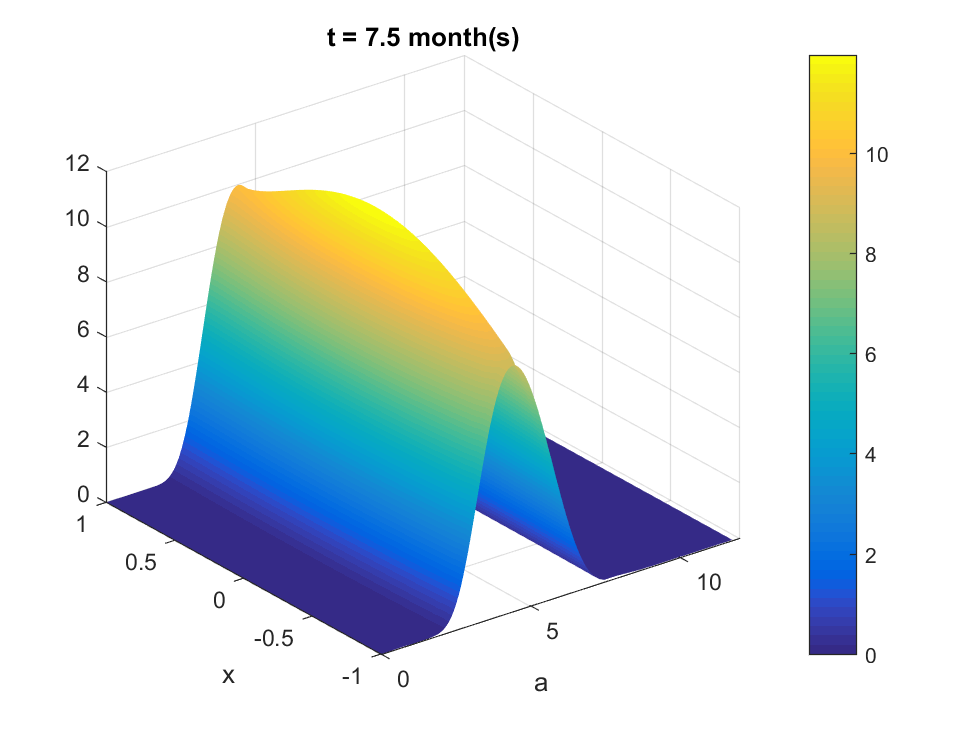}
\par\end{centering}
}\subfloat[$M=800$]{\begin{centering}
\includegraphics[scale=0.2]{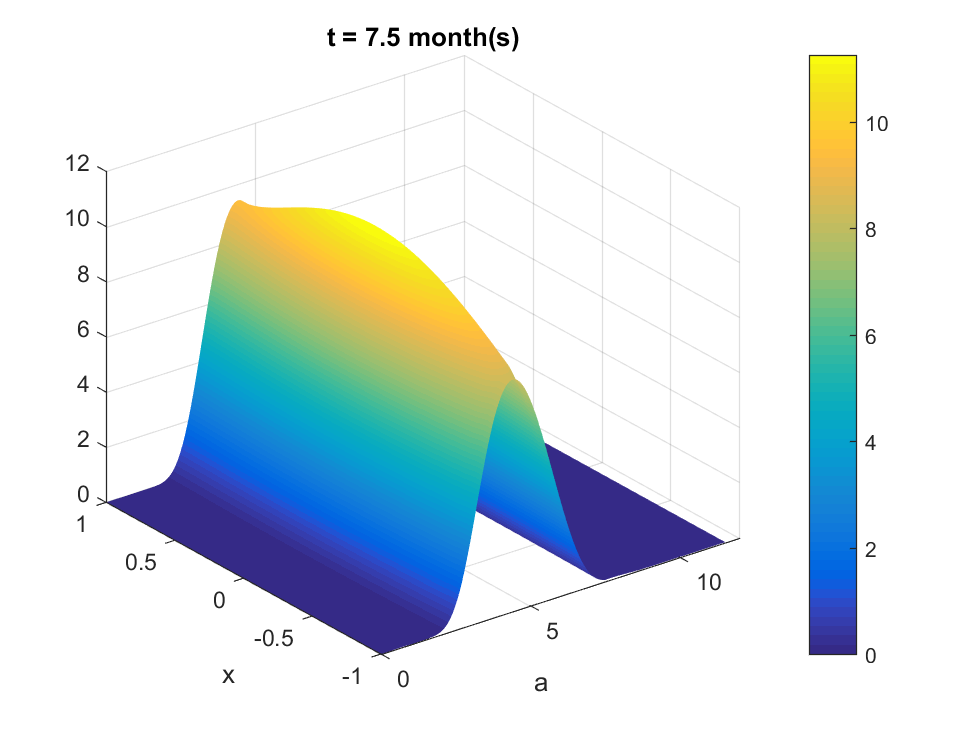}
\par\end{centering}
}
\par\end{centering}
\begin{centering}
\subfloat[$M=200$]{\begin{centering}
\includegraphics[scale=0.2]{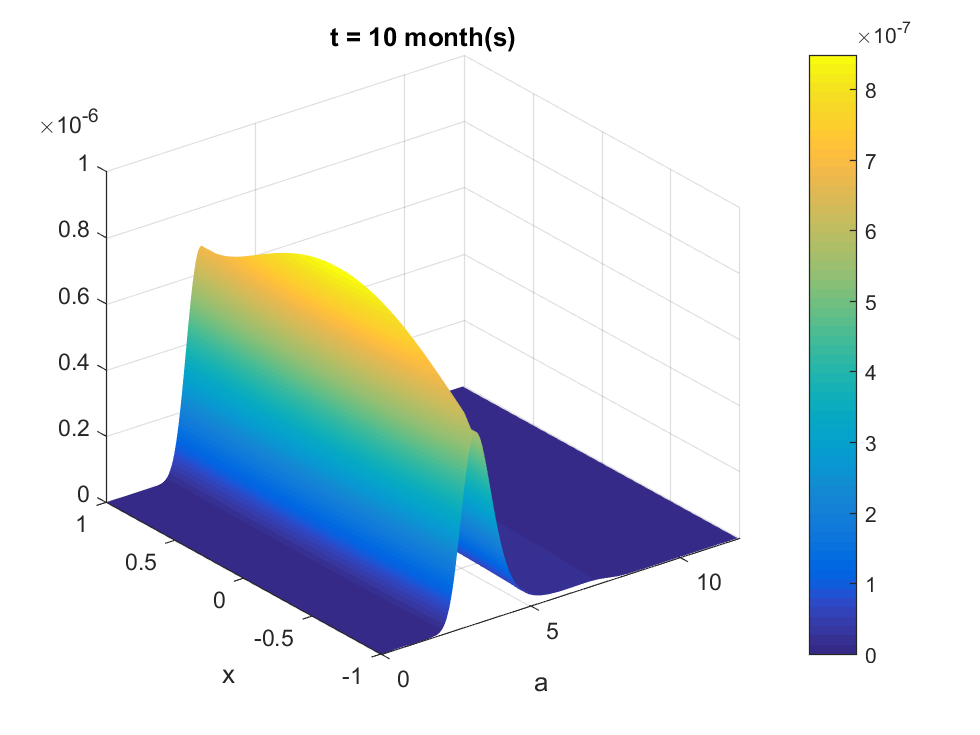}
\par\end{centering}
}\subfloat[$M=400$]{\begin{centering}
\includegraphics[scale=0.2]{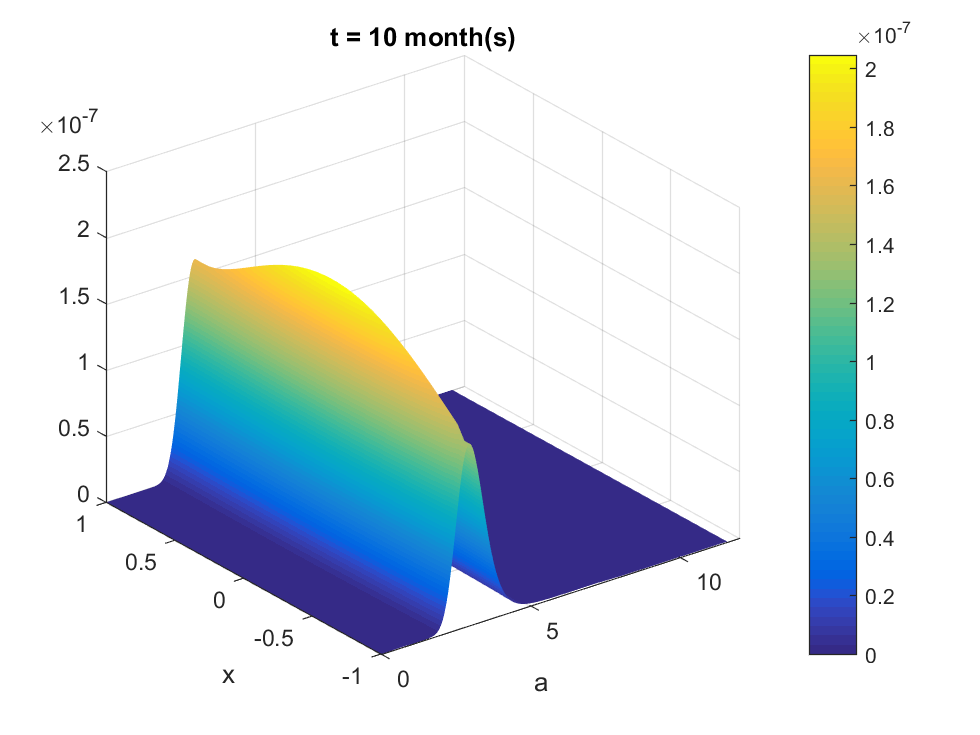}
\par\end{centering}
}\subfloat[$M=800$]{\begin{centering}
\includegraphics[scale=0.2]{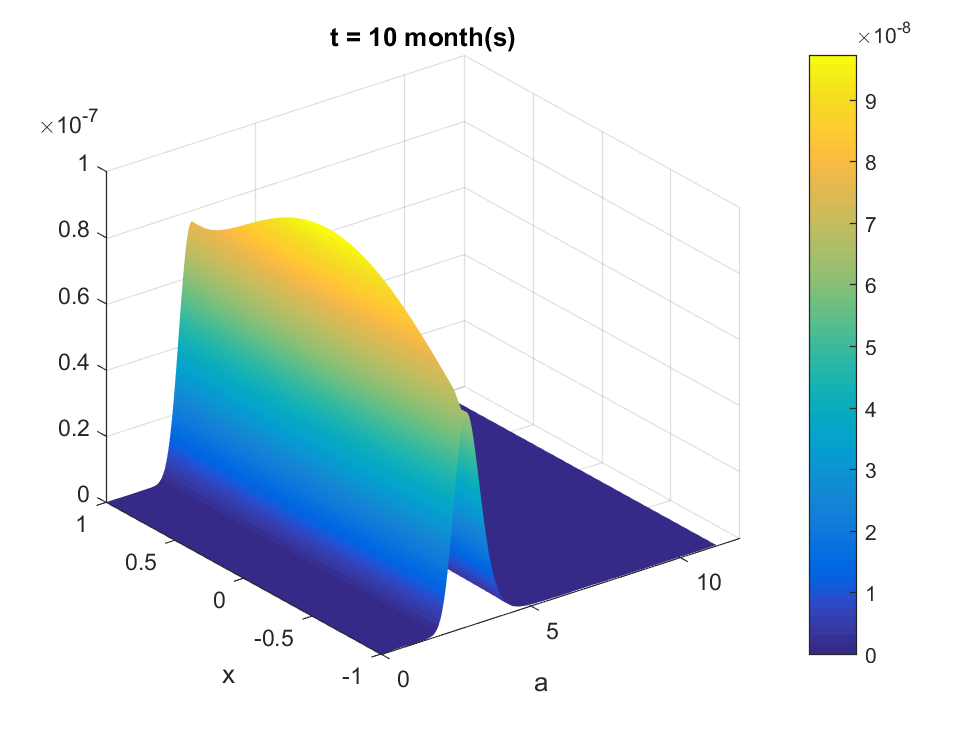}
\par\end{centering}
}
\par\end{centering}
\caption{Tumor cell density in Example 1 at $t=2.5,5.0,7.5,10.0$ for various
values of $M$. Column 1: $M=200,\Delta t=0.05$ (36 hours). Column
2: $M=400,\Delta t=0.025$ (18 hours). Column 3: $M=800,\Delta t=0.0125$
(9 hours).\label{fig:Ex1-2}}
\end{figure}
\par\end{center}

\subsection{Example 2: Gaussian initial profiles with elder's immobility}

In the second example, we take into account the Gaussian initial profiles
of the following form, for $\varepsilon=0.075$,
\[
u_{0}\left(a,x\right)=\frac{e^{-6x^{2}}}{\varepsilon+\cosh\left(a-7\right)},\quad\overline{u}_{0}\left(t,x\right)=\frac{e^{-6x^{2}}}{\varepsilon+\cosh\left(3t-7\right)}.
\]
This choice of $u_{0}$ and $\overline{u}_{0}$ satisfies the compatibility
condition $u_{0}\left(0,\cdot\right)=\overline{u}_{0}\left(0,\cdot\right)$;
however, compared to the choice in Example 1 (cf. (\ref{eq:4.1})
and (\ref{eq:4.2})), the $a$ and $t$ arguments here are not interchangeable.
Besides, we choose $\rho=7$ for a large net proliferation rate, and
the diffusion term is chosen as 
\[
D\left(t,a\right)=\exp\left(-\frac{\left(t-8T\right)^{2}}{T}\right)\left(a_{\dagger}-a\right),
\]
indicating that ``old'' individuals are very less mobile.

Our numerical results for Example 2 are presented in Figures \ref{fig:total2}
and \ref{fig:Ex2-2}. Figure \ref{fig:total2} illustrates that the
total population of tumor cells exhibits minimal variation when $M$
is varied between 100 to 800. Similarly, Figure \ref{fig:Ex2-2} demonstrates
numerical stability in the distribution of tumor cells for different
values of $M$. In contrast to Example 1, these graphs are highly
similar when $M$ is varied from 200 to 800, indicating that a too
large $M$ is unnecessary for simulating this example.

Our simulation shows that tumor cells reach a stable state quickly,
plateauing at around 2.8 thousand cells/cm across all time points
(as seen in Figure \ref{fig:Ex2-2}). The total population, shown
in Figure \ref{fig:total2}, reaches its peak of 60 thousand cells/cm
within three (3) months from an initial total population of about
3 thousand cells. Subsequently, the total population appears to stabilize
but gradually declines over time.
\begin{center}
\begin{figure}
\begin{centering}
\subfloat[$u_{0}$ (3D)\label{fig:u03d-1}]{\begin{centering}
\includegraphics[scale=0.2]{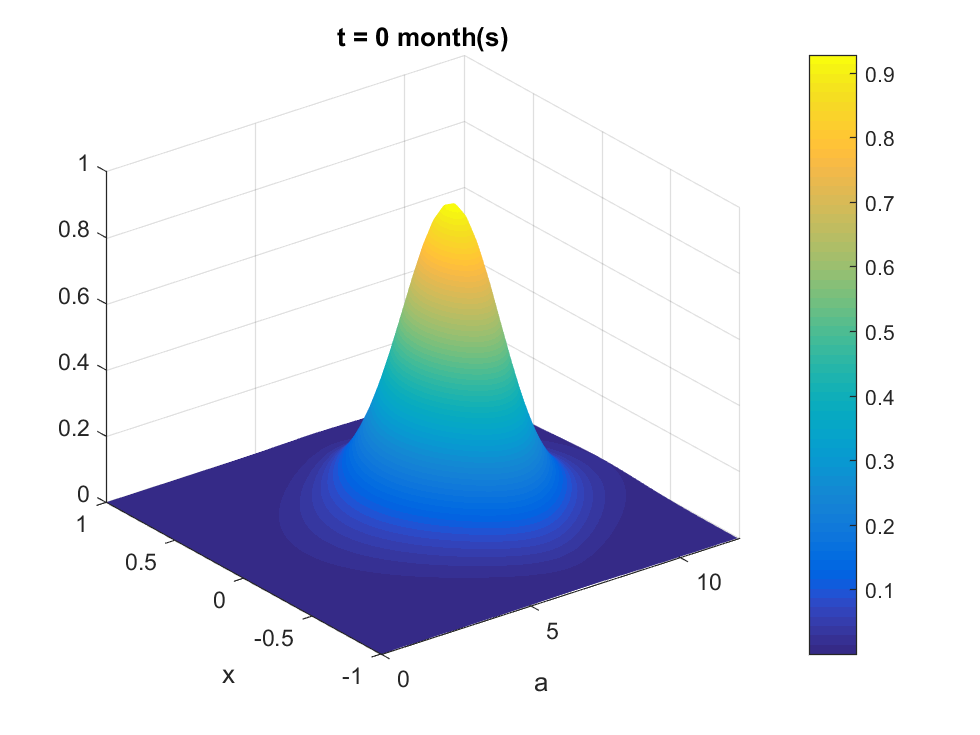}
\par\end{centering}
}\subfloat[$u_{0}$ (2D)]{\begin{centering}
\includegraphics[scale=0.2]{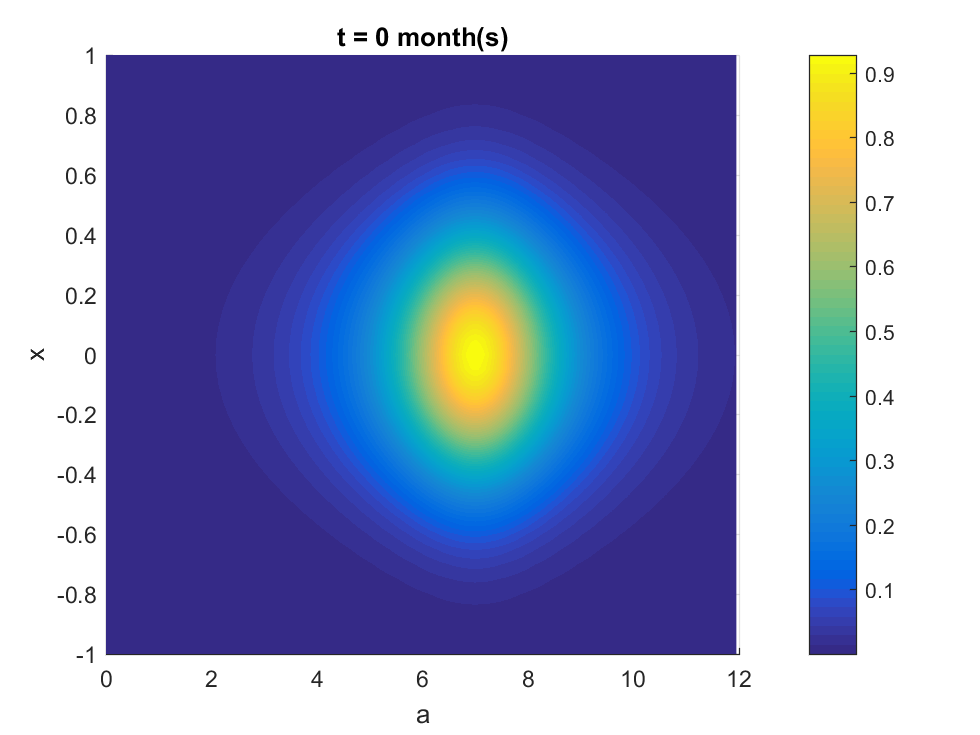}
\par\end{centering}
}\subfloat[$p\left(t\right)$\label{fig:total2}]{\begin{centering}
\includegraphics[scale=0.2]{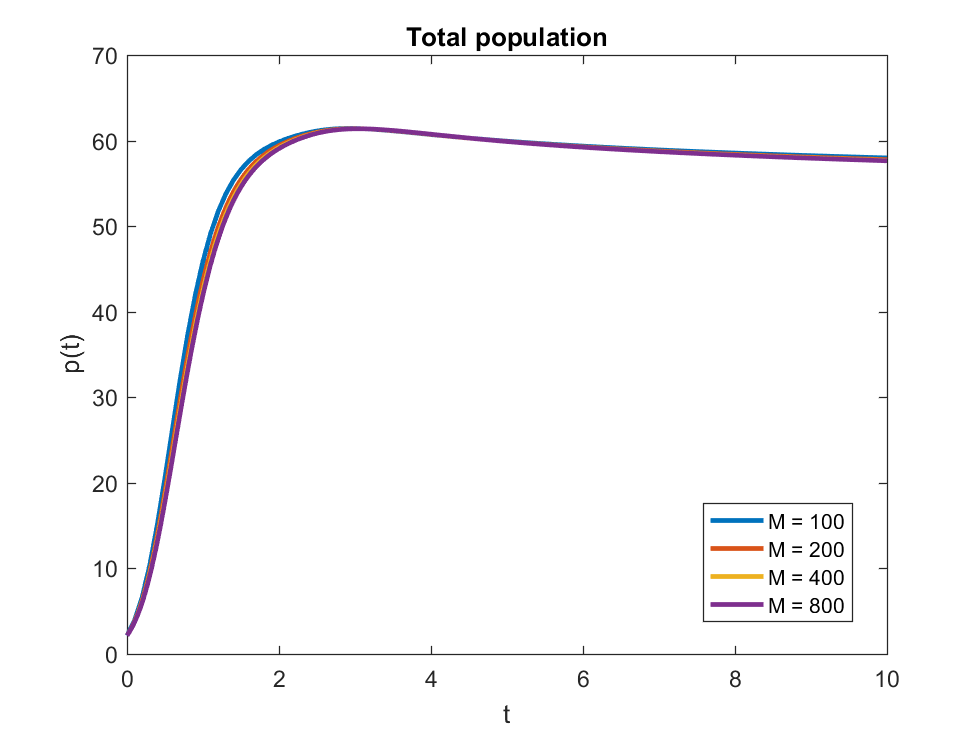}
\par\end{centering}
}
\par\end{centering}
\caption{Left: 3D representation of the initial data $u_{0}$ in Example 2.
Middle: 2D representation of the initial data $u_{0}$ in Example
2. Right: Total population of tumor cells $p\left(t\right)$ for varying
values of $M$.\label{fig:Ex2-1}}
\end{figure}
\par\end{center}

\begin{center}
\begin{figure}
\begin{centering}
\subfloat[$M=200$]{\begin{centering}
\includegraphics[scale=0.115]{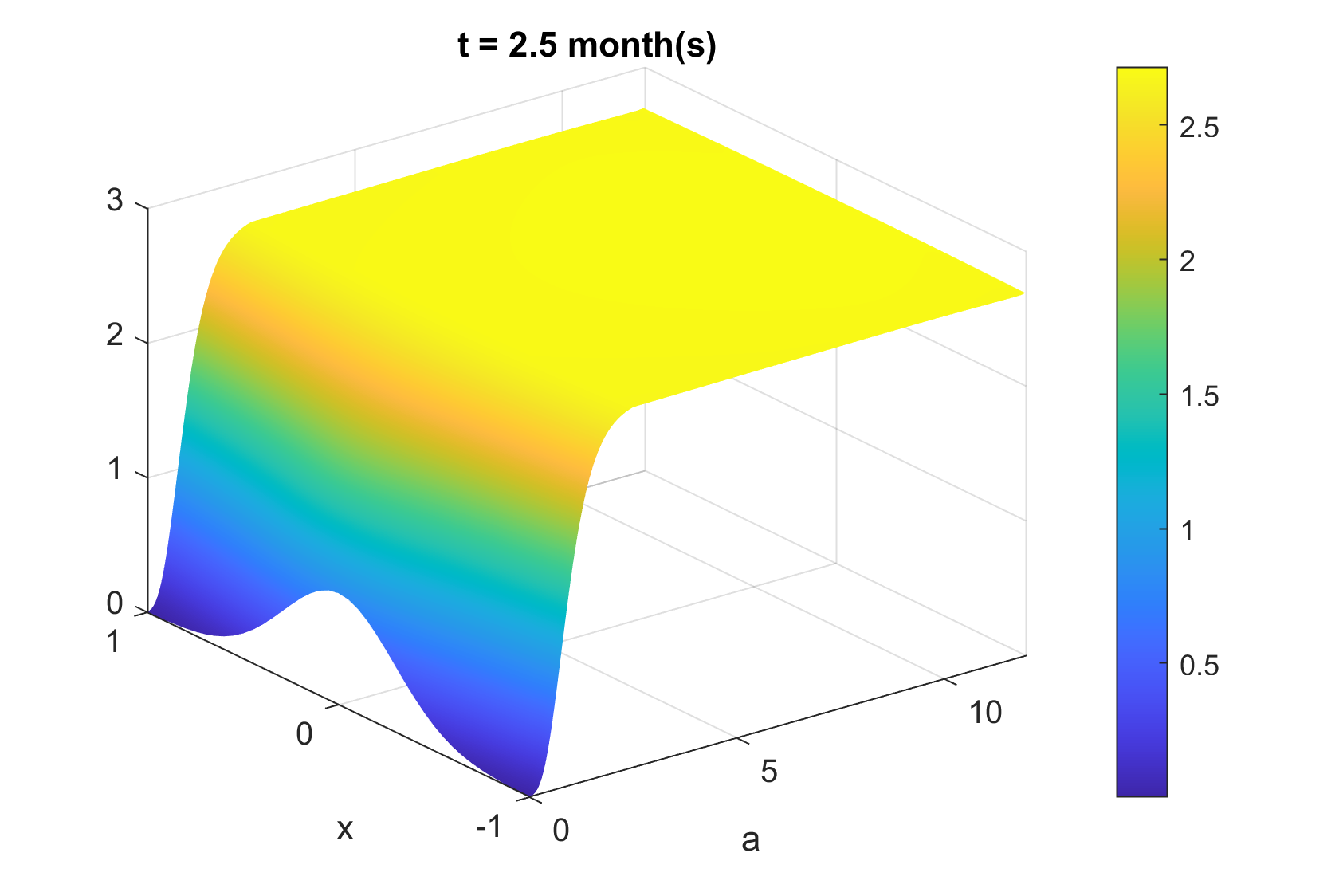}
\par\end{centering}
}\subfloat[$M=400$]{\begin{centering}
\includegraphics[scale=0.115]{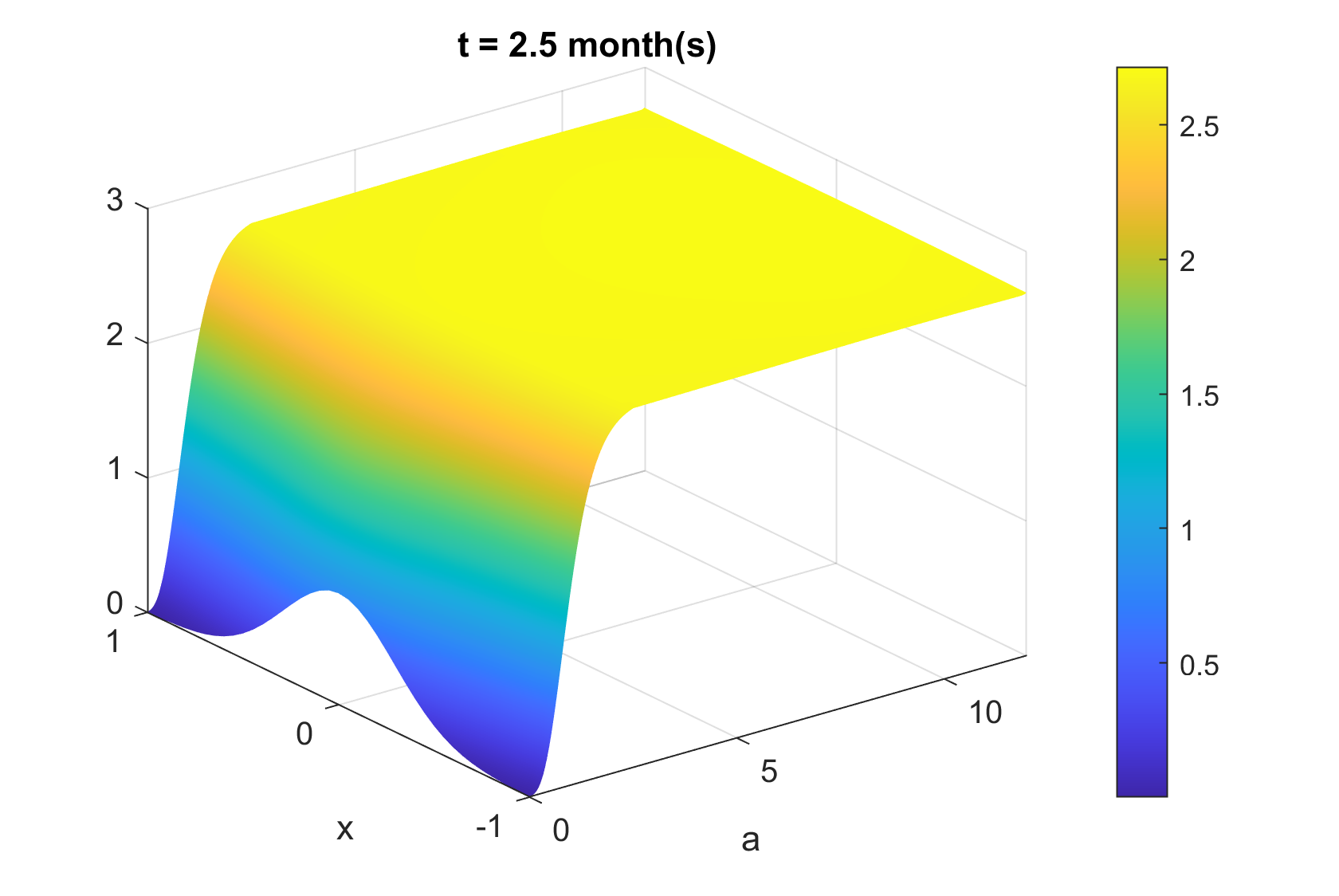}
\par\end{centering}
}\subfloat[$M=800$]{\begin{centering}
\includegraphics[scale=0.115]{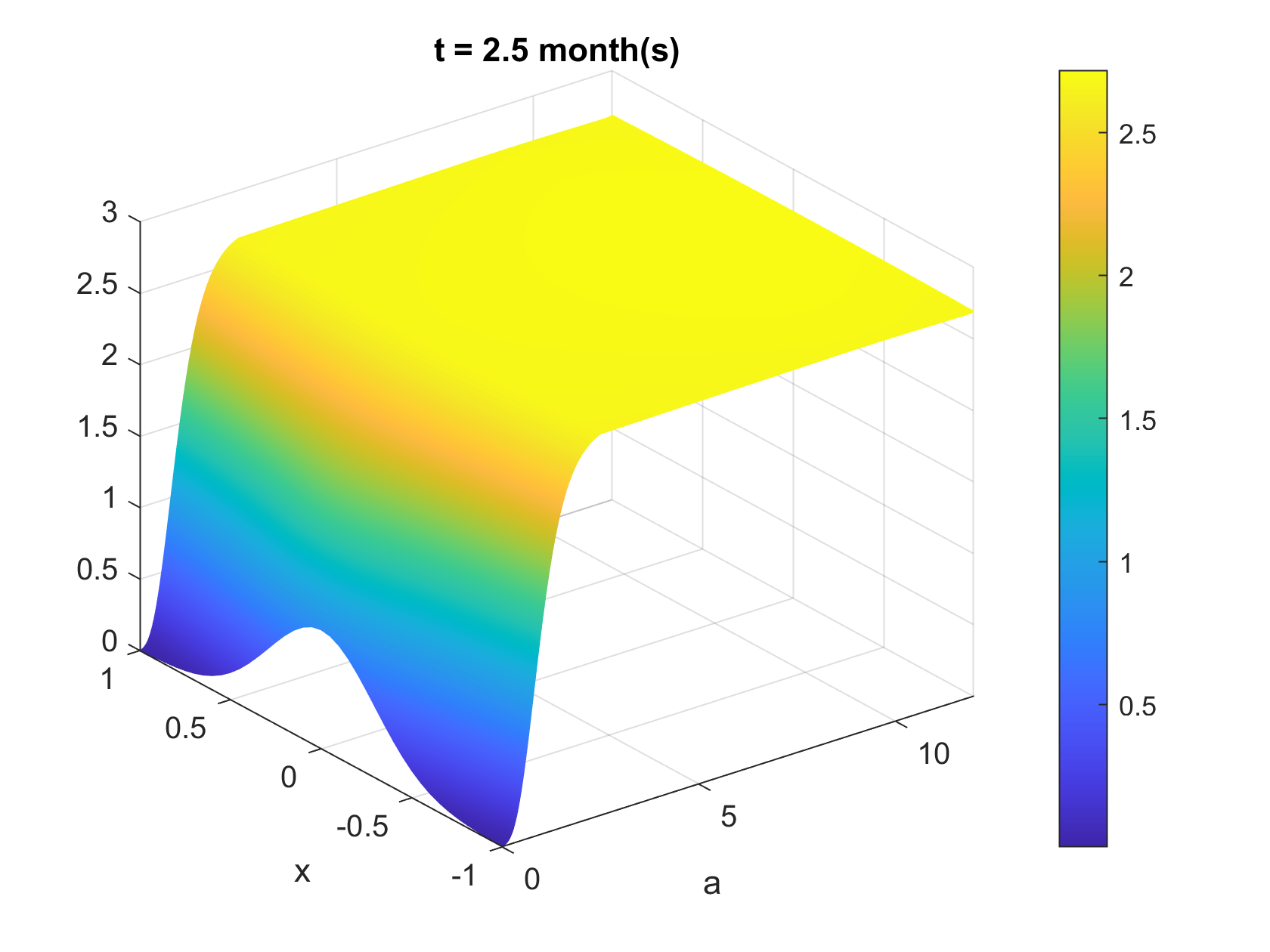}
\par\end{centering}
}
\par\end{centering}
\begin{centering}
\subfloat[$M=200$]{\begin{centering}
\includegraphics[scale=0.115]{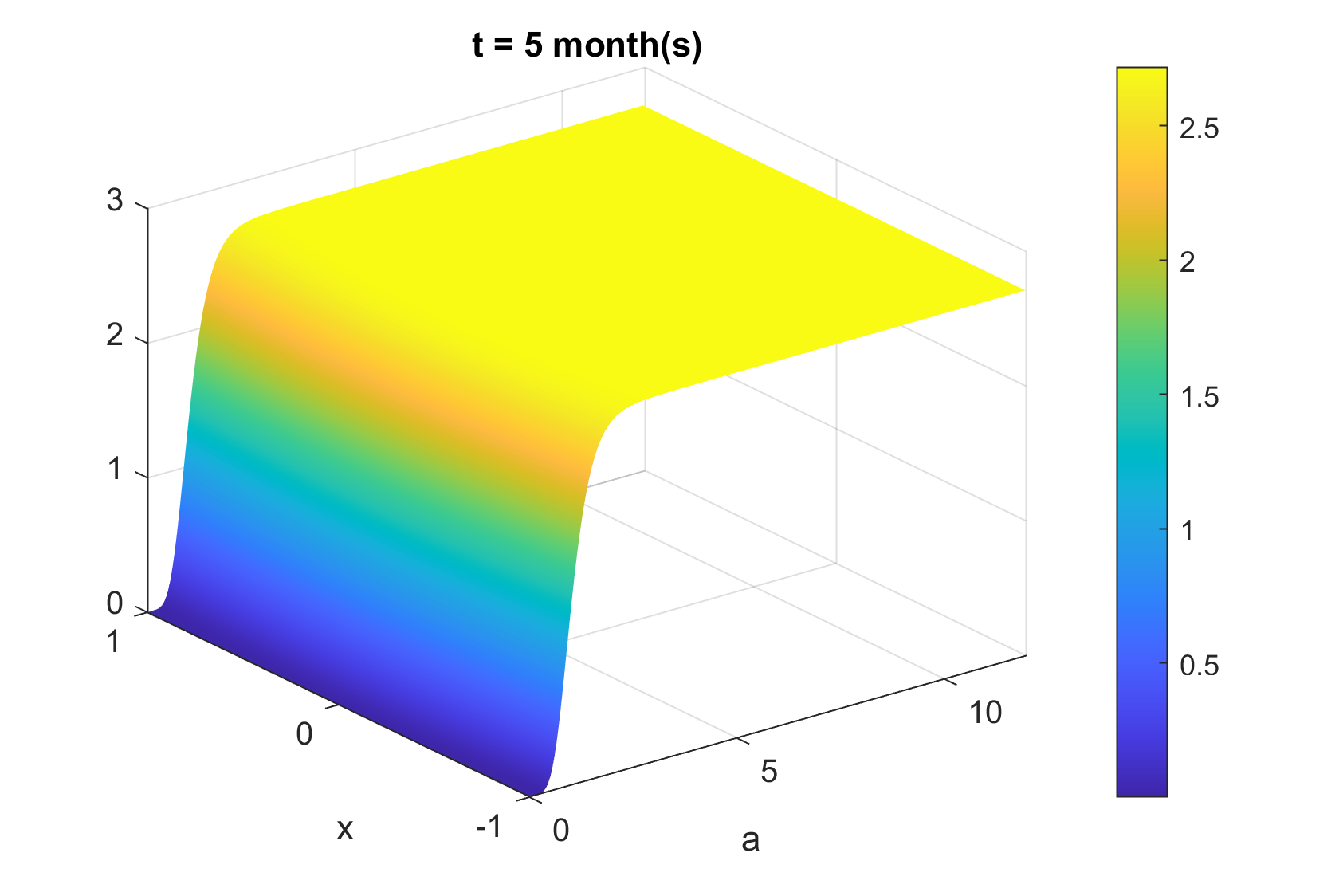}
\par\end{centering}
}\subfloat[$M=400$]{\begin{centering}
\includegraphics[scale=0.115]{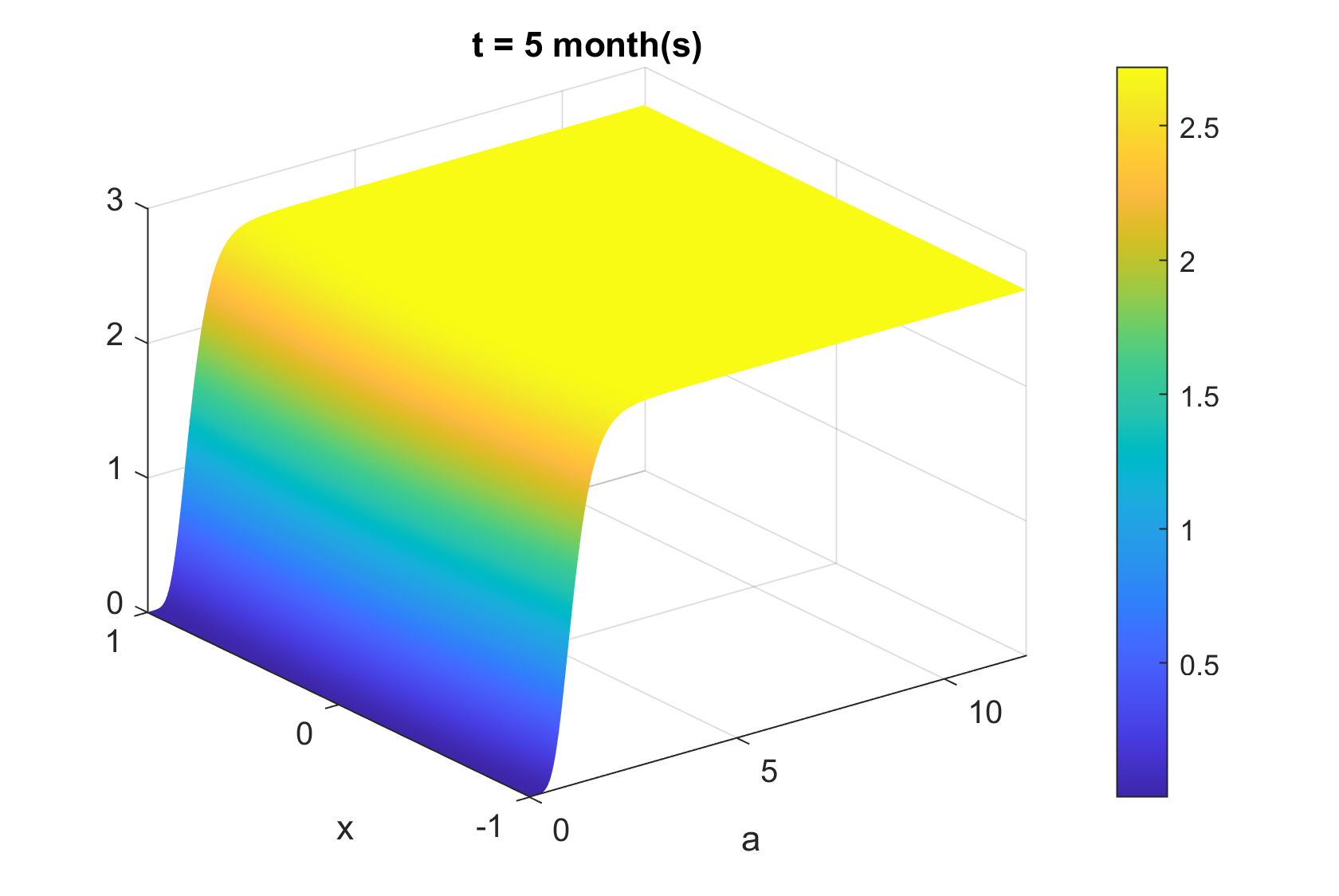}
\par\end{centering}
}\subfloat[$M=800$]{\begin{centering}
\includegraphics[scale=0.115]{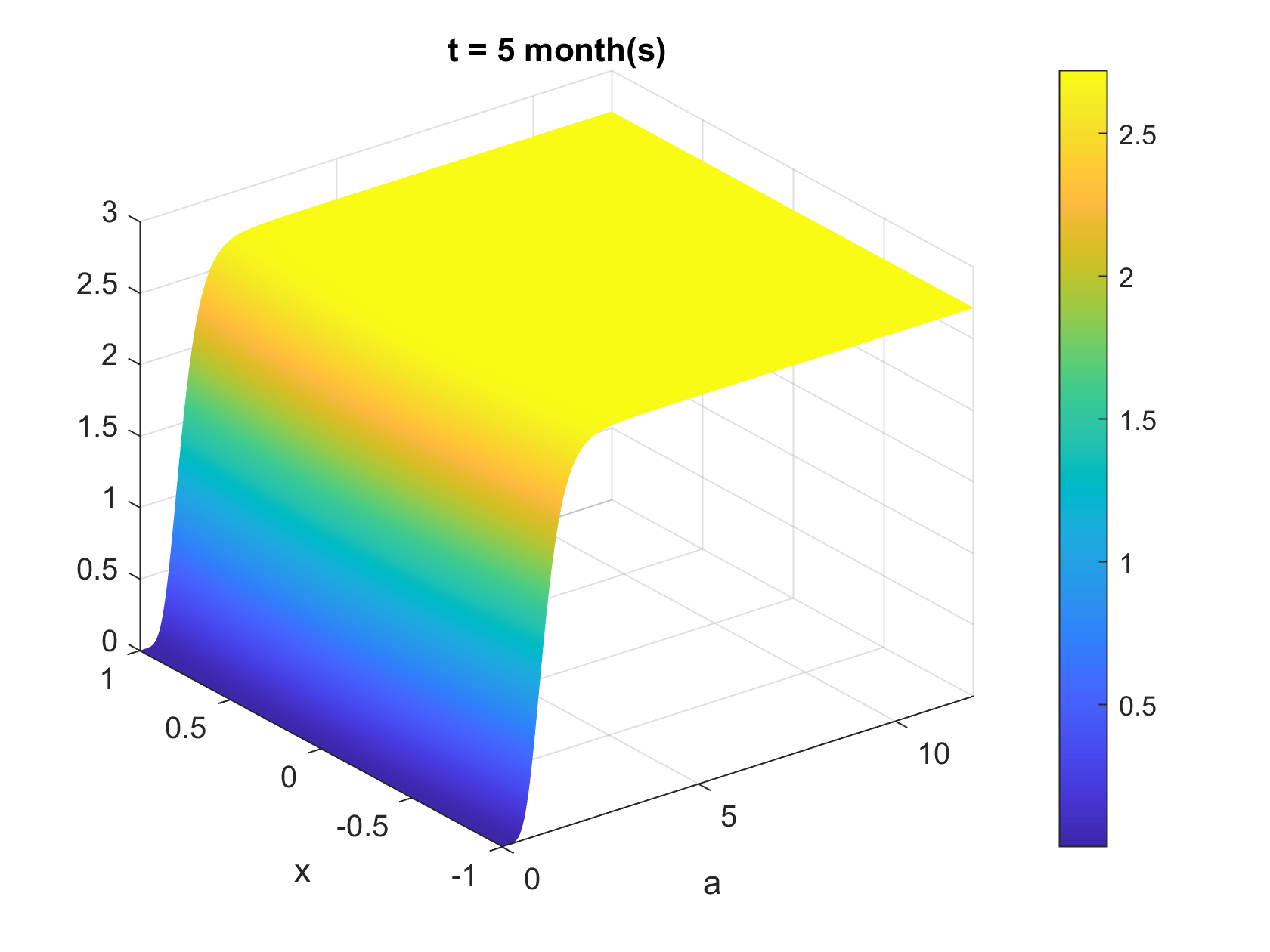}
\par\end{centering}
}
\par\end{centering}
\begin{centering}
\subfloat[$M=200$]{\begin{centering}
\includegraphics[scale=0.115]{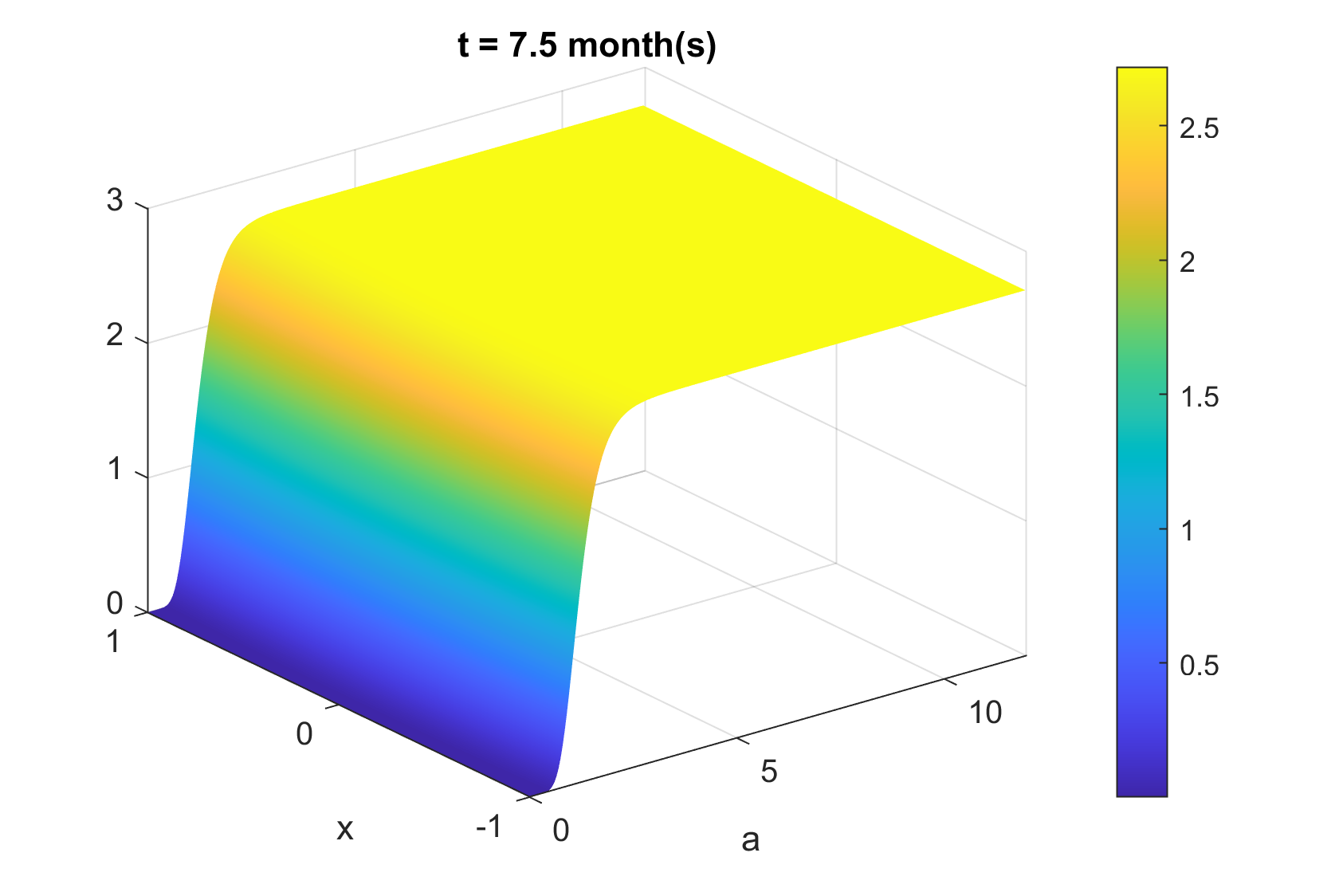}
\par\end{centering}
}\subfloat[$M=400$]{\begin{centering}
\includegraphics[scale=0.115]{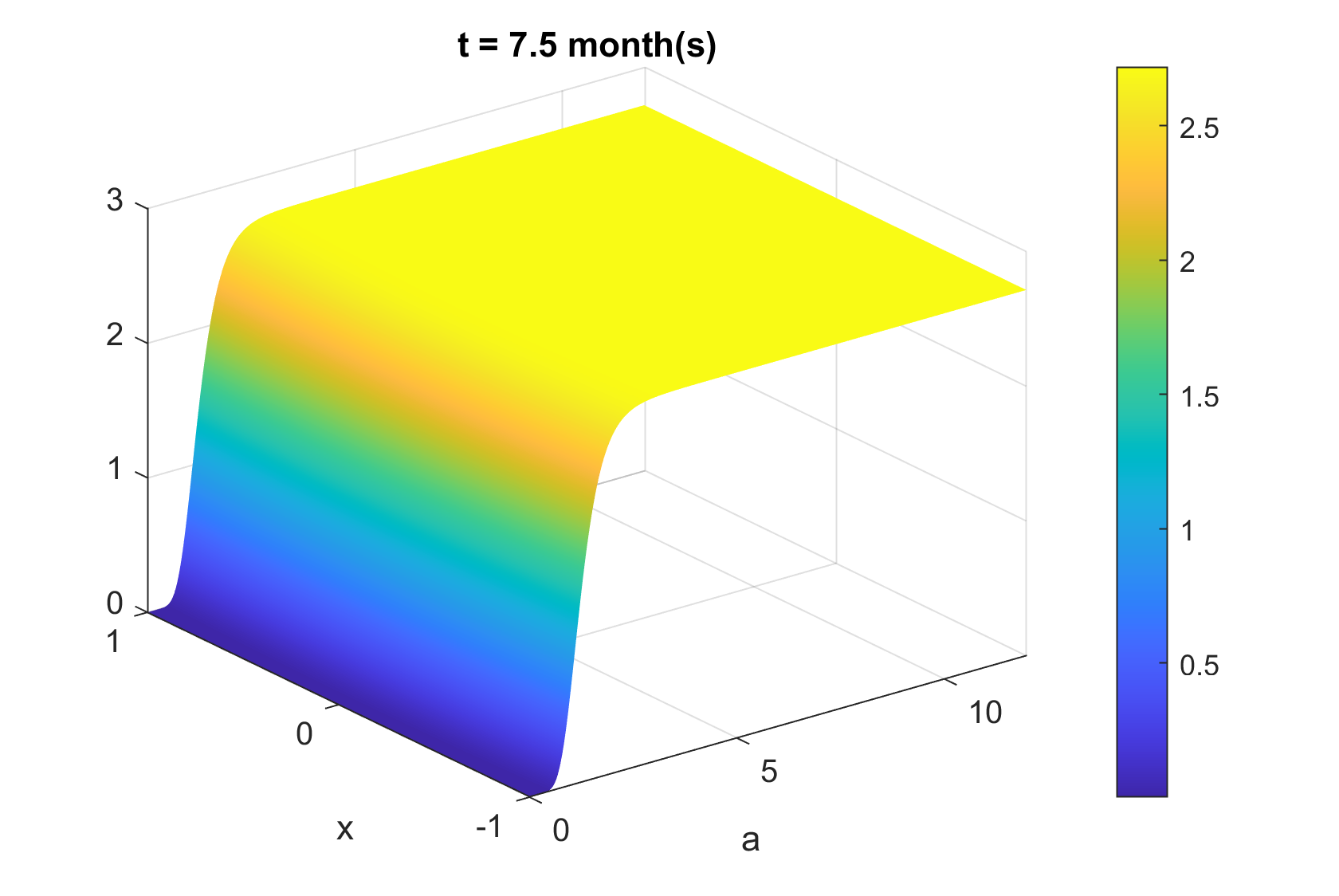}
\par\end{centering}
}\subfloat[$M=800$]{\begin{centering}
\includegraphics[scale=0.115]{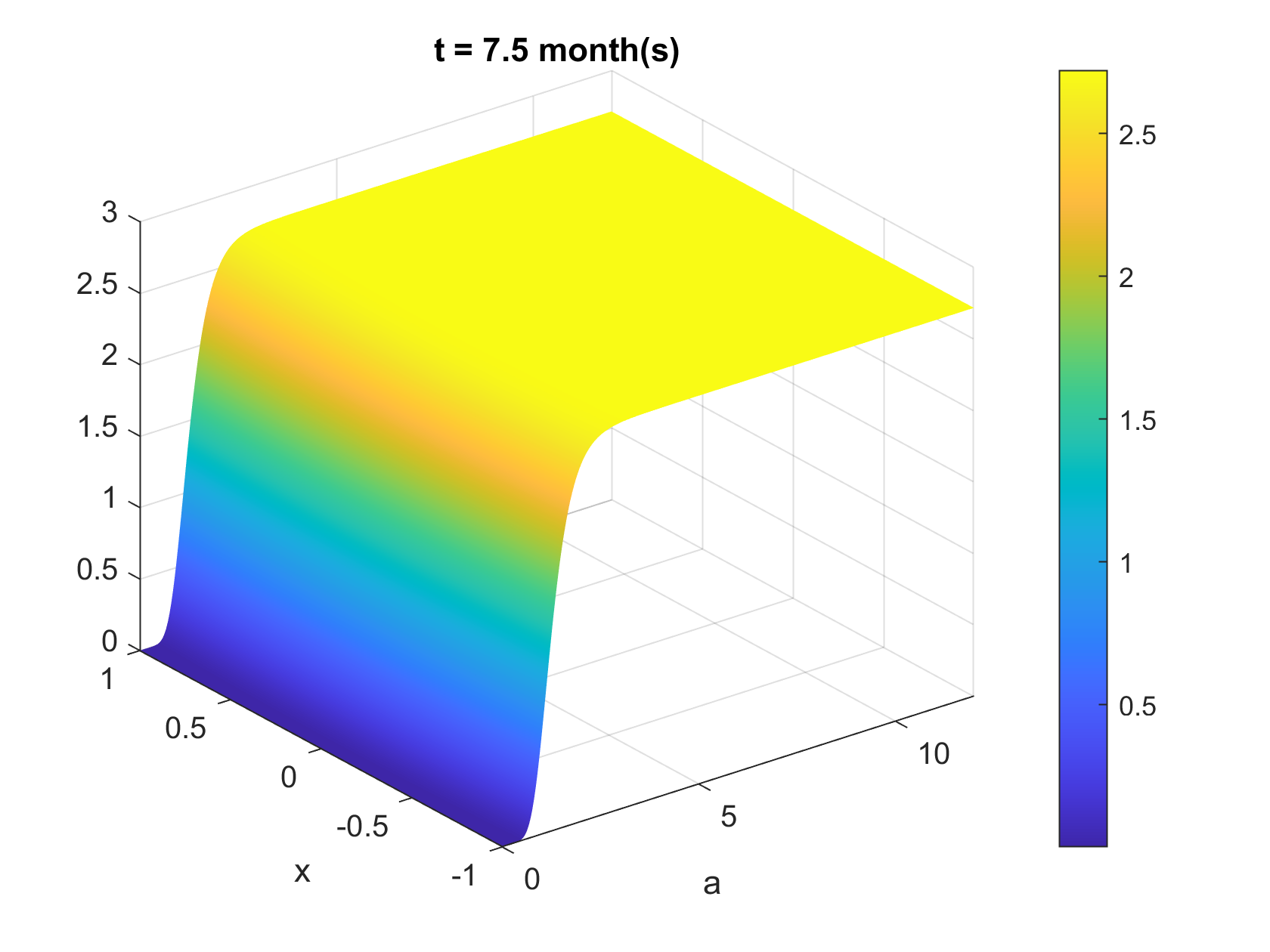}
\par\end{centering}
}
\par\end{centering}
\begin{centering}
\subfloat[$M=200$]{\begin{centering}
\includegraphics[scale=0.115]{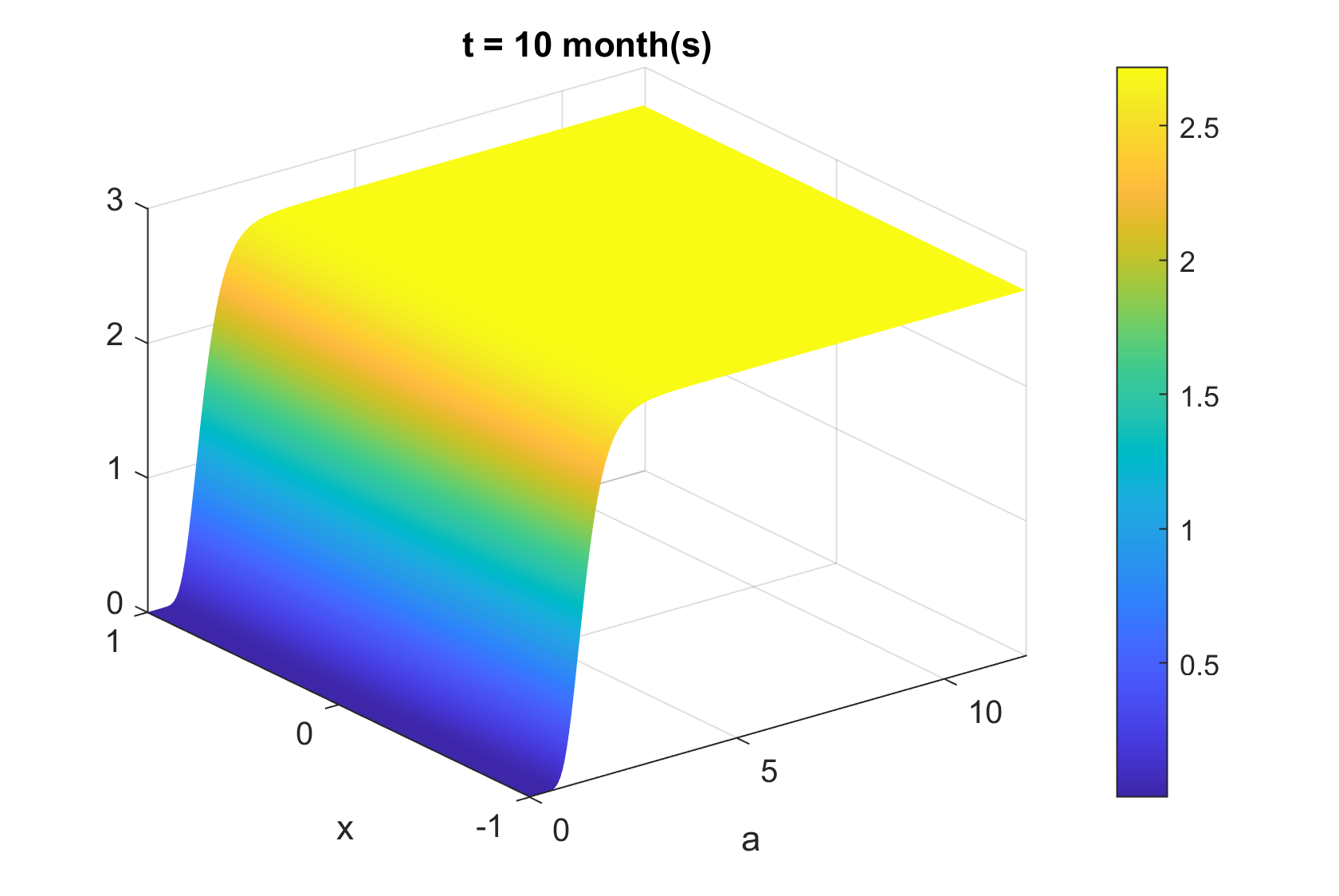}
\par\end{centering}
}\subfloat[$M=400$]{\begin{centering}
\includegraphics[scale=0.115]{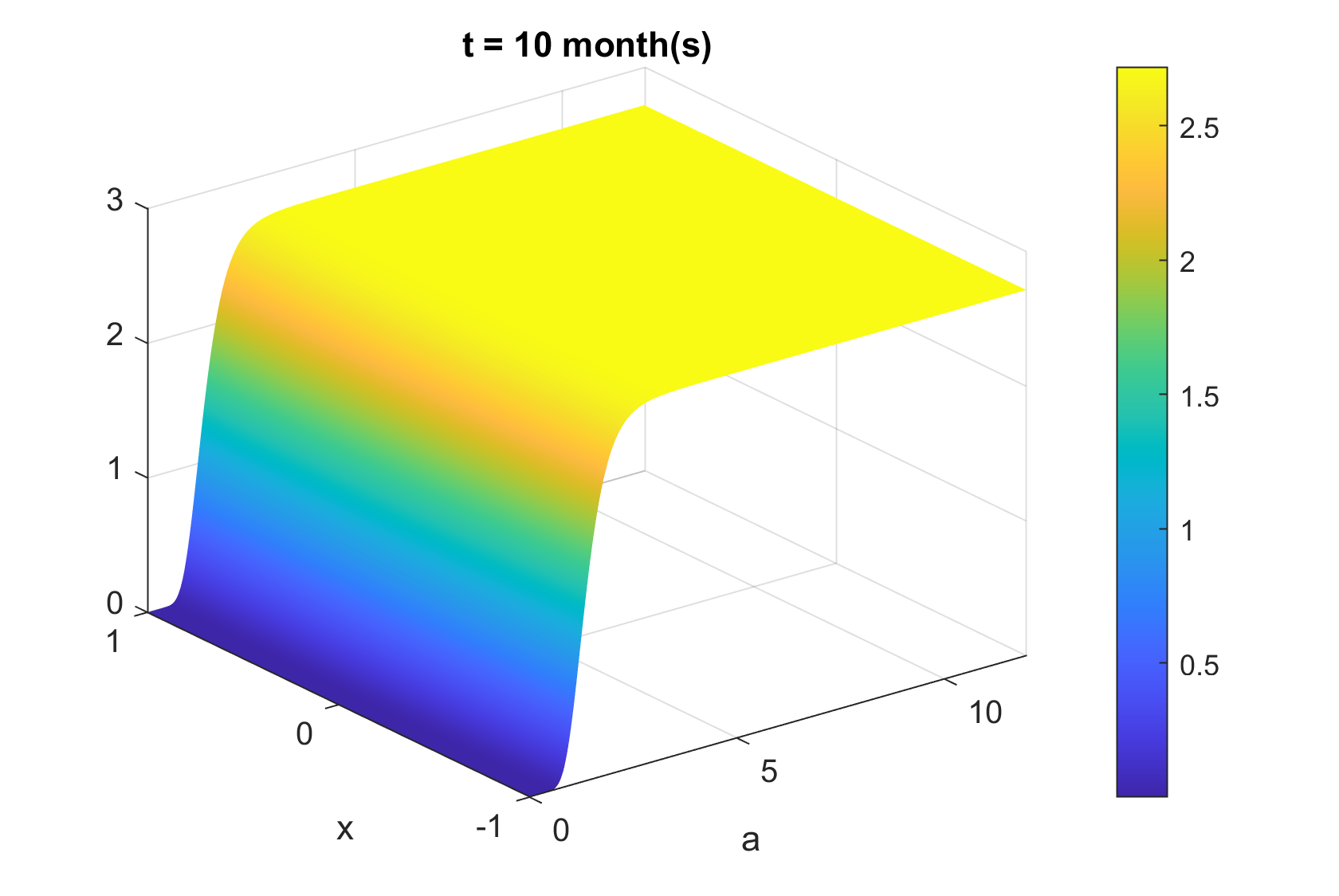}
\par\end{centering}
}\subfloat[$M=800$]{\begin{centering}
\includegraphics[scale=0.115]{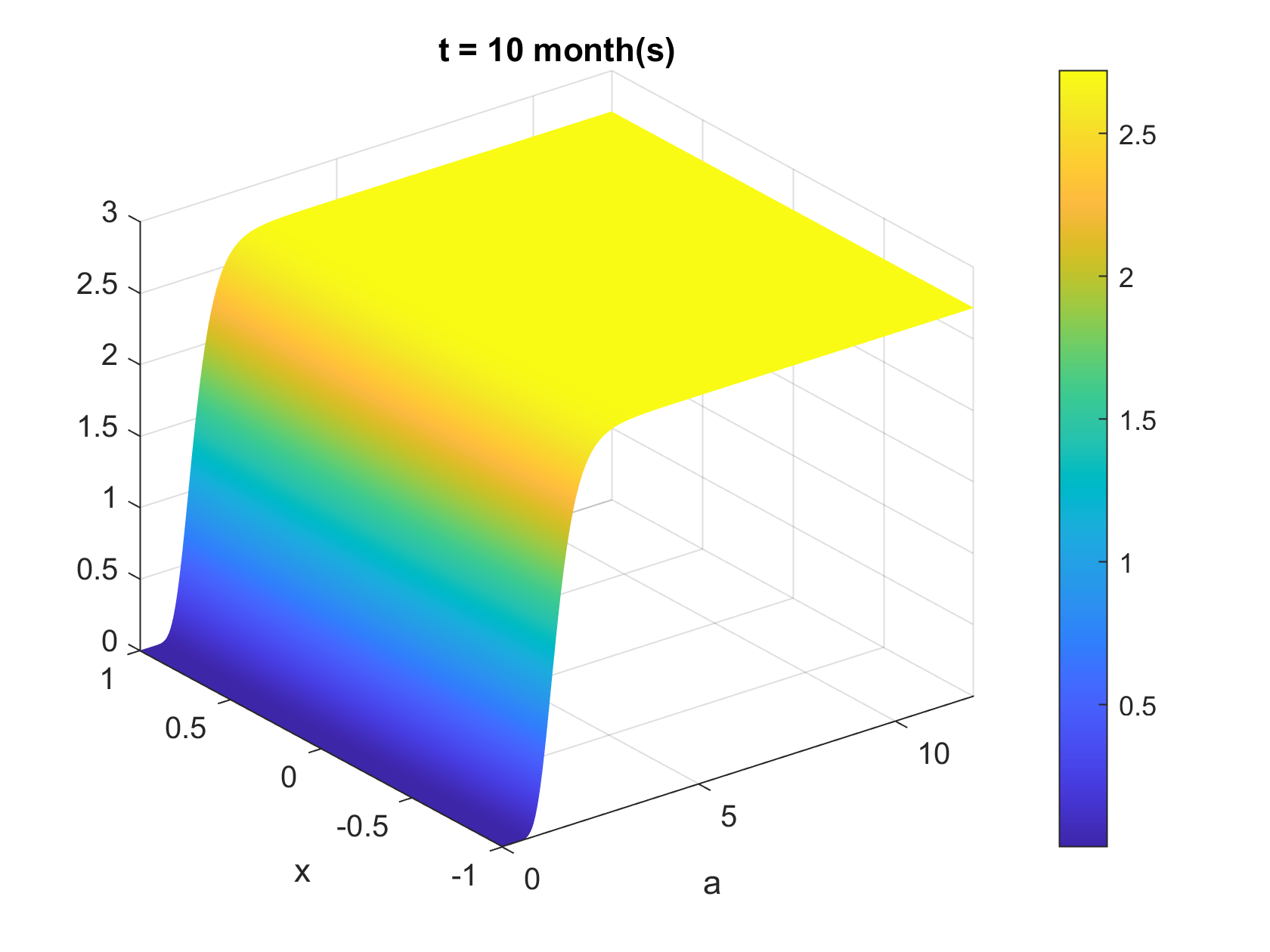}
\par\end{centering}
}
\par\end{centering}
\caption{Tumor cell density in Example 2 at $t=2.5,5.0,7.5,10.0$ for various
values of $M$. Column 1: $M=200,\Delta t=0.05$ (36 hours). Column
2: $M=400,\Delta t=0.025$ (18 hours). Column 3: $M=800,\Delta t=0.0125$
(9 hours).\label{fig:Ex2-2}}
\end{figure}
\par\end{center}

\subsection{Example 3: Hump-shaped initial profiles}

In this last example, we examine the explicit Fourier-Klibanov scheme
for our Gompertz system with the following initial conditions:
\[
u_{0}\left(a,x\right)=\frac{1}{\sqrt{2\pi}\varepsilon}\left[2-\sin\left(\frac{\pi}{4}\left(a-3\right)\right)\right]\exp\left(-\left(x-0.25\right)^{2}\right),
\]
\[
\overline{u}_{0}\left(t,x\right)=\frac{1}{\sqrt{2\pi}\varepsilon}\left[2-\sin\left(\frac{\pi}{4}\left(t-3\right)\right)\right]\exp\left(-\left(x-0.25\right)^{2}\right),
\]
where $\varepsilon=0.5$. Cf. Figure \ref{fig:-(2D)}, these functions
model well the hump-shaped profiles in the plane of $x$ and $a$.
It is also understood that there are two adjacent tumor distributions
presented in the sampled brain tissue. Using the proposed scheme,
we look for the dynamics of these tumor cell distributions with $\rho=0.36$
being the net proliferation rate and 
\[
D\left(t,a\right)=\exp\left(-\left(t-2T\right)^{2}-\left(a-2a_{\dagger}\right)^{2}\right).
\]
By the choice of $D\left(t,a\right)$ above, only the ``elders''
diffuse, but very slowly, at the final time observation.

We present our numerical findings for this example through: Figure
\ref{fig:Ex3-2} illustrates the tumor cell density, while Figure
\ref{fig:total3} depicts the total population. Based on our numerical
observations, we find that the approximation is very stable as $M$
increases. The curves representing the total population in Figure
\ref{fig:total3} almost coincide, indicating a very good accuracy
in the macroscopic sense. However, it is important to note that the
accuracy of the approximation decreases as we move further away from
the initial point, particularly towards the final time observation;
see again Figure \ref{fig:total3}.

The simulation provides some insights into the behavior of tumor cells,
as depicted in Figures \ref{fig:total3} and \ref{fig:Ex3-2}. Starting
with 3 thousand cells/cm, the total population gradually increases
and reaches its peak of approximately 44 thousand cells/cm in a span
of 10 months. The population somehow reaches a steady state between
the third and sixth months. Similar to Example 2, the distribution
of tumor cells has a modest peak, indicating that a too large $M$
is not required for numerical stability at all time points.

Furthermore, Figure \ref{fig:Ex3-2} shows the development and spread
of tumor cells from the original two distributions as both age and
time evolve. At certain time points, new tumor cells emerge and gradually
propagate to merge with the existing ``older'' distribution.
\begin{center}
\begin{figure}
\begin{centering}
\subfloat[$u_{0}$ (3D)]{\begin{centering}
\includegraphics[scale=0.2]{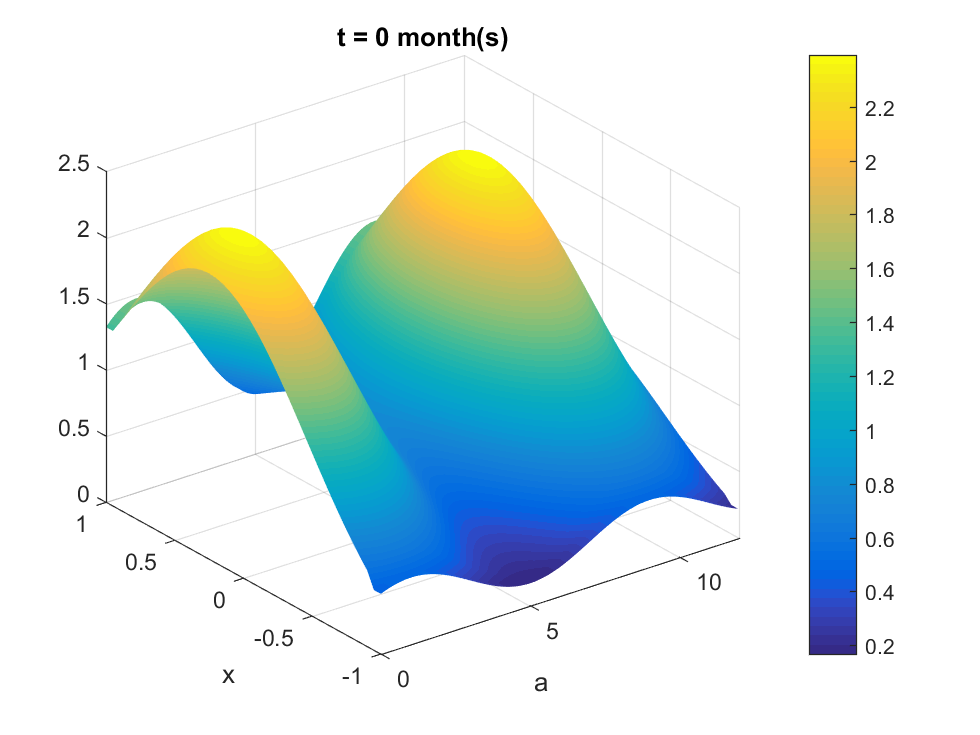}
\par\end{centering}
}\subfloat[$u_{0}$ (2D)\label{fig:-(2D)}]{\begin{centering}
\includegraphics[scale=0.2]{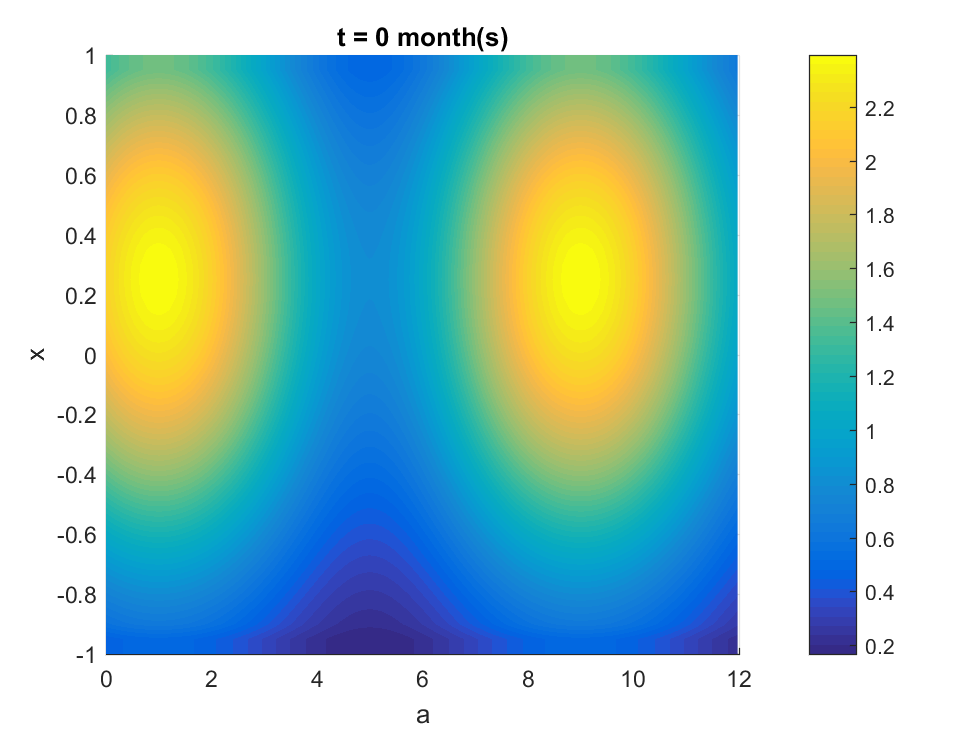}
\par\end{centering}
}\subfloat[$p\left(t\right)$\label{fig:total3}]{\begin{centering}
\includegraphics[scale=0.2]{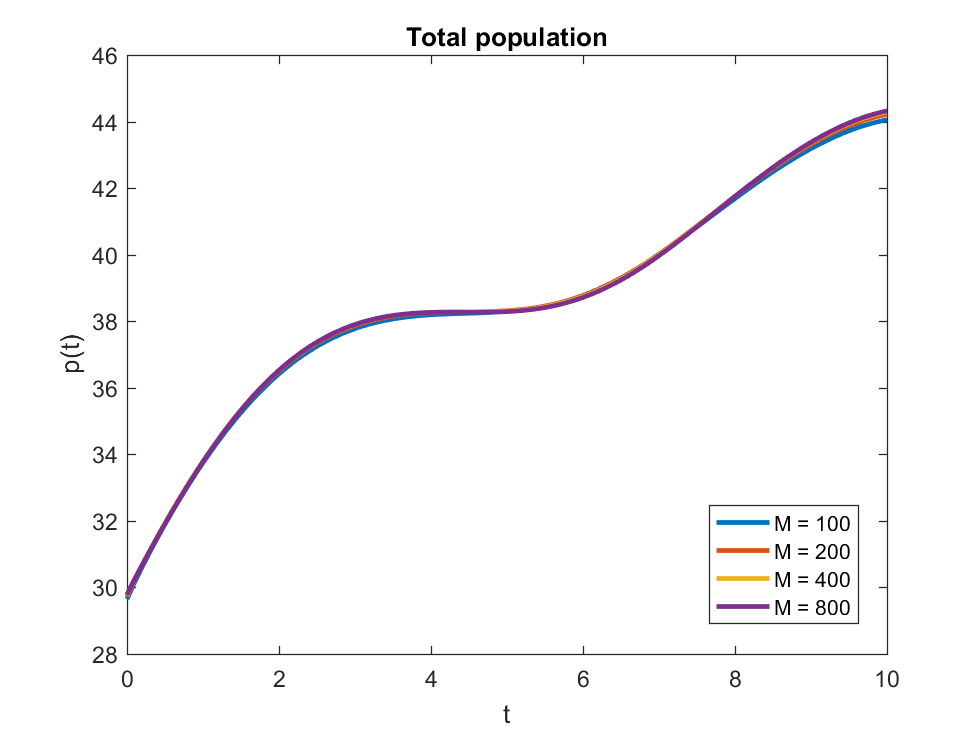}
\par\end{centering}
}
\par\end{centering}
\caption{Left: 3D representation of the initial data $u_{0}$ in Example 3.
Middle: 2D representation of the initial data $u_{0}$ in Example
3. Right: Total population of tumor cells $p\left(t\right)$ for varying
values of $M$.\label{fig:Ex3-1}}
\end{figure}
\par\end{center}

\begin{center}
\begin{figure}
\begin{centering}
\subfloat[$M=200$]{\begin{centering}
\includegraphics[scale=0.115]{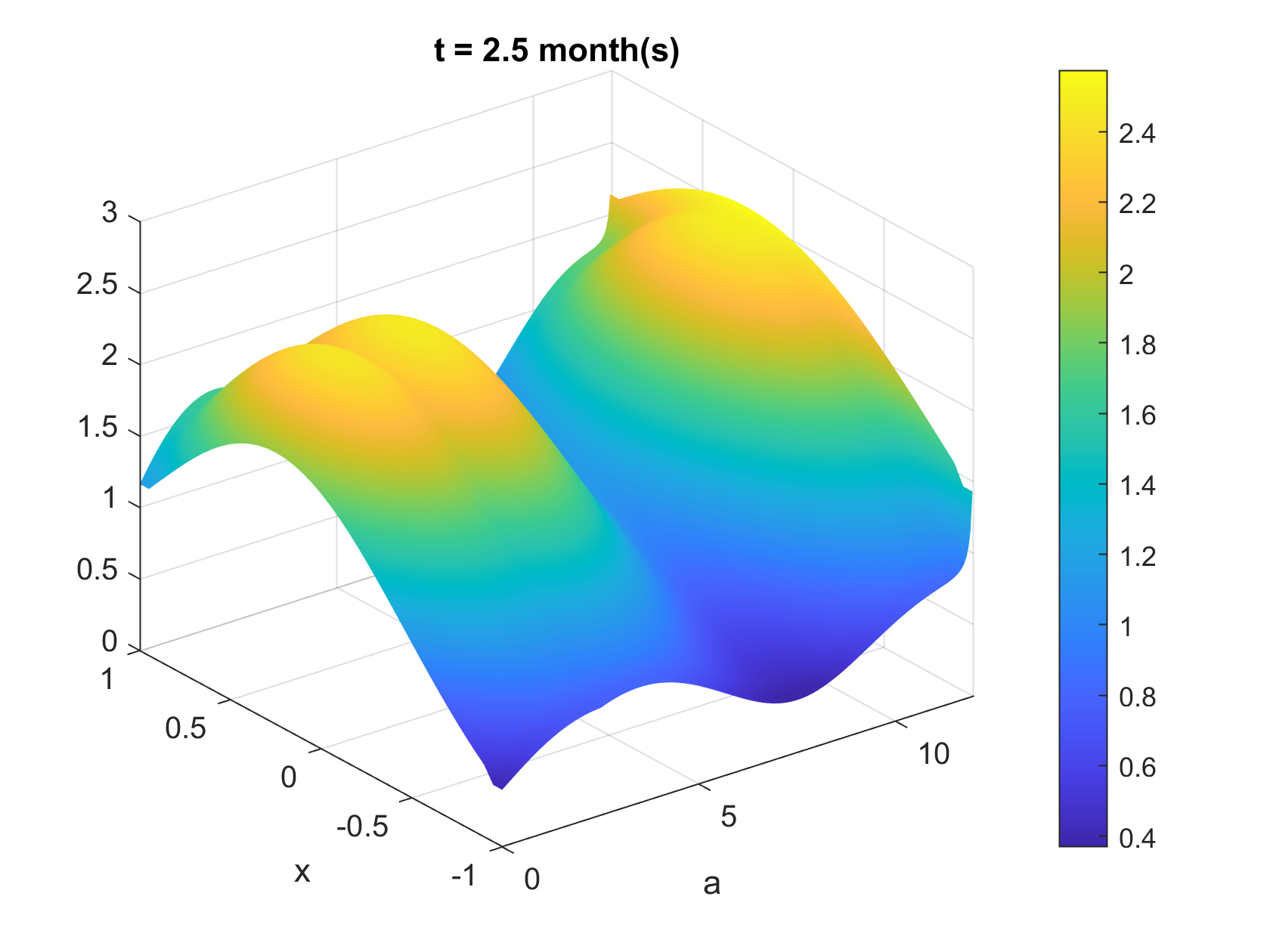}
\par\end{centering}
}\subfloat[$M=400$]{\begin{centering}
\includegraphics[scale=0.115]{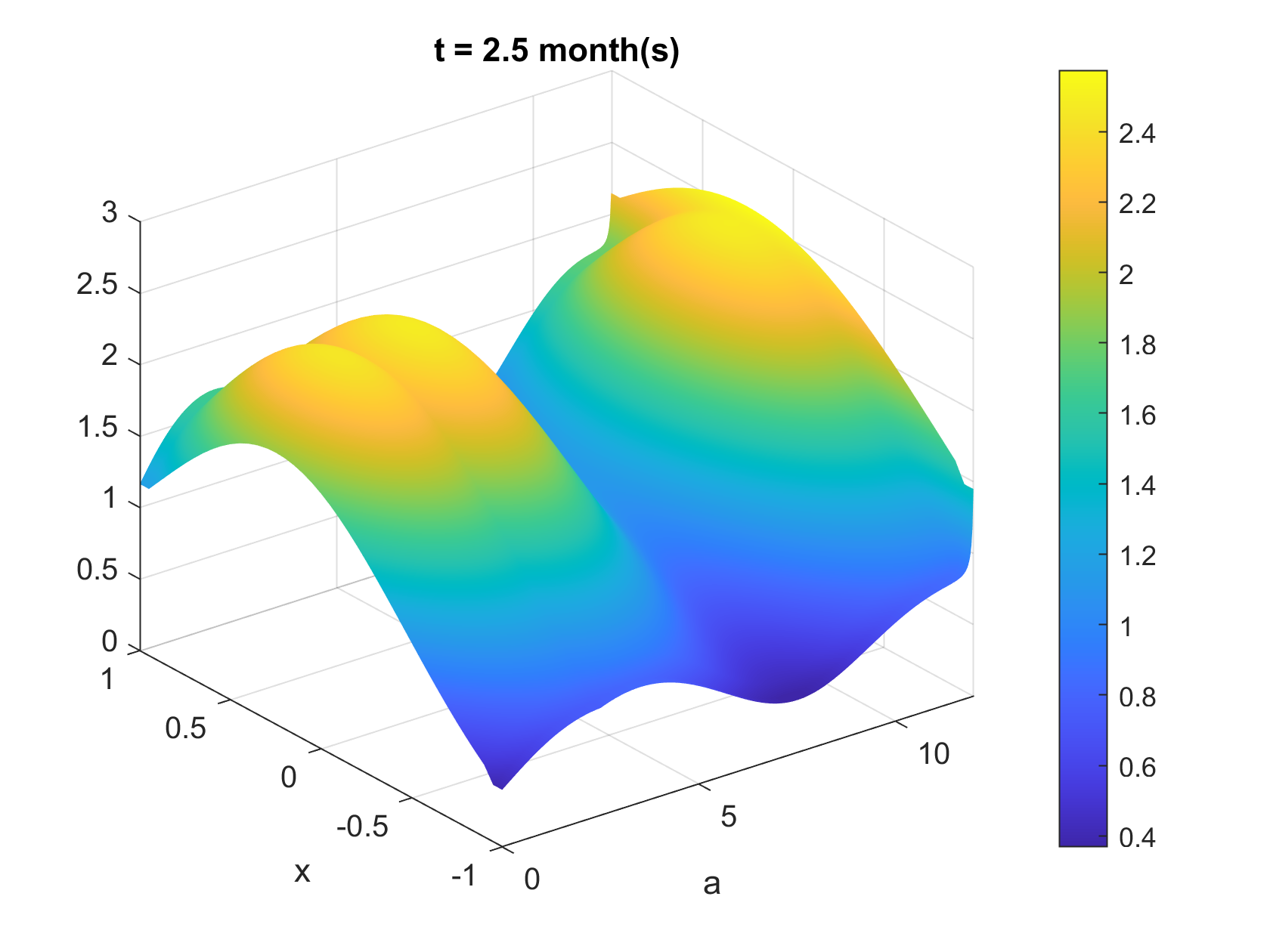}
\par\end{centering}
}\subfloat[$M=800$]{\begin{centering}
\includegraphics[scale=0.115]{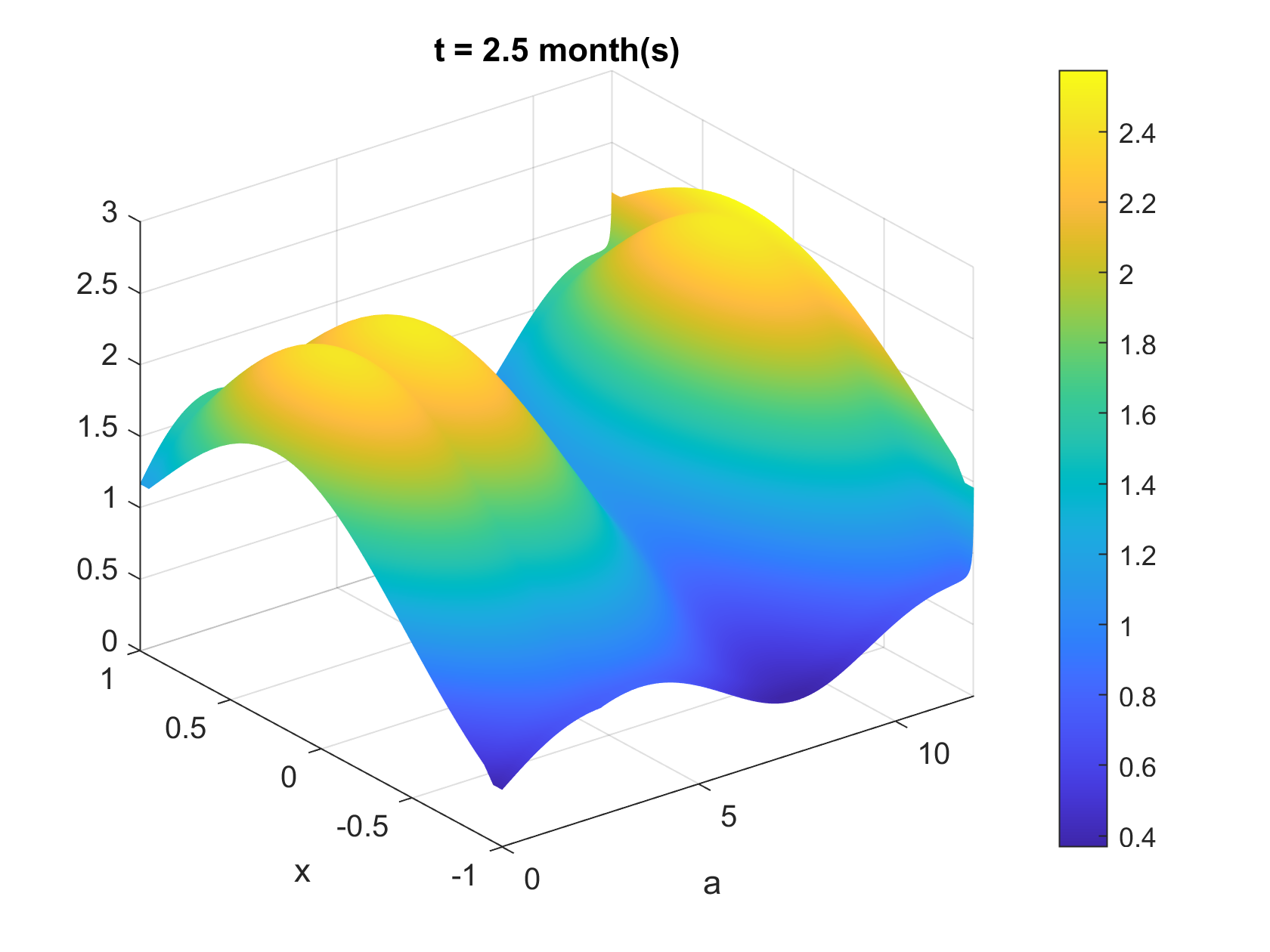}
\par\end{centering}
}
\par\end{centering}
\begin{centering}
\subfloat[$M=200$]{\begin{centering}
\includegraphics[scale=0.115]{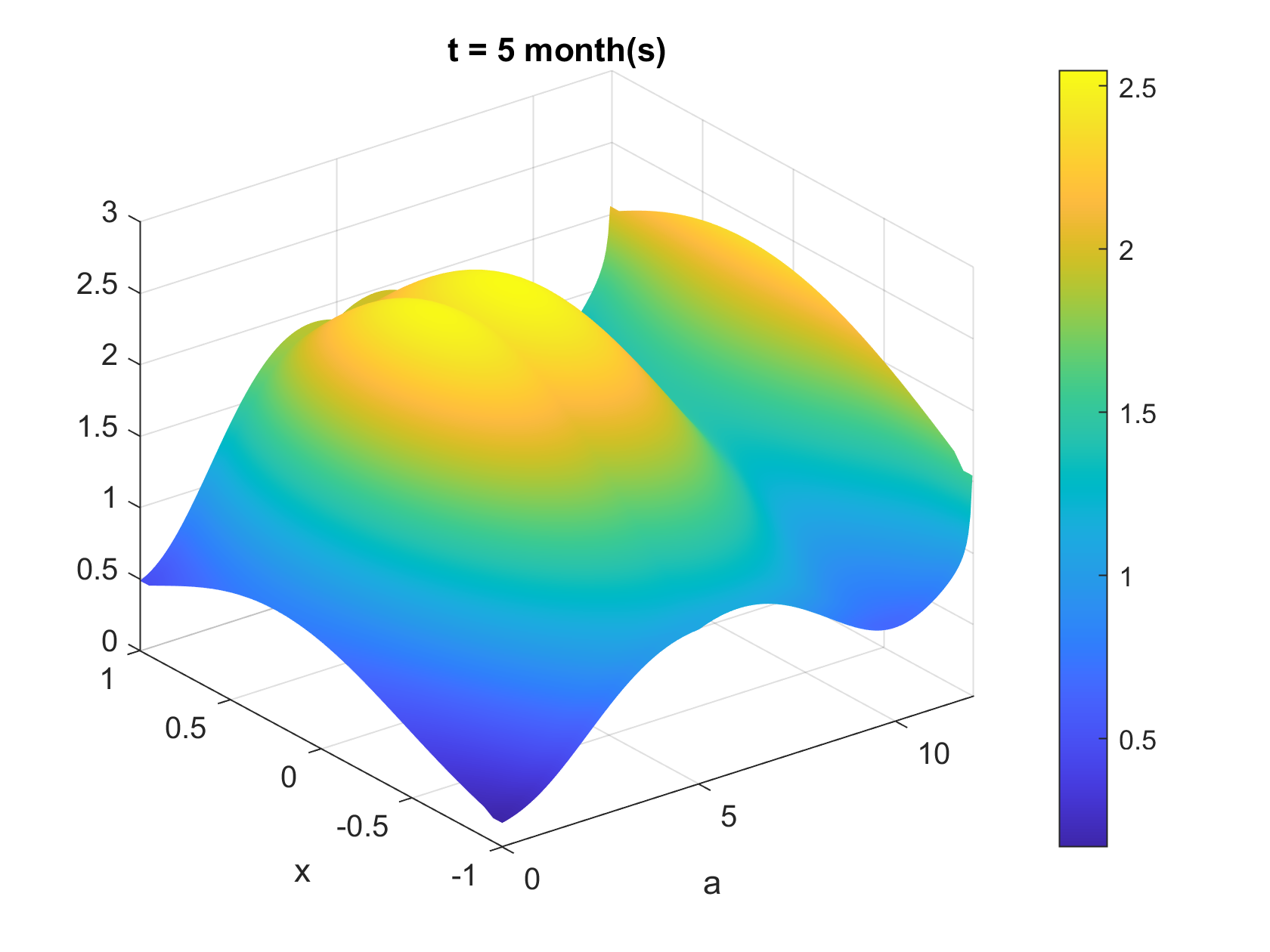}
\par\end{centering}
}\subfloat[$M=400$]{\begin{centering}
\includegraphics[scale=0.115]{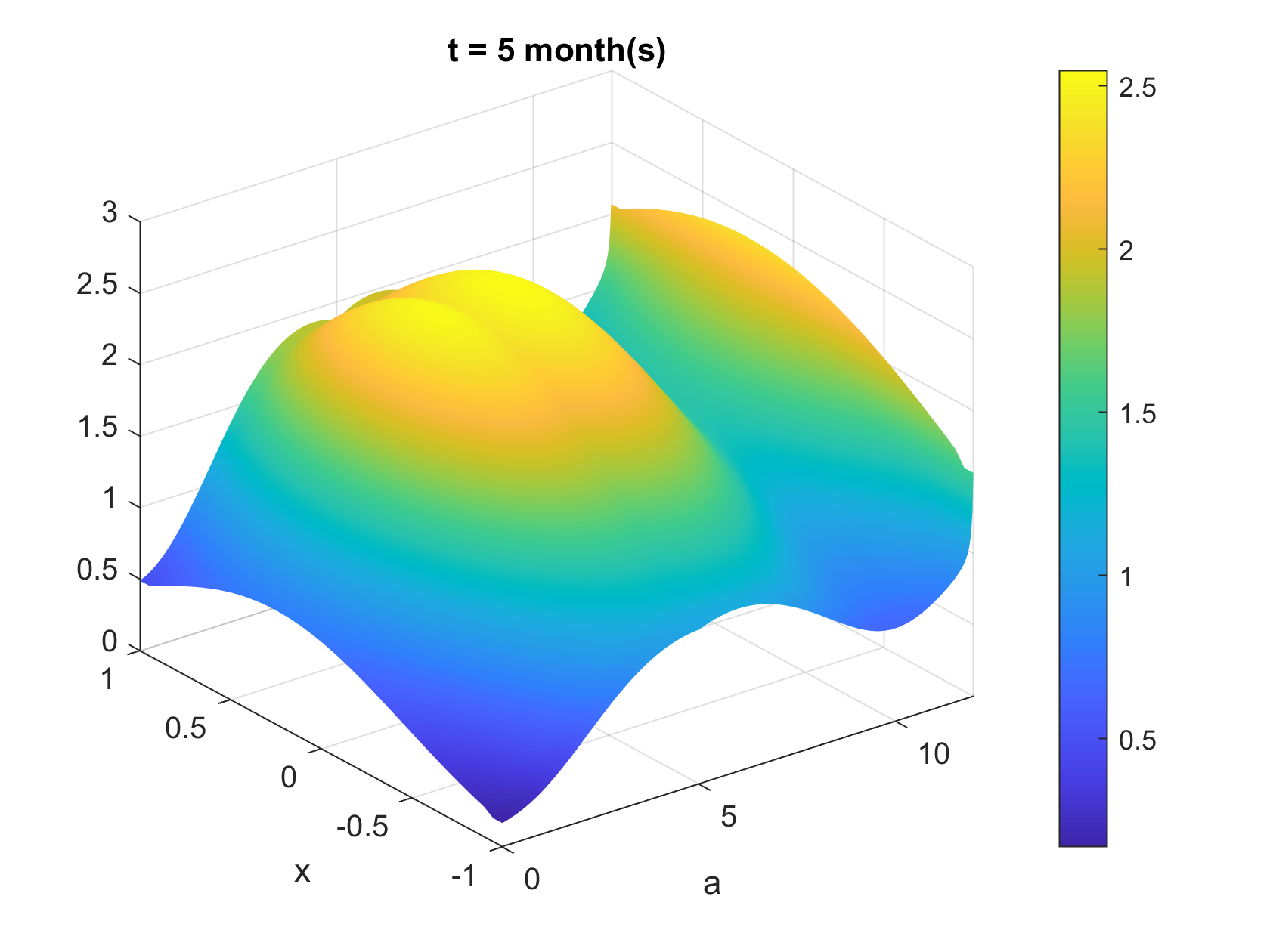}
\par\end{centering}
}\subfloat[$M=800$]{\begin{centering}
\includegraphics[scale=0.115]{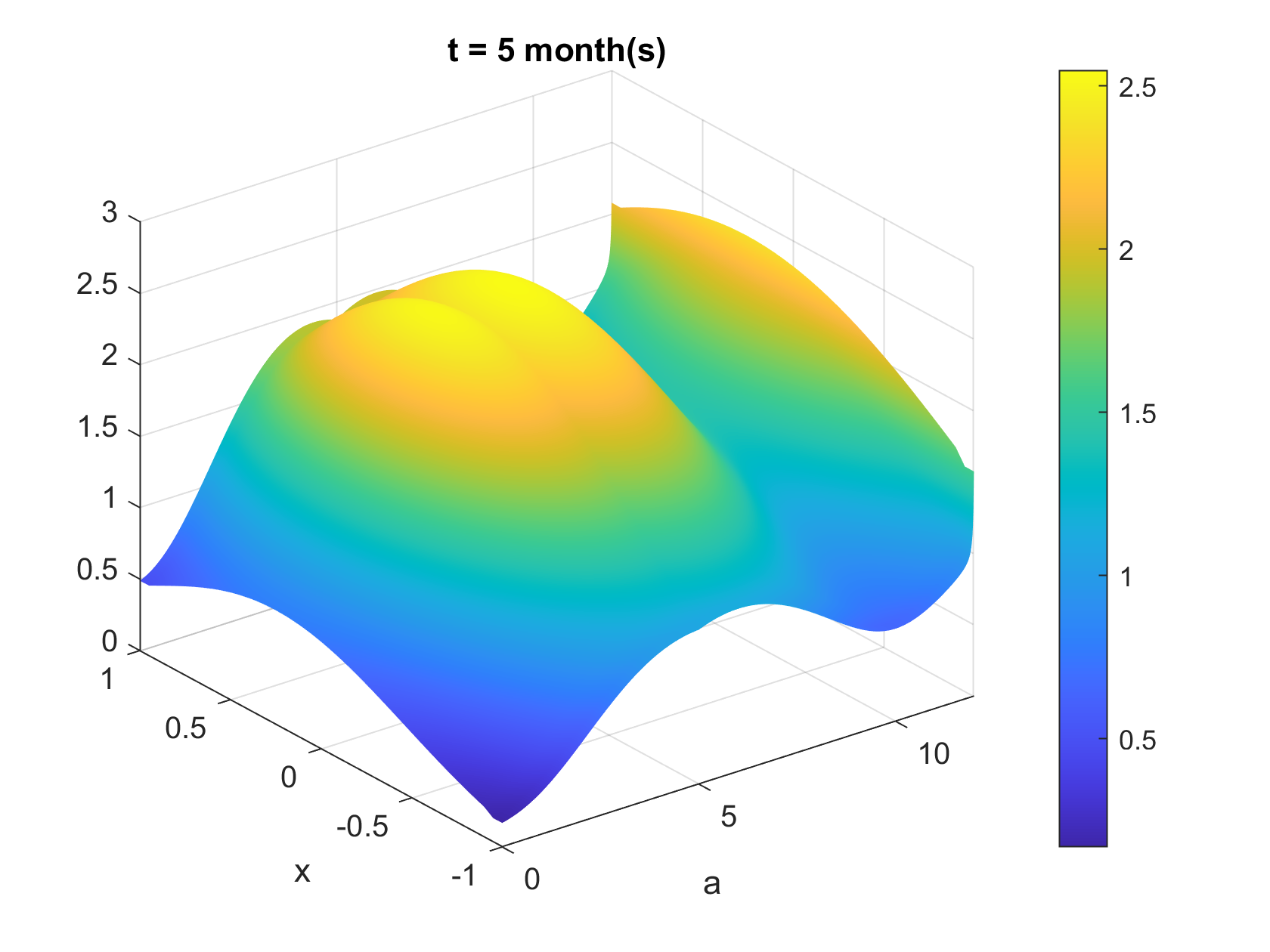}
\par\end{centering}
}
\par\end{centering}
\begin{centering}
\subfloat[$M=200$]{\begin{centering}
\includegraphics[scale=0.115]{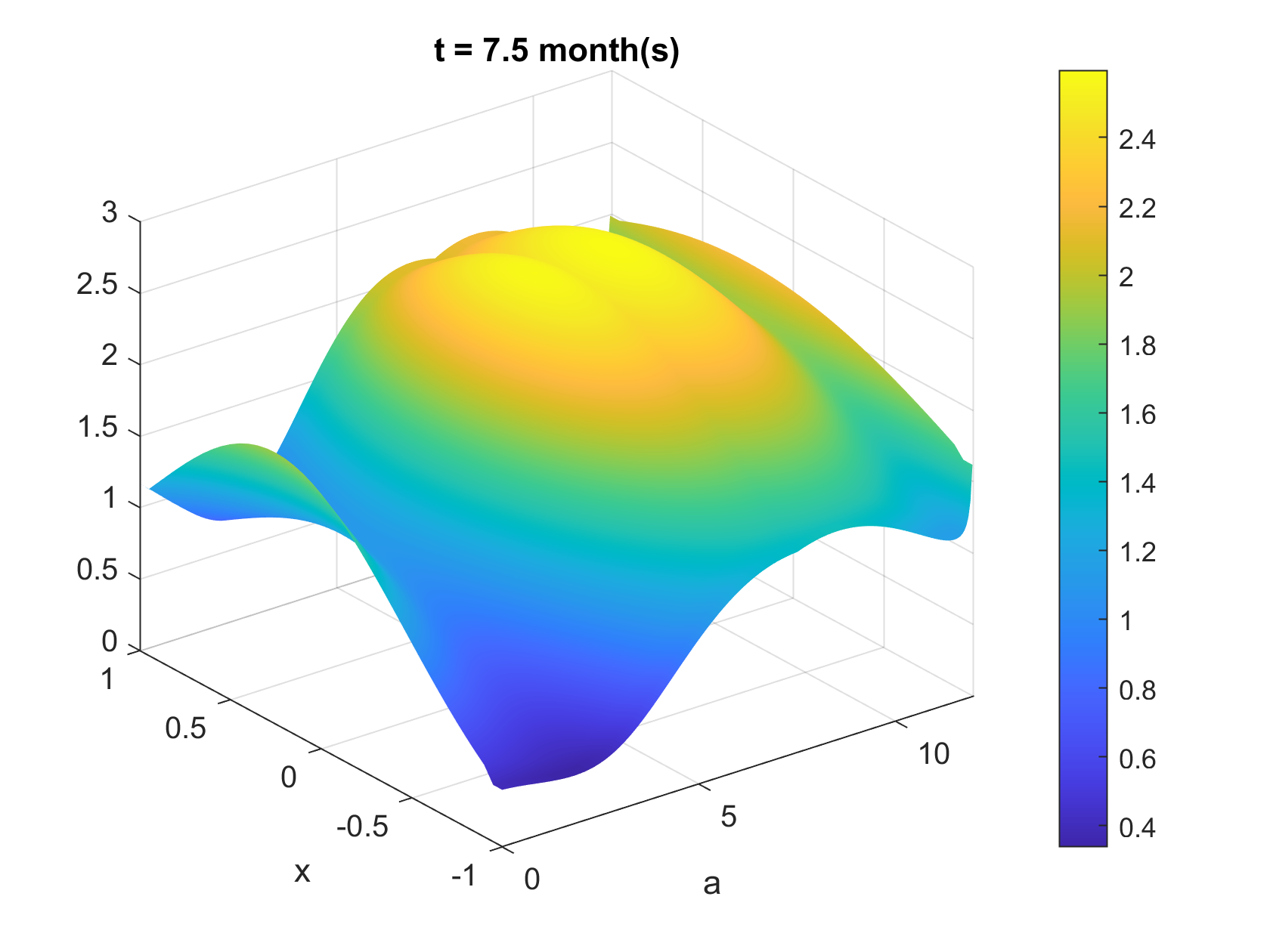}
\par\end{centering}
}\subfloat[$M=400$]{\begin{centering}
\includegraphics[scale=0.115]{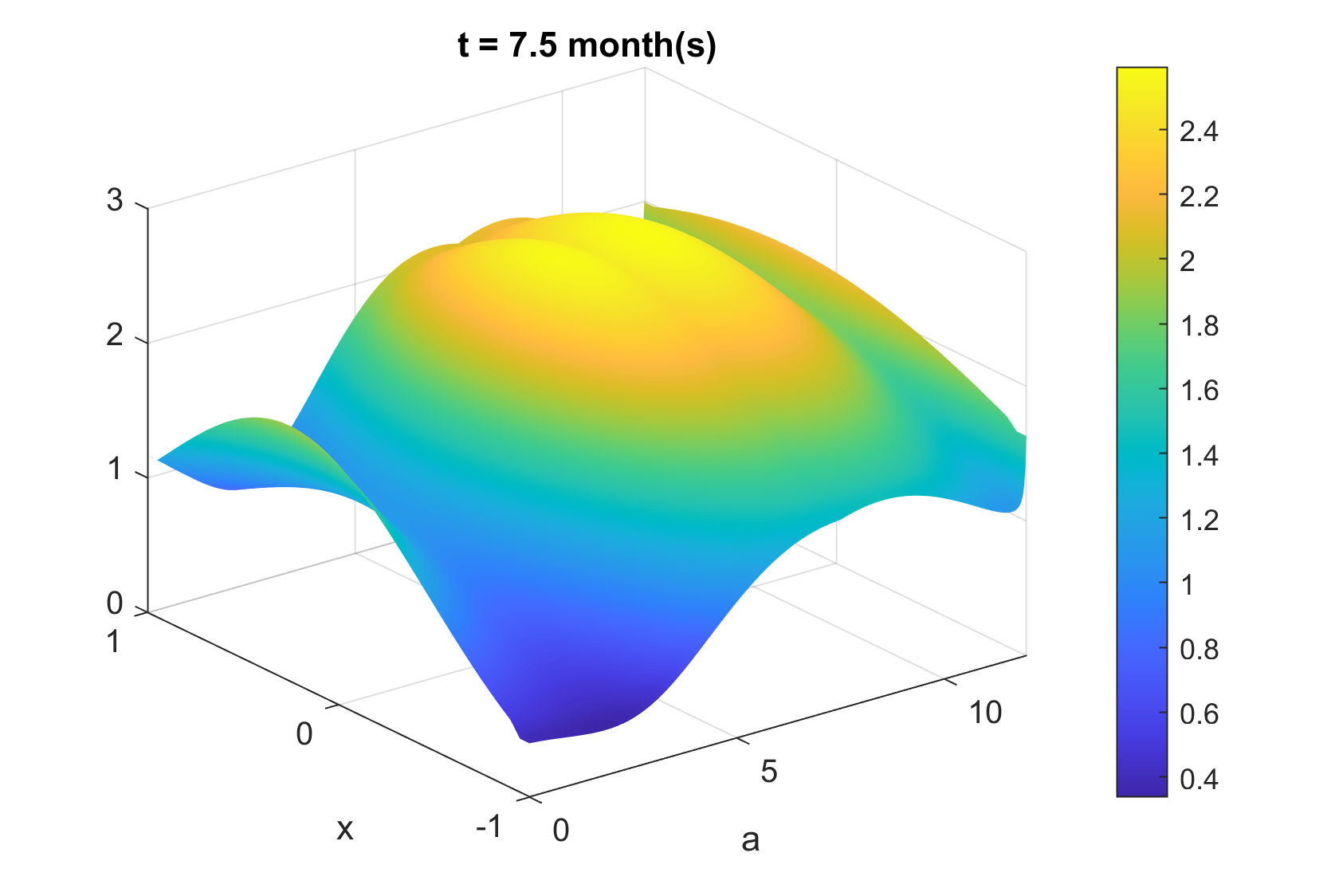}
\par\end{centering}
}\subfloat[$M=800$]{\begin{centering}
\includegraphics[scale=0.115]{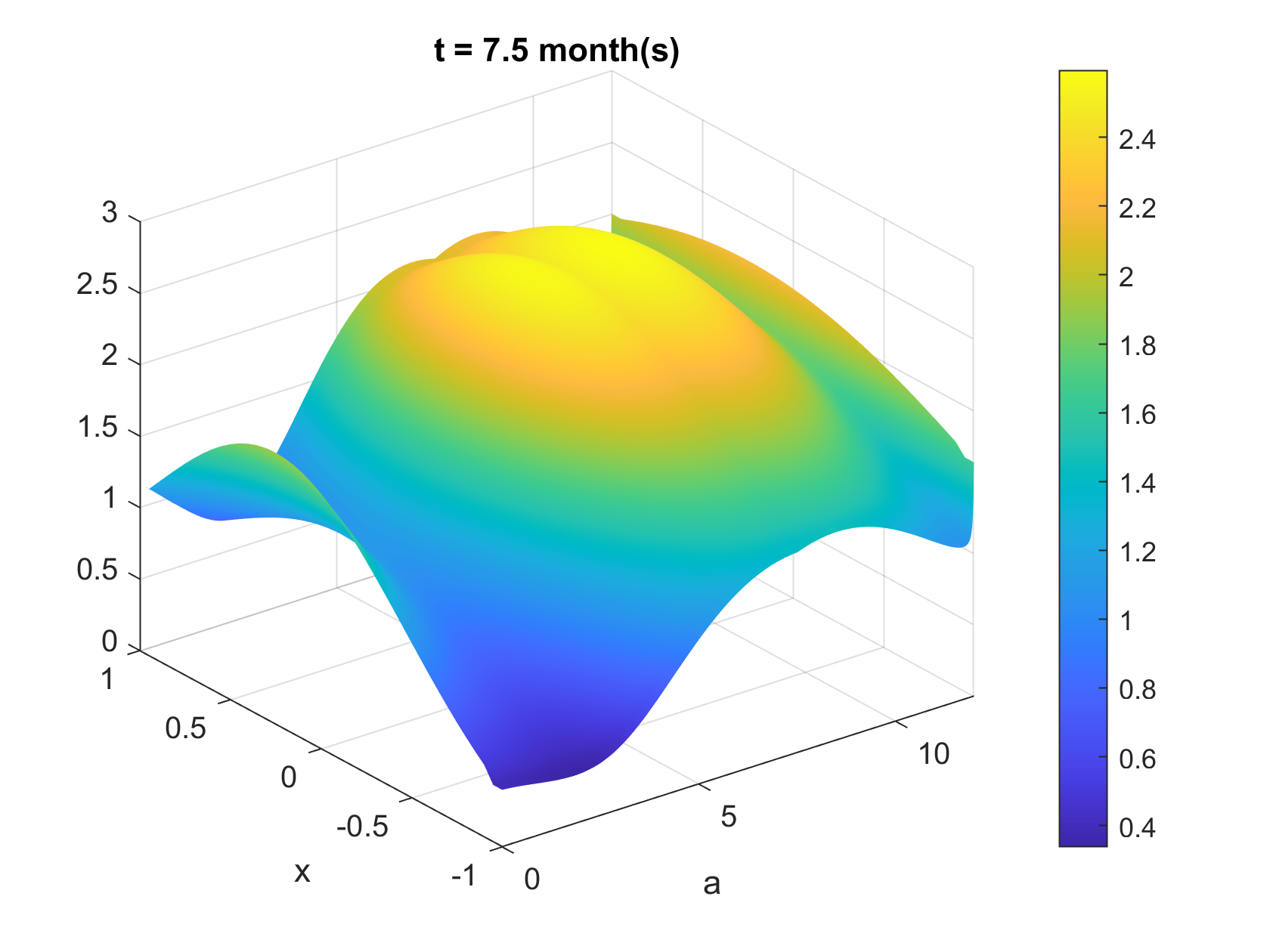}
\par\end{centering}
}
\par\end{centering}
\begin{centering}
\subfloat[$M=200$]{\begin{centering}
\includegraphics[scale=0.115]{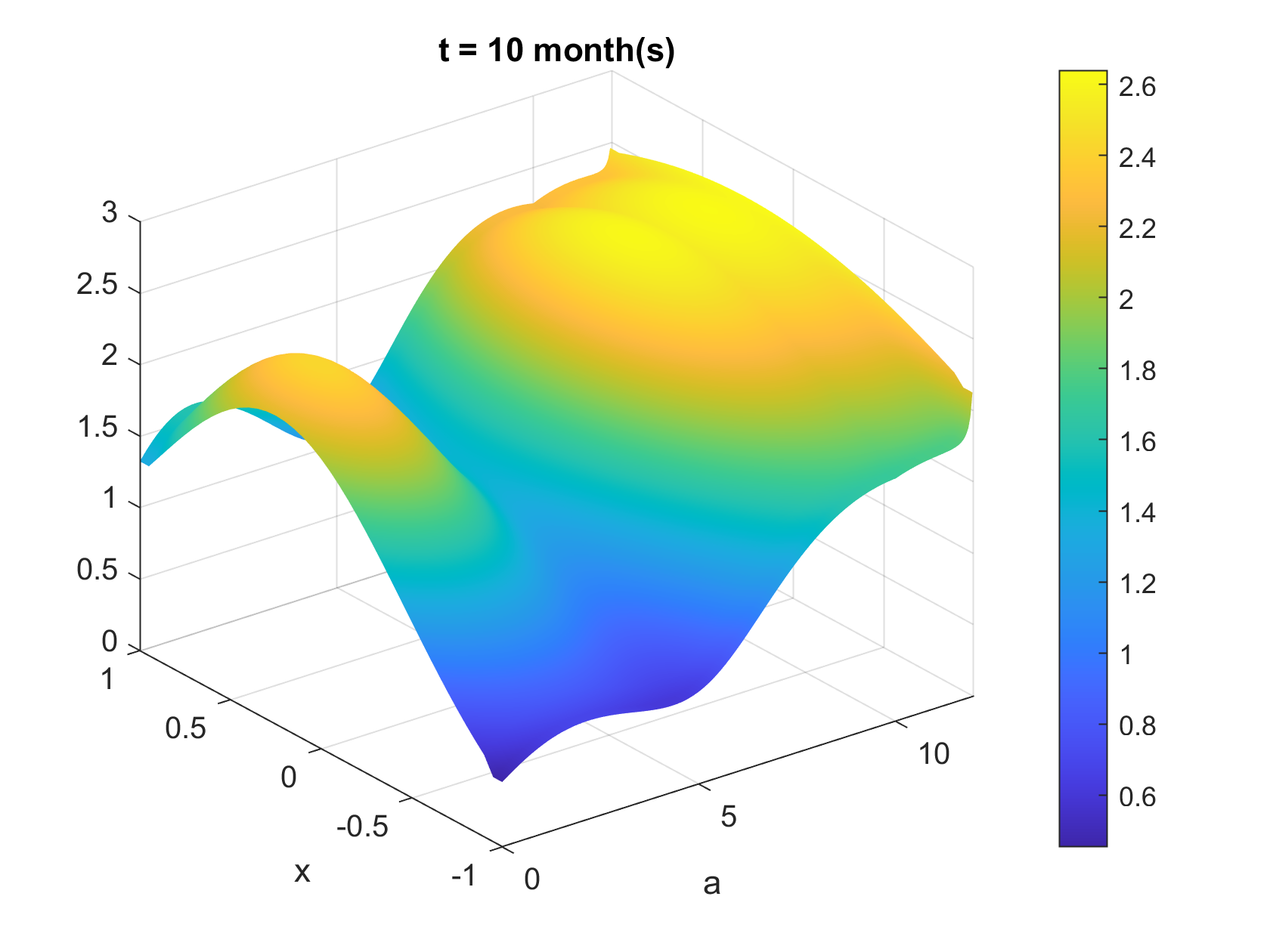}
\par\end{centering}
}\subfloat[$M=400$]{\begin{centering}
\includegraphics[scale=0.115]{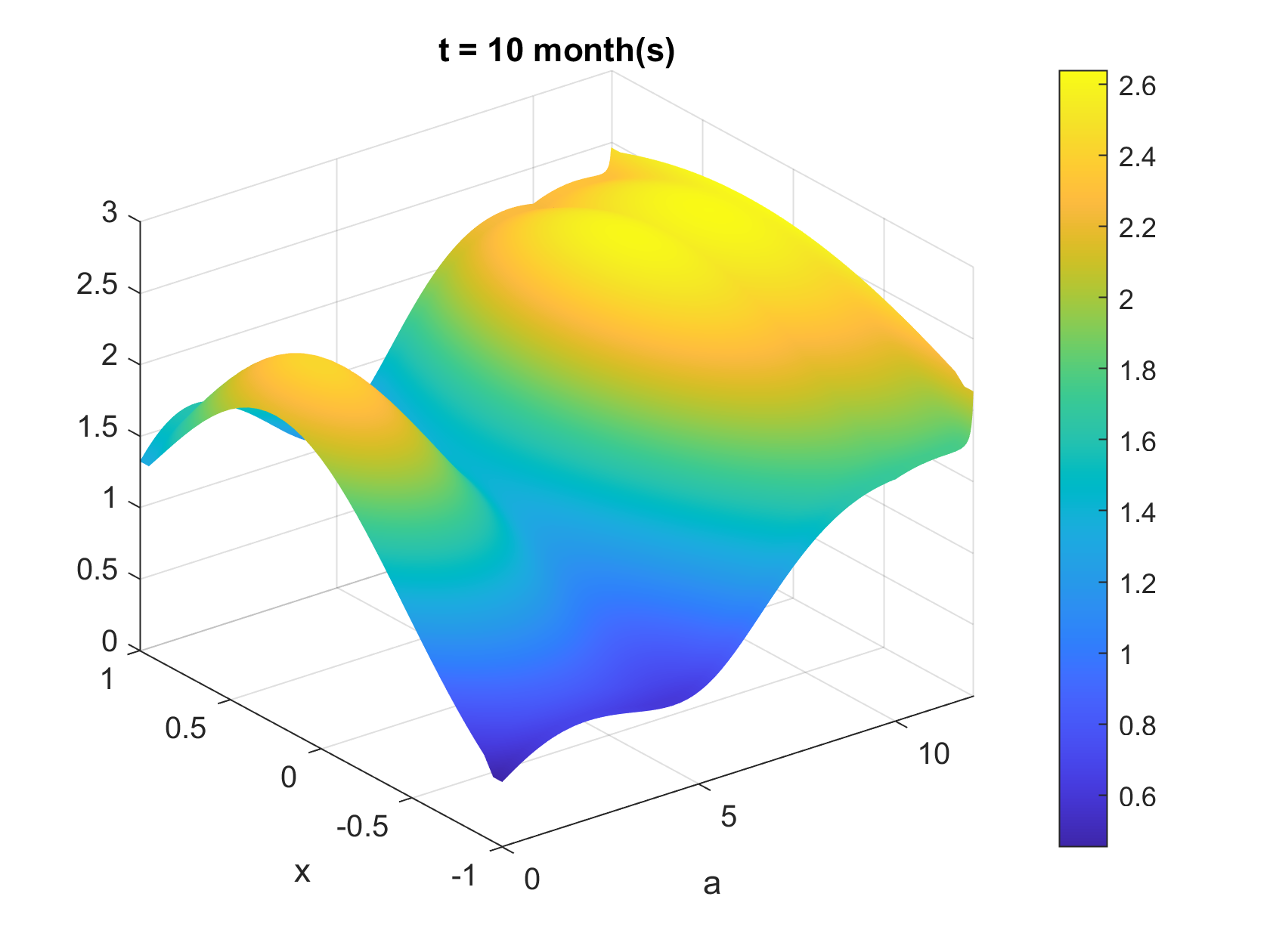}
\par\end{centering}
}\subfloat[$M=800$]{\begin{centering}
\includegraphics[scale=0.115]{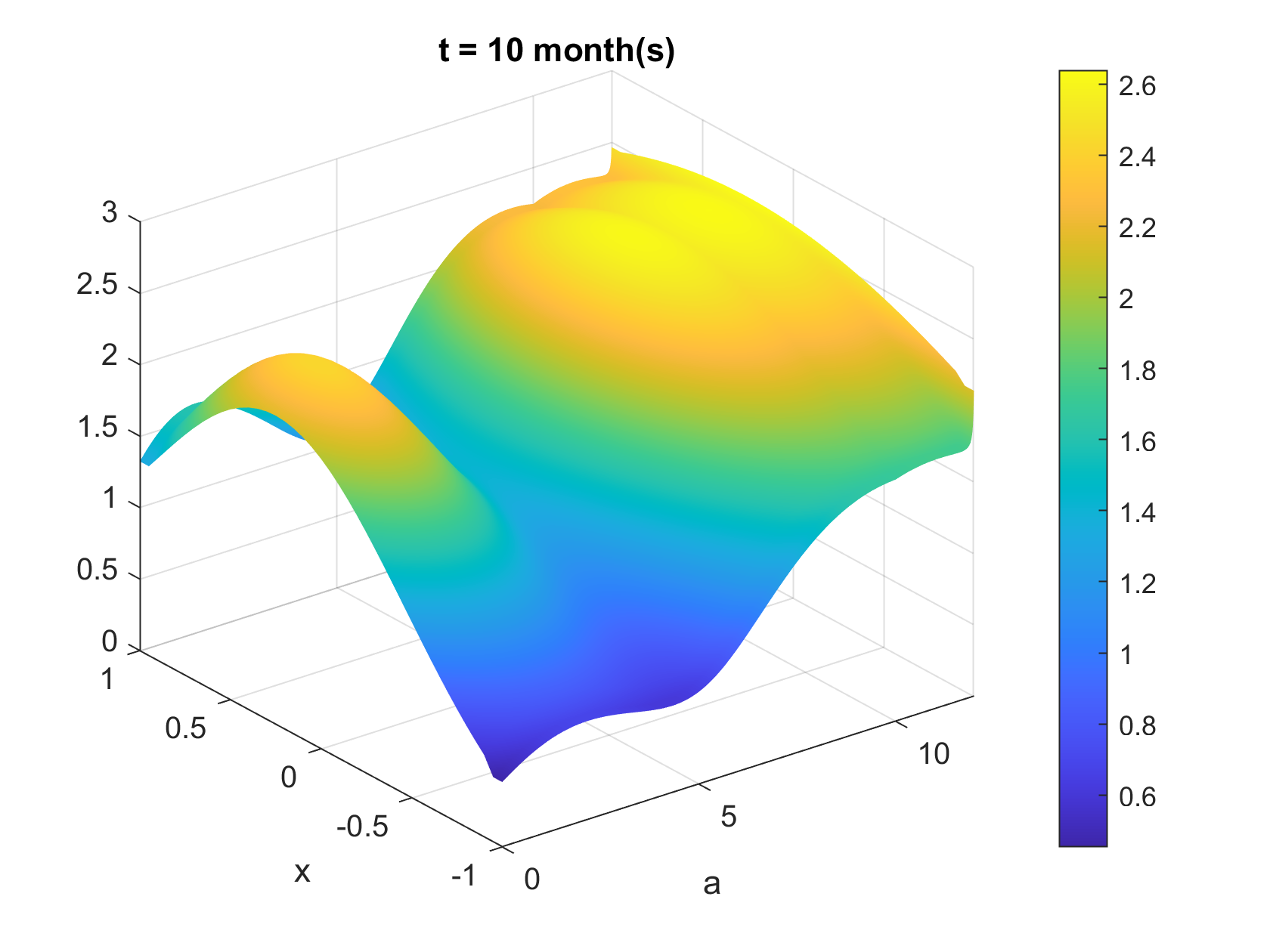}
\par\end{centering}
}
\par\end{centering}
\caption{Tumor cell density in Example 3 at $t=2.5,5.0,7.5,10.0$ for various
values of $M$. Column 1: $M=200,\Delta t=0.05$ (36 hours). Column
2: $M=400,\Delta t=0.025$ (18 hours). Column 3: $M=800,\Delta t=0.0125$
(9 hours).\label{fig:Ex3-2}}
\end{figure}
\par\end{center}

\section{Conclusions\label{sec:45}}

In this work, we have presented a new numerical approach to solve
the age-structured population diffusion problem of Gompertz type.
Our approach relies on a combination of the recently developed Fourier-Klibanov
method and the explicit finite difference method of characteristics.
The notion is that exploiting suitable transformations, the Gompertz
model of interest turns to a third-order nonlinear PDE. Then, the
Fourier-Klibanov method is applied to derive a coupled transport-like
PDE system. This system is explicitly approximated by the finite difference
operators of time and age.

In this work, we have focused on the numerics rather than the theory.
It remains to show the rate of convergence of the explicit scheme
under particular smoothness conditions of the involved parameters
and the true solution. Besides, it is still open in this age-dependent
Gompertz mode that if we have the non-negativity of the initial data,
the whole solution will follow.

As readily expected, the explicit finite difference method is conditionally
stable. Thus, the implicit scheme should be investigated in the upcoming
research. We also want to extend the applicability of the method to
nonlinear heterogeneous problems in multiple dimensions.

\section*{Acknowledgments}

This research is funded by University of Science, VNU-HCM under grant
number T2022-47. N. T. Y. N. would love to thank Prof. Dr. Nam Mai-Duy
and Prof. Dr. Thanh Tran-Cong from University of Southern Queensland
(Australia) for their support of her PhD period. V. A. K. would like
to thank Drs. Lorena Bociu, Ryan Murray, Tien-Khai Nguyen from North
Carolina State University (USA) for their support of his early research
career.

\bibliographystyle{plain}
\bibliography{mybib}

\end{document}